\documentclass[12pt]{amsart}

\usepackage{amssymb,latexsym}
\usepackage{bm}
\usepackage{graphicx}
\usepackage{psfrag}
\usepackage{wrapfig}
\usepackage{color}
\usepackage{enumitem}
\setlist[enumerate]{parsep=0pt plus 4pt,topsep=0pt plus 4pt}
\usepackage{url}
\usepackage{hyperref}
\usepackage{citeref}
\definecolor{darkblue}{RGB}{0,0,160}
\hypersetup{colorlinks,citecolor=magenta,filecolor=black,
            linkcolor=cyan,urlcolor=darkblue}
\definecolor{lightred}{rgb}{1,.3,.3}
\definecolor{darkpurple}{rgb}{.5,.2,.9}

\allowdisplaybreaks

\newcommand\excise[1]{}

\newtheorem{thm}{Theorem}[section]
\newtheorem{lemma}[thm]{Lemma}

\newtheorem{cor}[thm]{Corollary}
\newtheorem{prop}[thm]{Proposition}

\newtheorem{theorem}{Theorem}

\theoremstyle{definition}
\newtheorem{example}[thm]{Example}
\newtheorem{remark}[thm]{Remark}
\newtheorem{defn}[thm]{Definition}
\newtheorem{conv}[thm]{Convention}

\numberwithin{equation}{section}

\renewcommand\labelenumi{\theenumi.}

\newcounter{separated}

\newcommand{\Ring}[1]{\ensuremath{\mathbb{#1}}}

\renewcommand\>{\rangle}
\newcommand\<{\langle}
\newcommand\0{\mathbf{0}}
\newcommand\cI{\mathcal{I}}
\newcommand\cJ{\mathcal{J}}
\newcommand\cL{\mathcal{L}}
\newcommand\MM{\mathbb{M}}
\newcommand\NN{\Ring{N}}

\newcommand\QQ{\Ring{Q}}
\newcommand\RR{\Ring{R}}

\newcommand\XX{{\mathfrak X}}
\newcommand\ZZ{\Ring{Z}}
\newcommand\bbI{\mathbb{I}}

\newcommand\ee{{\mathbf e}}
\newcommand\ff{{\mathbf f}}
\newcommand\ii{{\mathbf i}}
\newcommand\jj{{\mathbf j}}
\newcommand\kk{\Bbbk}
\newcommand\mm{{\mathfrak m}}
\newcommand\nn{{\mathbf n}}
\newcommand\pp{{\mathfrak p}}
\newcommand\ppp{{\mathbf p}}
\newcommand\qq{{\mathbf q}}
\newcommand\rr{{\mathbf r}}
\newcommand\uu{{\mathbf u}}
\newcommand\vv{{\mathbf v}}
\newcommand\ww{{\mathbf w}}
\newcommand\xx{{\mathbf x}}
\newcommand\zz{{\mathbf z}}
\newcommand\hhh{{\mathbf h}}
\newcommand\cM{\mathcal{M}}
\newcommand\cN{{\mathcal N}}
\newcommand\cP{P}
\newcommand\oI{{\hspace{.25ex}\overline{\hspace{-.25ex}I}}}
\newcommand\vC{\check{\mathcal C}}

\newcommand\mmm{\mathbf{m}}

\renewcommand\aa{{\mathbf a}}
\renewcommand\ggg{{\mathbf g}}
\renewcommand\phi{\varphi}

\newcommand\sat{\mathrm{sat}}
\newcommand\gs{\geqslant}

\newcommand\ls{\leqslant}

\newcommand\too{\longrightarrow}

\newcommand\from{\leftarrow}
\newcommand\into{\hookrightarrow}

\newcommand\spot{{\hbox{\raisebox{1pt}{\tiny$\scriptscriptstyle\bullet$}}}}

\newcommand\minus{\smallsetminus}
\newcommand\simto{\mathrel{\!\ooalign{$\fillrightmap$\cr\raisebox{.75ex}{$\,\sim\ \hspace{.2ex}$}}}}

\newcommand\bigcupdot{\makebox[0pt][l]{$\hspace{1.05ex}\cdot$}\textstyle\bigcup}
\newcommand\Bigcupdot{\mathop{\rlap{$\hspace{1ex}\cdot$}{\bigcup}}}

\newcommand\fillrightmap{\mathord- \mkern-6mu
	\cleaders\hbox{$\mkern-2mu \mathord- \mkern-2mu$}\hfill
	\mkern-6mu \mathord\rightarrow}

\renewcommand\iff{\Leftrightarrow}
\renewcommand\epsilon{\varepsilon}
\renewcommand\implies{\Rightarrow}

\newcommand\ol[1]{{\overline{#1}}}

\newcommand{\aoverb}[2]{{\genfrac{}{}{0pt}{1}{#1}{#2}}}
\def\twoline#1#2{\aoverb{\scriptstyle {#1}}{\scriptstyle {#2}}}

\DeclareMathOperator\Ext{Ext} 
\DeclareMathOperator\Hom{Hom} 
\DeclareMathOperator\Ass{Ass} 
\DeclareMathOperator\Min{Min} 
\DeclareMathOperator\reg{reg} 
\DeclareMathOperator\soc{soc} 
\DeclareMathOperator\Tor{Tor} 
\DeclareMathOperator\Top{top} 
\DeclareMathOperator\NP{NP} 
\DeclareMathOperator\adeg{adeg} 
\DeclareMathOperator\eext{
	\hspace{.6pt}{\underline{\hspace{-.6pt}{\rm Ext}\hspace{-.6pt}}\hspace{1pt}}}
\DeclareMathOperator\hhom{
	\hspace{.6pt}{\underline{\hspace{-.6pt}{\rm Hom}\hspace{-.6pt}}\hspace{1pt}}}

\DeclareMathOperator\hdeg{hdeg} 
\DeclareMathOperator\rank{rank}
\DeclareMathOperator\depth{depth}
\DeclareMathOperator\image{im} 
\DeclareMathOperator\length{\ell}

\newcommand\HH[3]{H^{#1}_{#2}(#3)}

\newcommand\noheight[1]{\raisebox{0pt}[0pt][0pt]{#1}}

\newcommand\dis{\displaystyle}

\definecolor{myorange}{RGB}{255, 160, 70}
\definecolor{darkgreen}{RGB}{0, 166, 0}

\begin{document}

\mbox{}
\vspace{-3ex}
\title[Quasipolynomial behavior via constructibility in multigraded algebra]%
      {Quasipolynomial behavior via constructibility in multigraded algebra}
\author[H.\,Dao]{Hailong Dao}
\address{Department of Mathematics\\University of Kansas\\Lawrence, KS 66045}
\urladdr{\url{http://people.ku.edu/~hdao}}
\author[E.\,Miller]{Ezra Miller}
\address{Mathematics Department\\Duke University\\Durham, NC 27708}
\urladdr{\url{http://math.duke.edu/people/ezra-miller}}
\author[J.\,Monta\~no]{Jonathan Monta\~no}
\address{School of Mathematical and Statistical Sciences\\Arizona State University\\Tempe, AZ 85287}
\urladdr{\url{https://math.la.asu.edu/~montano}}
\author[C.\,O'Neill]{\\Christopher O'Neill}
\address{ Mathematics and Statistics Department\\San Diego State University\\San Diego, CA 92182}
\urladdr{\url{https://cdoneill.sdsu.edu}}
\author[K.\,Woods]{Kevin Woods}
\address{Department of Mathematics\\Oberlin College\\Oberlin, OH 44074}
\urladdr{\url{https://www2.oberlin.edu/faculty/kwoods}}

\makeatletter
\@namedef{subjclassname@2020}{\textup{2020} Mathematics Subject Classification}
\makeatother

\date{25 September 2026.  The authors gratefully acknowledge the
American Institute of Mathematics (AIM), where much of this research
was carried out, and the Vietnam Institute for Advanced Study in
Mathematics (VIASM), where N.\,V.~Trung and L.\,T.~Hoa hosted EM for a
productive week of development.  EM~was partly supported by
NSF~DMS-2515765.  JM~was partly supported by NSF~DMS-2401522.  Thanks
to Lukas Waas for finding an error in an earlier proof of
Proposition~\ref{p:multiply-ideal-module}.  No aspect of this paper
used generative~AI in any way.}

\begin{abstract}
Piecewise quasipolynomial growth of Presburger counting functions
combines with tame persistent homology module theory to conclude
piecewise quasipolynomial behavior of constructible families of finely
graded modules over constructible commutative semigroup rings.
Functorial preservation of constructibility for families under local
cohomology, $\Tor$, and $\Ext$ yield piecewise quasipolynomial,
quasilinear, or quasiconstant growth statements for length of local
cohomology, $a$-invariants, regularity, depth; length of $\Tor$ and
Betti numbers; length of $\Ext$ and Bass numbers; associated primes
via $v$-invariants; and extended degrees, including the usual degree,
Hilbert--Samuel multiplicity, arithmetic degree, and
homological~degree.
\end{abstract}
\maketitle

\vspace{-1.2ex}
\setcounter{tocdepth}{2}
\tableofcontents

\vspace{-4.2ex}
\section{Introduction}\label{s:intro}

\addtocontents{toc}{\protect\setcounter{tocdepth}{1}}
\subsection*{Overview}\mbox{}

\smallskip
\noindent
Finiteness conditions on families of algebraic, combinatorial, or
geometric objects lead to predictable growth behavior.  Typically this
growth $n \mapsto P(n)$ is roughly polynomial, or more precisely
\emph{quasipolynomial}: there is a subgroup $\Lambda \subseteq \ZZ$ of
finite index and polynomials $\{P_{\ol m} \in \QQ[n] \mid \ol m =
m + \Lambda \in \ZZ/\Lambda\}$ such that
$$
  P(n) = P_{\ol m}(n) \quad\text{for}\quad n \in m + \Lambda.
$$
\begin{example}
Fix a field~$\kk$.  Prototypical polynomial growth arises as follows.
\begin{enumerate}
\item%
Hilbert polynomial: for $R = \bigoplus_{n\in\NN} R_n$ noetherian standard-graded
over $R_0 = \kk$,
$$
\dim_\kk(R_n)
=
a_d n^d + a_{n-1}n^{d-1} + \cdots + a_0
\in
\QQ[n]
\text{ for } n \gg 0.
$$

\item%
Ehrhart polynomial: for $\cP \subset \QQ^d$ a lattice polytope,
$$
\#(n\cP\cap\ZZ^d)
=
b_d n^d + b_{n-1}n^{d-1} + \cdots + b_0
\in
\QQ[n].
$$

\item%
Snapper polynomial: for any invertible sheaf $\cL$ on a dimension~$d$
projective scheme $X$ over~$\kk$, the Euler characteristic $\chi$ satisfies
$$
\chi(\cL^{\otimes n})
=
c_d n^d + c_{n-1}n^{d-1} + \cdots + c_0
\in
\QQ[n].
$$
\end{enumerate}\setcounter{separated}{\value{enumi}}
These are all honest polynomials, but quasipolynomials result after
\begin{enumerate}\renewcommand\labelenumi{\arabic{enumi}$'$.}
\item%
removing ``standard'', allowing the generators to lie in degrees
greater than~$1$, or
\item%
changing ``lattice'' to ``rational'', allowing arbitrary integer
bounding hyperplanes without requiring the vertices to be integer
points.
\end{enumerate}
\pagebreak[2]
Further instances in commutative algebra abound for an ideal $I$ in
a noetherian ring~$R$:
\begin{enumerate}\setcounter{enumi}{\value{separated}}
\item%
the Hilbert--Samuel function $\length(R/I^n)$ is a polynomial
for $n \gg 0$ if $I$ is~$\mm$-primary for a maximal ideal~$\mm$;

\item%
the number of generators $\mu(I^n)$ is a polynomial for $n \gg 0$;

\item%
$\Ext$, $\Tor$, Betti and Bass numbers all yield
polynomials for $n \gg 0$ \cite{kodiyalam1993}:
\begin{itemize}
\item%
$\length(\Tor_i^R(R/I^n, M))$,
\item%
$\length(\Ext^i_R(M,R/I^n))$,
\item%
$\mu(\Tor_i^R(R/I^n, M))$,
\item%
$\mu(\Ext^i_R(M,R/I^n))$;
\end{itemize}

\item%
the Castelnuovo--Mumford regularity $\reg(R/I^n)$ is linear for $n \gg 0$
when $I$ is homogeneous in a standard-graded ring \cite{kodiyalam2000,
cutkosky-herzog-trung99};

\item%
the $v$-invariant $v_\pp(I^n) = \min\{|\aa| \;\big|\,
R/I^n \supseteq R/\pp \text{ generated in}\,\deg\,\aa\}$ is linear for $n \!\gg\! 0$
when the prime~$\pp$ and~$I$\hspace{-.2ex} are homogeneous in a
standard-graded~ring~\cite{conca2024};

\item%
the local cohomology length $\length(H^i_\mm R/I^n)$ is quasipolynomial
for $n \gg 0$ if $I$ is a monomial ideal in a polynomial ring and
$\length(H^i_\mm R/I^n) < \infty$ for $n \gg 0$ \cite{dao-montano-2019}.
\end{enumerate}
\end{example}

Classically, this quasipolynomial growth derives from the noetherian
condition on some related construction, often a Rees algebra or
similar.  However, quasipolynomial growth has been observed in
settings that are demonstrably not noetherian.

\begin{example}\label{e:intro}
Let $I_n = \<x^n, y\> \subseteq \kk[x,y]$ be a graded family of ideals
parametrized by $n \in \NN$, meaning that $I_n I_m \subseteq I_{n+m}$.  The Rees
algebra $\bigoplus_{n\in\NN} I_n t^n$ is not noetherian.  However, the numerical
values of classical functors, such as the length
$$
\length\bigl(\Tor_1(\kk[x,y]/I_n, M)\bigr)
$$
for a fixed noetherian $\ZZ^2$-graded $\kk[x,y]$-module~$M$ still exhibit
quasipolynomial growth.  For instance, when
$M = \kk[x,y]/\<x^3y, y^2\>$, the lengths are
$$
2, 4, 6, 7, 8, 9, 10, \ldots.
$$
\end{example}

Prior research on quasipolynomial behavior typically considers
a single, non-recursive functor applied to a
noetherian family of ideals,
with an entire paper devoted to each such setting,
e.g.~\cite{dao-montano-2019, nguyen-trung19}.  In contrast, the
methods of this paper encompass all
prior
multigraded cases
and unify them as stemming from a
pervasive functorial constructibility principle.
Consequently, the
results here vastly generalize all of the special cases, in the
multigraded setting, by showing that quasipolynomiality
\begin{itemize}
\item%
holds for families of multigraded modules, not merely ideals;
\item%
does not require noetherian hypotheses on the family;
\item%
works for multiparameter families; and
\item%
persists under repeated---that is, recursive---applications of typical
functors.
\end{itemize}

More precisely, the main theoretical results of the paper assert
functorial preservation of constructibility on families under local
cohomology, $\Tor$, and $\Ext$
(Theorems~\ref{t:loc-coh-family},~\ref{t:tor-family},
and~\ref{t:ext-family}).
The payoff arrives in \S\ref{s:applications}, with specific piecewise
quasipolynomial, quasilinear, or quasiconstant growth statements for
\begin{itemize}
\item%
length of local cohomology, $a$-invariants, regularity, and depth
(Theorem~\ref{t:loc-coh-final});
\item%
length of $\Tor$ and Betti numbers
(Theorem~\ref{t:tor-final});
\item%
length of $\Ext$ and Bass numbers (Theorem~\ref{t:ext-final});
\item%
associated primes via $v$-invariants (Theorem~\ref{t:v-invariants});
and
\item%
extended degrees, including usual degree, Hilbert--Samuel multiplicity,
\mbox{arithmetic} degree, and homological degree
(Theorem~\ref{t:hilbert-coeffs}, Corollary~\ref{c:hdeg}, and
Remark~\ref{r:adeg}).
\end{itemize}
In contrast to previous approaches, which
show functors on well behaved families enjoy well behaved numerical
behavior, the conceit here is that in the multigraded setting,
functors on families preserve regulated structure.  Only later are
vector-space dimensions extracted, with good numerical behavior
certified by the preserved structure.

%
%
%

Beyond the theoretical advances and
arbitrarily complicated, recursive novel specific occurrences of
quasipolynomial behavior (see Examples~\ref{e:combine-constructions},
\ref{e:multigraded-families}, \ref{e:direct-sums},
\ref{e:weird-quasilinear}, and~\ref{e:length-of-tor}, for instance),
the additional non-noetherian generality is promising for potential
applications in Ehrhart theory, because harmonic algebras of polytopes
\cite{reiner-rhoades25} need not be finitely generated \cite{cavey25}
but are still conjectured to have rational generating functions
\cite[Conjecture~5.5]{reiner-rhoades25}.

Given the lack of finite generation in Example~\ref{e:intro} and the
main theorems, what kind of finiteness drives the conclusions there?
Our answer, in the setting of multigraded modules over semigroup
rings, melds two themes from recent advances in disparate areas,
namely (i)~Presburger arithmetic at the intersection of logic,
computer science, combinatorics, and polyhedral geometry on one hand,
and (ii)~tame module theory from applied topology on the other.  The
upshot is that
\begin{itemize}
\item%
it makes sense for a $\ZZ^d$-graded module over a semigroup ring
$\kk[Q_+]$ with $Q_+ \subseteq\nolinebreak \ZZ^d$
to be \emph{constructible} in a Presburger sense, materially weaker
than noetherian;
\item%
constructible families of such modules exhibit
quasipolynomial growth; and
\item%
enough functors preserve constructibility to conclude
\item%
piecewise quasipolynomial growth in surprisingly rich and varied
circumstances.
\end{itemize}
Roughly speaking, after the introduction to Presburger arithmetic in
\S\ref{s:arithmetic}, these considerations respectively occupy
\S\ref{s:constructible-mods} and~\ref{s:families}; \S\ref{s:numerics}
and~\ref{b:quasipolynomiality}; \S\ref{s:flat}, \ref{s:functors},
and~\ref{s:functors-on-families}; and \S\ref{b:from-ideals},
\ref{b:products}, and~\ref{s:applications}.  The rest of this
Introduction goes into more detail about the methods, background, and
results, organized in the same order as the paper is presented.

\subsection*{Numerical constructibility via Presburger arithmetic}\mbox{}

\smallskip
\noindent
Presburger arithmetic is the first-order logic of the natural numbers
under addition.  Presburger
introduced it in 1929 \cite{presburger1930} and proved its decidability.
Its modern importance \cite{woods2015} stems from its ability to describe seemingly complicated sets in the free
abelian group~$\ZZ^d$---for example, defined using combinations of lattices, polyhedra, boolean operations, and quantifiers---in terms of simpler
sets, such as  translates of affine semigroups with linearly independent
generating sets (Theorem~\ref{t:presburger}).  The
consequences for algebraic combinatorics include rationality of
generating functions with denominators of the special form familiar to
algebraists from Hilbert series \cite{woods2015}.  And for enumerative
combinatorics, the consequences include ubiquitous quasipolynomial
growth, or more precisely piecewise quasipolynomial growth, meaning
quasipolynomial behavior on each of finitely many polyhedral regions
(Definition~\ref{d:piecewise-quasipolynomial}).


The lattice points that Presburger arithmetic organizes into
\emph{semisimple sets} (Definition~\ref{d:semisimple}) serve as
multigraded basis vectors for input modules: monomial ideals and their
powers, or symbolic powers, or integral closures of powers, and so on.
Coarsening the grading to a standard or other linear grading, these
modules can be thought of as families of vector spaces
whose dimensions grow quasipolynomially because Presburger arithmetic
dicates that counting functions behave that way
(Theorem~\ref{t:num-constructible=>quasipol}).  That is our main
numerical observation stemming from Presburger arithmetic.

However, the goal is to conclude quasipolynomial growth for modules
output by functors applied to the input modules: $\Tor$, $\Ext$, local
cohomology, and so on.  These output modules often do not have lattice
points as basis elements; instead, the lattice points parametrize a
family of finite-dimensional vector spaces that is
\emph{constructible} (Definition~\ref{d:constructible}) in the sense
that the vector spaces vary in Presburger definable ways.  That is, a
module is constructible when, roughly speaking, the grading group can
be subdivided into Presburger definable regions over which the graded
pieces of the module do not vary (take ``class~$\XX$'' to be
``Presburger definable'' in Definition~\ref{d:tame}).  More to the
point,
for enumeration, the family of vector spaces is \emph{numerically
constructible} (Definition~\ref{d:numerical-family}) in the sense that
the dimensions of the vector spaces are constant on each of finitely
many Presburger definable subsets that partition the set of parameter
lattice points.  Any such numerically constructible family
automatically exhibits piecewise quasipolynomial growth by
Theorem~\ref{t:num-constructible=>quasipol}, the main numerical result
again.

\subsection*{Presburger constructible families}\mbox{}

\smallskip
\noindent
Numerical constructibility is too weak a condition to hope that it
might persist after applying a homological construction, since vector
spaces can share dimensions without being naturally isomorphic; see
Examples~\ref{e:numerical}, \ref{e:not-presburger-family},
and~\ref{e:semisimple}.  The key, then, is to identify conditions
under which the desired functors preserve not merely numerical
constructibility but module-theoretic constructibility of the input in
the sense of Definition~\ref{d:constructible}.
Tensoring with localizations is the easiest and, for our purposes,
most elemental example; see \S\ref{s:flat}, which also includes
background and notation
for combinatorics of subsemigroups of~$\ZZ^d$ and their multigraded
modules.

That localization preserves constructibility of modules and morphisms
(Lemma~\ref{l:summable}) translates without much fuss into assertions
that local cohomology, tensor products, and higher $\Tor$ functors
preserve constructibility, as well (Theorems~\ref{t:cech}
and~\ref{t:tor}).  By appealing to Matlis duality
(Definition~\ref{d:matlis}), modules of homomorphisms and higher
$\Ext$ functors also do (Theorem~\ref{t:ext}).

These functorial preservation results are for individual modules.  But
the primary interest is in families of modules.  Classically these
families are indexed by a single integer: a power or symbolic power,
for example.  One of the benefits of constructible module theory is
that it adapts to apply to the parameters of the family as easily as
it applies to the grading group of the individual modules in the
family.

Organizing the members of the family is a \emph{Rees monoid}
(Definition~\ref{d:rees-monoid}): the Presburger lattice-point
analogue of the classical Rees algebra.  A~constructible family is
then
simply a constructible module over the Rees monoid algebra
(Definition~\ref{d:family}).  Note that the noetherian hypothesis is
not
relevant; instead the focus is on whether the monoids---and the modules
graded by them---are constructible in the (equivalent) semisimple or
Presburger senses, for that is what produces piecewise quasipolynomial
behavior (Lemma~\ref{l:unassuming}, given
Theorem~\ref{t:num-constructible=>quasipol}).

What results is a host of families of ideals and related objects,
easily built from operations like powers, colons, saturation, taking
multiplier ideals, integral closures, sums, intersections, or products
that are automatically constructible.  The following extensive list is
compiled from Theorem~\ref{t:ideals-families-strong},
Proposition~\ref{p:ideals-families}, and
Proposition~\ref{p:multiply-ideal-module}.

\begin{theorem}\label{t:intro-ideals-family}
Fix a constructible family $\{I_n\}_{n \in \ZZ}$ of monomial ideals
in an affine semigroup
ring $R = \kk[Q_+]$.
Let $K \subseteq R$ be a monomial ideal.  Each of the following is a
constructible family of monomial ideals of~$R$.
\begin{itemize}
\item%
The constant family $\{K\}_{n \in \ZZ}$.
\item%
The powers $\{I_n^{an + b}\}_{n \in \ZZ}$ for any fixed $a,b \in \ZZ$.
\item%
The integral closures $\{\ol{I_n}\}_{n \in \ZZ}$.
\item%
For $\QQ$-Gorenstein~$R$, the multiplier ideals
$\{\cJ(I_n)\hspace{-.37pt}\}_{n \in \ZZ}$.
\end{itemize}
If $\{I_\nn\}_{\nn \in \ZZ^k}$ is a multiparameter constructible family of
monomial ideals in
$R = \kk[Q_+]$, then
the following families of ideals of~$R$ are again constructible.
\begin{itemize}
\item%
The saturations $ \{I_\nn:_R K^{\infty}\}_{\nn \in \ZZ^k}$.
\item%
The colon ideals $\{I_\nn :_R K^{\aa\cdot\nn+b}\}_{\nn \in \ZZ^k}$
for fixed $\aa = (a_1, \ldots, a_k) \in \ZZ^k$ and $b \in \ZZ$.
\end{itemize}
If instead of a fixed ideal~$K$ a constructible family $\{K_\mmm\}_{\mmm \in
\ZZ^\ell}$ of monomial ideals is given, then the following families of
ideals of~$R$ are again constructible.
\begin{itemize}
\item%
The sums $\{I_\nn + K_\mmm\}_{\nn\mmm \in \ZZ^{k+\ell}}$.
\item%
The intersections $\{I_\nn \cap K_\mmm\}_{\nn\mmm \in \ZZ^{k+\ell}}$.
\item%
The quotients $\{I_\nn/K_\mmm\}_{\nn\mmm \in \ZZ^{k+\ell}}$, when $K_\mmm \subseteq I_\nn$ for
every~$\nn$ and~$\mmm$.
\end{itemize}
If $\{K_\mmm\}_{\mmm \in \ZZ^\ell}$ is a constructible family of $Q$-modules (i.e.,
not necessarily ideals) and $I_\0 = \kk[Q_+]$ is the unit ideal, then
the following are constructible families of $Q$-modules.
\begin{itemize}
\item%
The products $\{I_\nn K_\mmm\}_{\nn\mmm \in \ZZ^{k+\ell}}$.
\item%
The quotients $\{K_\mmm/I_\nn K_\mmm\}_{\nn\mmm \in \ZZ^{k+\ell}}$.
\end{itemize}
\end{theorem}

The formal statements of constructibility in \S\ref{s:families} for
these families name specific Rees monoids underlying each
constructibility assertion.  However,
in practice it is
only necessary to know what each family is and that each family is
constructible over some Rees monoid; it is usually not important which
Rees monoids are actually selected.  Our theorems specify explicitly
how to build suitable Rees monoids from
given data.

Strikingly complicated families---much more elaborate than previously
available from noetherian family theory---emerge from meager building
blocks
using Theorem~\ref{t:intro-ideals-family} because it applies
recursively: taking constructible families for input guarantees that
the output remains constructible.  Even before taking into account the
recursive functors applied to the input family, this flexibility in
the input family itself already enriches the theory far beyond what
was previously available.

\begin{example}\label{e:combine-constructions}
When $R = \kk[\NN^d]$ is a polynomial ring and $I, J, K$ are monomial
ideals, the family $\{\ol{\mathcal{J}(I^n:J^{3n-2}) :_R K^\infty}\}_{n\in\NN}$ is
constructible.  See Example~\ref{e:combine-constructions'} for more
detail in this case, including an explicit Rees monoid underlying this
constructibility.
\end{example}

Most proofs in \S\ref{s:families} proceed directly via
Presburger arithmetic, avoiding the need to demonstrate directly that
various sets are
semisimple.
An exception is Proposition~\ref{p:multiply-ideal-module}, to the
effect that the product of constructible families of ideals
and modules is a (doubly indexed) constructible family of modules; its
proof instead
applies general homological theory of modules over posets
\cite{hom-alg-poset-mods}, using the syzygy theorem there
\cite[Theorem~6.12]{hom-alg-poset-mods} to reduce the question to one
about Minkowski sums of ideals and arbitrary upsets.  These varied
techniques offer a sample of the flexibility of constructible module
theory,
with~its~complementary~roots~in~logic~and~\mbox{applied}~\mbox{topology}.

\subsection*{Functorial preservation of constructibility}\mbox{}

\smallskip
\noindent
The main theoretical heart of the paper asserts that constructibility
of a family is preserved by local cohomology $H^i_I(-)$ supported on
an ideal~$I$ (Theorem~\ref{t:loc-coh-family}), $\Tor$
(Theorem~\ref{t:tor-family}), and~$\Ext$ (Theorem~\ref{t:ext-family}).
The following is a summary; the more precise versions should be cited
for treatment of
underlying Rees monoids and, in the $\Tor$ and~$\Ext$ cases,
a bit more generality, with $N$ replaced by a (very
constrained)~family.

\begin{theorem}\label{t:functors-intro}
Fix a constructible family $\{M_\nn\}_{\nn\in\ZZ^k}$ of~$Q$-modules,
a monomial ideal $I$ in the monoid algebra~$R = \kk[Q_+]$, and a
localization $N$ of a noetherian or artinian $R$-module.  Each of the
following is a constructible family for any homological index~$i$.
\begin{itemize}
\item%
$\bigl\{H^i_I(M_\nn)\bigr\}_{\nn\in\ZZ^k}$
\vspace{.5ex}
\item%
$\dis
  \bigl\{\Tor_i^Q(M_\nn, N)\bigr\}_{\nn\in\ZZ^k}
$
\vspace{.5ex}
\item%
$\dis
  \bigl\{\eext^i_Q(M_\nn, N)\bigr\}_{\nn\in\ZZ^k}
$
\vspace{.5ex}
\item%
$\dis
  \bigl\{\eext^i_Q(N, M_\nn)\bigr\}_{\nn\in\ZZ^k}
$
\vspace{.5ex}
\end{itemize}
\end{theorem}
It bears emphasizing that not only does quasipolynomial numerical
behavior result from the families
itemized in Theorem~\ref{t:functors-intro}, but constuctibility of the
\pagebreak
family itself is preserved.  Hence these functors can be applied
recursively---a new phenomenon.

The proof of Theorem~\ref{t:functors-intro} in
\S\ref{s:functors-on-families} is little more than the join of the
corresponding (easy) assertions for individual modules---as opposed to
families of modules---in~\S\ref{s:functors} and the way families are
defined in \S\ref{b:rees-monoids}--\ref{b:constructible-families} as
individual modules over Rees monoids.

\subsection*{Quasipolynomial growth}\mbox{}

\smallskip
\noindent
The culmination arrives in \S\ref{s:applications}, with specific
piecewise quasipolynomial, quasilinear, or quasiconstant growth
statements for length of local cohomology, $a$-invariants, regularity,
and depth (Theorem~\ref{t:loc-coh-final}); length of $\Tor$ and Betti
numbers (Theorem~\ref{t:tor-final}); length of $\Ext$ and Bass numbers
(Theorem~\ref{t:ext-final}); associated primes via $v$-invariants
(Theorem~\ref{t:v-invariants}); and extended degrees, including the
usual degree, Hilbert--Samuel multiplicity, arithmetic degree, and
homological degree (Theorem~\ref{t:hilbert-coeffs},
Corollary~\ref{c:hdeg}, and Remark~\ref{r:adeg}).  A~cross-referenced
compendium is as follows.
			
\begin{theorem}\label{t:quasipoly-growth}
Any constructible family $\{M_\nn\}_{\nn\in\ZZ^k}$ of $Q$-modules yields
piecewise quasi\-polynomials including the following, in which $J$ is
a monomial ideal in the affine semigroup ring $R = \kk[Q_+]$ and $L$ is
a localization of an artinian or noetherian $R$-module.
\begin{enumerate}
\item
$\length\Bigl(\dis\bigoplus_{\alpha|\nn| + \beta \leq |\aa| \leq
\gamma|\nn| + \delta} H^i_J(M_\nn)_\aa\Bigr)$
\hfill
(bounded-degree local cohomology,
Thm~\ref{t:loc-coh-final}.\ref{i:lengths-local-coh})

\item%
$\length\bigl(H^i_\mm(M_\nn)_{\geq \alpha|\nn| + \beta}\bigr)$
\hfill
(bounded-degree local cohomology,
Thm~\ref{t:loc-coh-final}.\ref{i:lengths-local-coh})%
\vspace{.5ex}

\item%
$\length\bigl(\Tor^R_i(M_\nn, L)\bigr)$
\hfill
(length of\/~$\Tor$, Thm~\ref{t:tor-final}.\ref{i:gens-tor})%
\vspace{.5ex}

\item%
$\length\bigl(\Ext_R^i(L, M_\nn)\bigr)$
\hfill
(length of\/~$\Ext$, Thm~\ref{t:ext-final}.\ref{i:gens-ext})%
\vspace{.5ex}

\item%
$\length\bigl(\Ext_R^i(M_\nn, L)\bigr)$
\hfill
(length of\/~$\Ext$, Thm~\ref{t:ext-final}.\ref{i:gens-ext})%
\vspace{.5ex}
\end{enumerate}\setcounter{separated}{\value{enumi}}
These five remain true after replacing length $\length$ by
\begin{itemize}
\item%
number $\mu$ of generators or
\hfill
(Thms~\ref{t:tor-final}.\ref{i:bettis}
and~\ref{t:ext-final}.\ref{i:bass})
\item%
maximum or minimum nonzero degree under any linear function $\lambda :
\text{degrees} \to \ZZ$\\
($\mu$, $\max$, and $\min$ are piecewise quasi\emph{linear})
\hfill
(Cors~\ref{c:loc-coh-family-quasipol}, \ref{c:tor-family->quasipol},
and~\ref{c:ext-family->quasipol})
\end{itemize}
In addition, the following are piecewise quasipolynomial (or better,
as indicated).
\begin{enumerate}\setcounter{enumi}{\value{separated}}
\item%
$\beta_i(M_\nn) = \dim_\kk\Tor_i^R(M_\nn, \kk)$
\hfill
(Betti numbers, Thm~\ref{t:tor-final}.\ref{i:bettis})
  	
\item\label{i:bass-intro}%
$\mu_i(\pp, M_\nn) = \dim_{(R/\pp)_\pp} \Ext_R^i(R/\pp, M_\nn)_\pp$
\hfill
(Bass numbers, Thm~\ref{t:ext-final}.\ref{i:bass})

\item%
$a_i(M_\nn) = \max\{|\aa| \mid H^i_\mm(M_\nn)_\aa \neq 0\}$, quasilinear
\hfill
($a$-invariant, Thm~\ref{t:loc-coh-final}.\ref{i:a-inv})

\item%
$\depth(M_\nn)$, quasiconstant
\hfill
(depth, Thm~\ref{t:loc-coh-final}.\ref{i:depths})

\item%
$\dim(M_\nn)$, quasiconstant
\hfill
(dimension, Thm~\ref{t:loc-coh-final}.\ref{i:depths})

\item%
$\reg(M_\nn) = \max\{a_i(M_\nn) + i\}$, quasilinear
\hfill
(regularity, Thm~\ref{t:loc-coh-final}.\ref{i:regs})%
\vspace{.25ex}

\item%
$\reg\bigl(\Ext_R^i\bigl(\ol{\cJ(I_n : J^{3n-2}) : K^\infty}, R/K\bigr)\bigr)$
\hfill
(nested functors; see Example~\ref{e:weird-quasilinear})%

\item%
$v_\pp(M_\nn) = \min\{|\aa| \;\big|\,
M_\nn \supseteq R/\pp \text{ in}\,\deg\,\aa\}$, quasilinear
\hfill
($v$-invariant, Thm~\ref{t:v-invariants})

\item%
$\deg(M_\nn)$ over Gorenstein standard-graded $\kk[Q_+]$
\hfill
(usual degree, Thm~\ref{t:hilbert-coeffs})

\item%
$\adeg(M_\nn)$ over Gorenstein standard-graded $\kk[Q_+]$
(arithmetic degree, Rmk~\ref{r:adeg})

\item\label{i:hdeg}%
$\hdeg(M_\nn)$ over Gorenstein standard-graded $\kk[Q_+]$
(homological degree, Cor~\ref{c:hdeg})
\end{enumerate}
\end{theorem}

The degrees of the relevant quasipolynomials are specified in the more
detailed and often more general formal statements in
Sections~\ref{s:functors-on-families} and~\ref{s:applications}.
Typically the degree is the rank of the underlying affine
semigroup~$Q_+$, but lower degrees occur as specified in the
quasilinear or quasiconstant parts of Theorem~\ref{t:quasipoly-growth}
or, for example, in the case of Bass numbers at nonmaximal primes
(Theorem~\ref{t:quasipoly-growth}.\ref{i:bass-intro}), where the
degree decreases by the dimension of the prime; see
Theorem~\ref{t:ext-final}.\ref{i:bass}.

Note again that the constructible family~$M_\nn$ serving as input to
Theorem~\ref{t:quasipoly-growth} can already have various functors
within it, such as local cohomology, Ext, Tor, or composites of these,
by Theorem~\ref{t:functors-intro}.  Furthermore, the base input family
of ideals (if there is one) can already have various operations
applied to it, such as sums, products, powers, integral closures,
multiplier ideals, and so on, as specified in
Theorem~\ref{t:intro-ideals-family}.  Amenability to nesting is
evident and crucial in Theorem~\ref{t:quasipoly-growth}.\ref{i:hdeg}:
homological degree $\hdeg$ is a linear combination of ordinary
degrees~$\deg(M_\nn)$ and degrees of iterated
$\Ext$~\mbox{modules}~of~the~form%
\enlargethispage{-1.9ex}%
\vspace{.5ex}
$$
  \Ext_R^{i_1}(\Ext_R^{i_2}(\cdots(\Ext_R^{i_r}(M_\nn, R),R),\ldots,R)
\vspace{.5ex}
$$
Growth of $\hdeg$ was not previously known to be polynomially bounded
in any setting, even for simple input families like $\{R/I^n\}_{n\in\NN}$
for a monomial ideal~$I$.



The diversity and quantity of the conclusions in
Theorem~\ref{t:quasipoly-growth} account for the length of this paper.
The theory is short, in comparison.
For example, in \S\ref{s:families},
navigating the numerous families of ideals in
\mbox{\S\ref{b:from-ideals}--\ref{b:products}} occupies one and a~half
times
the four pages of theory introducing Rees monoids, constructible
families, and the consequent quasipolynomiality in
\S\ref{b:rees-monoids}--\ref{b:quasipolynomiality}.  Compare also to
prior
papers, where the focus of each
was typically numerical behavior of one functor applied to a family of
ideals.

Analyzing our results
in the context of prior literature in more detail, the general spirit
is: most of the quasi\-polynomial growth results are known for
individual functors applied to powers of ideals or, with a small
modification, to noetherian graded families of ideals.  Our theory
restricts to the multigraded setting, but once there it generalizes by
relaxing the noetherian hypothesis on a family of ideals indexed by
integers to the much more inclusive constructible hypothesis on a
family of modules indexed by integer vectors.
Consequently, since the central conceit is that various functors such
as $\Tor$, $\Ext$, and local cohomology preserve constructibility
(Theorem~\ref{t:functors-intro} or \S\ref{s:functors-on-families},
particularly Theorems~\ref{t:loc-coh-family}, \ref{t:tor-family},
and~\ref{t:ext-family}), any sequence of such functors outputs
piecewise quasipolynomial numerical behavior given a constructible
input family (\S\ref{b:functorial-quasipol}, particularly
Corollaries~\ref{c:loc-coh-family-quasipol},
\ref{c:tor-family->quasipol}, and \ref{c:ext-family->quasipol}).  Rich
sources of input constructible families abound
(Theorem~\ref{t:intro-ideals-family} or
\S\ref{b:from-ideals}--\ref{b:products}) because operations on families
of ideals, such as the formation of powers, integral closures,
multiplier ideals (Theorem~\ref{t:ideals-families-strong}), colon
ideals, saturations, sums, products, and intersections
(Proposition~\ref{p:ideals-families})
\mbox{preserve}~\mbox{constructibility}.

\pagebreak
\subsection*{Broader context}\mbox{}

\smallskip
\noindent
The algebraic and homological machinery that drives the theory here
was conceived
for Topological Data Analysis, specifically to
subdivide spaces of parameters in persistent homology into
subanalytic, semialgebraic, or polyhedral constant regions.  Part of
the purpose was to connect the multigraded algebra of multiparameter
persistence \cite{hom-alg-poset-mods, essential-real} to the
sheaf-theoretic side \cite{kashiwara-schapira2018,
kashiwara-schapira2019}, which turns out to be more or less equivalent
to the subanalytic case \cite{strat-conical}.
But few hypotheses were needed to enable the homological theory of
multiparameter persistence, so it was developed under a unifying
``class~$\XX$'' condition analogous to an $o$-minimal structure.
An integral observation, for the current quasipolynomial purpose, is
that the class~$\XX$ condition has an arithmetic incarnation: instead
of continuously
subdividing subsets of real vector spaces into subanalytic or
semialgebraic pieces,
discretely
subdivide subsets of finitely generated free abelian groups  into
simple pieces.

Certain aspects of semisimple constructibility theory are developed
here in their natural generality, namely for a \emph{class~$\XX$}
family of subsets (Definition~\ref{d:tame}), because~it requires no
additional effort.  However, the reader interested solely in
semisimple (equivalently, by Theorem~\ref{t:presburger}, Presburger
definable) sets and the resulting \mbox{semisimple} constructibility
can always take ``class~$\XX$'' to mean ``semisimple'' or ``semisimply
construct\-ible''.  For this reason, we introduce Presburger
arithmetic and semisimple sets~(\S\ref{s:arithmetic}) before the
general discussion of tameness and class~$\XX$ objects and
\mbox{morphisms}~(\S\ref{s:constructible-mods}).

The arithmetic take on subdividing degrees of multigraded modules is
not entirely new.
It was initiated---in the primitive form of ``sector partitions''
\cite{helm-miller2005}---for algo\-rithmic purposes via monomial
matrices \cite{alexdual2000}.  Treating graded degrees of local
cohomology as flexible geometric subsets of vector spaces has also
seen deep applications to hypergeometric functions
\cite{matusevich-miller-walther2005}, where Zariski closures of such
subsets play a pivotal role.  For the present purpose, tracing the
idea back to its inception
suggests that monomial matrices have the potential to make the current
constructible constructions computable
via general techniques surrounding upset and downset
resolutions~\cite{hom-alg-poset-mods}.

Perhaps with algorithmic concerns
in mind, a natural next step beyond the present work would be to
determine bounds on the constituents of various subdivisions that
occur here, including in semisimple subdivisions numerically
subordinate to given modules or families
(Definitions~\ref{d:semisimple-partition}
and~\ref{d:numerically-subordinate}), polyhedral subdivisions that
underlie the ``piecewise'' part of piecewise quasipolynomial
(Definition~\ref{d:piecewise-quasipolynomial}), and coset partitions
that underlie the ``quasi'' part (Definition~\ref{d:quasipolynomial}).
It would also be valuable to ascertain whether methods based on
constructible modules can conclude anything about multigraded
regularity.  The degrees and leading coefficients of the
quasipolynomials in Section~\ref{s:applications} vary in ways that
would be interesting to study.  For example, when the family of
functors is indexed by $n\in\NN$, constancy in the degree and leading
coefficient of one of these quasipolynomials implies existence of the
limit of the corresponding sequence over a power of~$n$, which relates
to results in the literature \cite{cutkosky14}.  It is also of
interest whether any of our results extend to families of modules over
graded rings endowed with coarser gradings.
%

\addtocontents{toc}{\protect\setcounter{tocdepth}{2}}

\begin{conv}\label{conv:conventions}
Throughout, $\kk$ denotes an arbitrary field and $Z$ denotes a free
abelian group.  Starting in \S\ref{b:localization},
$Q$~denotes a class~$\XX$ group (Definition~\ref{d:class-X}) unless
otherwise stated; see, in particular, the opening of each section.
Before that, in \mbox{\S\ref{s:arithmetic}--\S\ref{b:complexes}},
explicit hypotheses on~$Q$ are provided
each~time.  Note that
for the family~$\XX$ of all subsets of~$Q$, being a class~$\XX$ group
merely means $Q$ is a full (Definition~\ref{d:pogroup}) partially
ordered subgroup of a finite-dimensional real vector~space.%
\end{conv}

\section{Presburger arithmetic}\label{s:arithmetic}

\subsection{Presburger definability}\label{b:definable}

\begin{defn}\label{d:presburger-formula}
A \emph{Presburger formula} is a boolean formula with variables
in~$\ZZ$ that can be written using quantifiers
($\exists$~and~$\forall$), boolean operations (or, not, and), and
integer affine-linear inequalities in the variables.
\end{defn}

Henceforth we use bold letters, such as~$\uu$, to denote either group
elements or vectors (of numbers or variables).

\begin{defn}\label{d:free-variable}
A variable in a Presburger formula is \emph{free} if it is not
quantified.  A Presburger formula is written $F(\uu)$ to indicate that
$\uu$ is the list of free variables.
\end{defn}

\begin{defn}\label{d:presburger-definable}
A set $S \subseteq \ZZ^d$ is \emph{Presburger definable} if it can be
defined via a Presburger formula~$F(\uu)$, meaning that $S = \{\uu \in
\ZZ^d \mid F(\uu)\}$.
\end{defn}

\begin{remark}\label{r:implication}
In any formula defining a Presburger definable set, we sometimes write
$F(\uu) \Rightarrow G(\uu)$ to denote $\neg F(\uu) \vee G(\uu)$.  The latter notation is
justified as the two formulas $F(\uu) \Rightarrow G(\uu)$ and $\neg F(\uu) \vee G(\uu)$ are
logically equivalent.
\end{remark}

\begin{example}\label{e:presburger-formula}
The positive integers with remainder~$1$ mod~$3$ are defined by the
formula
$$
  F(u) = (\exists c\in \ZZ)(c\gs  0\, \wedge \,u=3c+1).
$$
Note that the variable~$c$ in~$F(u)$ is not free.
\end{example}

\begin{example}\label{e:presburger-formula2}
The formula 
$$
G(u_1,u_2)=\left( 2u_1+u_2\gs 3\, \wedge\, 3u_1-u_2\gs 2\right)
$$
defines the set of integer points in the translated cone $(1,1) +
\RR_{\gs 0}\triangle$, where $\triangle$ is the triangle with vertices
$(0,0), (1,3), (1,-2)$.
\end{example}


\subsection{Decomposing Presburger sets}\label{b:semisimple}

\begin{defn}\label{d:semisimple}
Given a finitely generated free abelian group~$Q$, a subset $S
\subseteq Q$ is \emph{simple} if $S$ is a translate of a subsemigroup
of~$Q$ isomorphic to~$\NN^k$ for some~$k$; that is, $S = \qq + N$ for
some $\qq \in Q$ and $N$ generated by linearly independent elements
of~$Q$.  The subset $S \subseteq Q$ is \emph{semisimple} if it is a
finite disjoint union of simple subsets of~$Q$.
\end{defn}

\begin{defn}\label{d:affine-semigroup}
An \emph{affine semigroup} is a semigroup that is isomorphic to a
finitely generated submonoid of a free abelian group.
\end{defn}

\begin{remark}\label{r:smooth}
A semigroup isomorphic to~$\NN^k$ for some~$k$ is often called a
\emph{smooth} semigroup (or smooth affine semigroup) because
smoothness is equivalent to the semigroup algebra being the coordinate
ring of a smooth affine toric~variety.  A~translate of an affine
semigroup has been called a \emph{linear} set in the literature, with
a finite union of such sets being called \emph{semilinear}
\cite{eilenberg-schutzenberger1969}.  The terminology in
Definition~\ref{d:semisimple} comes from
\cite{eilenberg-schutzenberger1969} and persists in modern
developments \cite{d'alessandro-intrigila-varricchio2012}. 
\end{remark}

The following characterizations of Presburger definable sets
summarizes well known results \cite{eilenberg-schutzenberger1969}; see
\cite{d'alessandro-intrigila-varricchio2012} for a recent viewpoint.
Part of a proof is included here because that part is short and may
provide insight to readers from commutative algebra.

\begin{thm}\label{t:presburger}
For a subset $S \subseteq \ZZ^d$, the following are equivalent.
\begin{enumerate}
\item%
$S$ is Presburger definable.

\item%
$S$ is a finite union of sets of the form $P \cap (\qq + Q)$, where $P
\subseteq \RR^d$ is a rational polyhedron, $\qq \in \ZZ^d$, and $Q$ is
a subgroup of $\ZZ^d$.

\item%
$S$ is a finite union of translates of affine semigroups.

\item%
$S$ is semisimple.
\end{enumerate}
In particular, given the final item, any of the unions here can be
chosen to be disjoint.
\end{thm}
\begin{proof}
1 $\iff$ 2 by \cite[Theorem 1.15]{woods2015}.

2 $\iff$ 3 by \cite[p.234~(19)]{schrijver1986}.

3 $\iff$ 4 uses three reductions.
\begin{itemize}
\item%
Every affine semigroup is a disjoint union of translates of normal
affine semigroups, where an affine semigroup $A$ is \emph{normal} if
$A = \RR_{\gs 0} A \cap \ZZ A$.

\item%
Every rational cone admits a unimodular triangulation: a simplicial
subdivision in which each cone is generated by a basis for the ambient
lattice.

\item%
Any unimodular triangulation of a cone~$C$ induces an expression of
the lattice points in~$C$ as a finite disjoint union of translates of
semigroups each isomorphic to~$\NN^k$ for some~$k$.
\end{itemize}
The first reduction follows from \cite[Theorem~5.2]{stanley1982}; see
also \cite[Lemma~2.2]{affine-strat} for a short proof based on primary
decomposition in commutative algebra.  The second reduction is a
special case of \cite[Theorem~2.72]{bruns-gubeladze2009}.
For the third reduction, the translated semigroup associated to each
face of the triangulation is the set of lattice points in the relative
interior of that face; the relevant semigroup is isomorphic to $\NN^k$
because each face is unimodular.
\end{proof}

\begin{remark}\label{r:unimodular}
Our proof of Theorem~\ref{t:presburger} can work with an
arbitrary---that is, not necessarily unimodular---triangulation: add
one simple subset to each lattice point in the fundamental
parallelepiped of each simplicial face in the triangulation.
\end{remark}

\subsection{Semisimple atoms}\label{b:atoms}

\begin{prop}\label{p:atoms-definable}
Fix a submonoid~$A$ with trivial unit group in a free abelian
group~$Q$.  Then $A$ is semisimple if and only if the set of atoms
of~$A$ is semisimple in~$Q$.
If $A$ is semisimple and $U \subseteq A$ is an upset (so $\aa \in U$ and $\aa \preceq \aa'
\implies \aa' \in U$), then $U$ is semisimple if and only if the minimal
generating set of~$U$ is semisimple.
\end{prop}
\begin{proof}
Fix an isomorphism $Q \simto \ZZ^d$, and let $\hat A$ be the image
of~$A$.  Assume $A$ is semisimple, so by Theorem~\ref{t:presburger}
its image $\hat A$ is definable by a Presburger formula $F(\uu) = F(u_1,
\dots, u_d)$.
An atom of~$\hat A$ is precisely an element that is not the sum of two
nonidentity elements of~$\hat A$.  The set of atoms of~$\hat A$ is
thus Presburger via the formula
\begin{align}
\tag{$*$}\label{eq:*}
H(\uu)
=
\neg(\exists \vv, \rr \in \ZZ^d) \big(\vv \neq \0 \wedge \rr \neq \0
\wedge F(\vv) \wedge F(\rr) \wedge \uu = \vv + \rr\big).
\end{align}
By Theorem~\ref{t:presburger} again, the set of atoms in~$A$
itself is semisimple.

Conversely, let $S \subset \ZZ^d$ be the set of atoms of $\hat A$, and
write $S = \bigcup_{i=1}^r (\qq_i + N_i)$ with $\qq_i\in \ZZ^d$ and
$N_i \subseteq \ZZ^d$ affine semigroups as in
Theorem~\ref{t:presburger}.  Let $F_i$ be a Presburger formula that
defines membership in $N_i$.  Thus $\hat A$ is Presburger definable by
$$
F(\uu)
=
(\exists n_1, \dots, n_r \in \NN, \exists \aa_1, \dots, \aa_r \in \ZZ^d)
\Big(F_1(\aa_1) \wedge \cdots \wedge F_r(\aa_r)
\wedge \uu = \sum_{i=1}^r (n_i\qq_i + \aa_i)
\Big).
$$
By Theorem~\ref{t:presburger} we conclude that $\hat A$, and then $A$,
is semisimple.

The proof of the upset claim is similar, using formulas for semisimple
subsets of the image $\hat U$ of~$U$ in~$\ZZ^d$ where needed, such as
instead of~$F(\rr)$ in~\eqref{eq:*} or instead of some of the~$F_i$ in
the subsequent display.
\end{proof}


\begin{example}\label{e:non-fg}
A submonoid $A \subseteq \NN^d$ can be infinitely generated but still
semisimple.\vspace{-.8ex}
\end{example}
\begin{wrapfigure}{L}{0.175\textwidth}
\vspace{-2.5ex}
\quad
  \includegraphics[height=8ex]{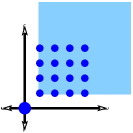}
\vspace{-5ex}
\end{wrapfigure}
\noindent
For instance, $\NN^2$ with its axes removed and its origin put back in
is not finitely generated but is semisimple.
\vspace{2ex}

\begin{example}\label{e:non-fg'}
Failure to be semisimple arises when the genesis of infinite
generation is irrational, such as the submonoid of $\NN^2$ on or above
the $x$-axis and below the line~$y =\nolinebreak \sqrt2 x$, or
nonlinear, such as the submonoid of $\NN^2$ on or above the $x$-axis
and below the parabola~$y\hspace{-.4pt} =\nolinebreak x^2$.
\end{example}

\section{Numerics of constructibility}\label{s:numerics}

This section lays out what it means for a family of vector spaces to
behave, when counting dimensions, in a semisimple (equivalently, by
Theorem~\ref{t:presburger}, Presburger definable) manner.  The notion
of \emph{numerically constructible} family
(Definition~\ref{d:numerical-family}) is only a shadow of the
module-theoretic constructibility in \S\ref{s:families}, but it
encapsulates the data that extract quasipolynomial behavior
(Theorem~\ref{t:num-constructible=>quasipol}) from constructibility.

\subsection{Numerical constructibility}\label{b:numerical}

\begin{defn}\label{d:semisimple-partition}
A \emph{semisimple subdivision} of a semisimple set $S$
(Definition~\ref{d:semisimple} and Theorem~\ref{t:presburger})
in a free abelian group~$Q$ is an expression of~$S$ as a finite
disjoint union
$$
  S = \Bigcupdot_{\alpha \in A} S_\alpha
$$
in which each block~$S_\alpha$ of the partition is semisimple in~$Q$.
\end{defn}

\begin{defn}\label{d:numerically-subordinate}
Let $V = \bigoplus_{\ggg \in G} V_\ggg$ be a direct sum of $\kk$-vector spaces
indexed by a free abelian group~$G$. 
A semisimple subdivision of~$G$ is \emph{numerically subordinate}
to~$V$ if $\ggg \mapsto \dim_\kk V_\ggg$ is constant on the regions of
the subdivision.
\end{defn}

\begin{defn}\label{d:numerical-family}
Let $Z$ be a free abelian group and $\{V_\nn\}_{\nn \in Z}$ be a
family of $\kk$-vector spaces each of which is graded by~$Q$, so
$V_\nn = \bigoplus_{\qq \in Q} V_{\nn\qq}$, where $V_{\nn\qq} =
(V_\nn)_\qq$.  The family $\{V_\nn\}_{\nn \in Z}$ over~$Q$ is
\emph{numerically constructible} if $\dim_\kk V_{\nn\qq} < \infty$ for
all~$\nn\qq \in\nolinebreak G$~and $G =\nolinebreak Z \times Q$ admits a
semisimple subdivision numerically subordinate to $V = \bigoplus_{\nn\qq
\in G} \hspace{-1pt}V_{\nn\qq}$.
\end{defn}

\begin{remark}\label{r:other-class-X}
To generalize this theory to another class~$\XX$ beyond semisimple,
Definition~\ref{d:numerical-family} would need to explicitly require
each slice with~$\nn$ fixed to be class~$\XX$.  In contrast, the
intersection of a semisimple set with $\{\nn\} \times Q$ is automatically
semisimple.
\end{remark}

\begin{remark}\label{r:zero-part-of-family}
Often the vector spaces $V_\nn$ are $Q$-modules that are nonzero only
for $\nn$ lying in a submonoid of~$Z$ that is semisimple as a subset.
For example, $Z$ could be~$\ZZ^k$ with the vector spaces~$V_\nn$ only
nonzero for $\nn \in \NN^k$.  Graded families of ideals (see
Definition~\ref{d:graded-family-of-ideals}), for instance, have $Z =
\ZZ$ and are only nonzero for $n \in \NN \subseteq \ZZ$.
\end{remark}

\begin{example}\label{e:numerical}
Fix $Q = \ZZ^2$ with positive cone $Q_+ = \NN^2$ and $G = \ZZ \times
Q$.  Let $I \subset \kk[x,y]$ be a monomial ideal and set
$$
M_n
=
\begin{cases}
  I \oplus \kk[x,y]/I &\text{if } n \text{ is even}
  \\
  \quad\ \ \,\kk[x,y] &\text{if } n \text{ is odd.}
\end{cases}
$$
In addition, set $M_n = 0$ for $n < 0$.  The family
$\{M_n\}_{n \in \ZZ}$ is numerically constructible because it is
subordinate to the subdivision of~$G$ with two constant regions:
\begin{itemize}
\item%
the nonnegative octant $\NN \times Q_+ \cong \NN^3$, where the ``Hilbert
functions'' $\ggg \mapsto \dim_\kk V_\ggg$ always take the value~$1$; and

\item%
elsewhere, where the functions $\ggg \mapsto \dim_\kk V_\ggg$ always take the
value~$0$.
\end{itemize}
\end{example}


\subsection{Quasipolynomials from numerical constructibility}\label{b:num-const}\mbox{}

\medskip
\noindent
The following definition is standard; see \cite[Definitions~1.8
and~1.9]{woods2015}, for example.

\begin{defn}\label{d:quasipolynomial}
A function $Q: Z \to \QQ$ on a free abelian group~$Z$ of rank~$k$ is a
\emph{quasipolynomial} if there exist a rank~$k$ sublattice $\Lambda \subseteq Z$
and polynomials $P_{\ol \mmm} \in \QQ[x_1, \ldots, x_k]$, one for each coset $\ol
\mmm \in Z/\Lambda$, such that $Q(\nn) = P_{\ol \mmm}(\nn)$ for every $\nn \in \ol \mmm$.
\end{defn}

\begin{defn}\label{d:piecewise-quasipolynomial}
Fix a subset $T \subseteq Z$ in a free abelian group~$Z$.  A function $g: T \to
\QQ$ is \emph{piecewise quasipolynomial} if there exist
\begin{itemize}
\item%
a finite set~$A$,

\item%
a partition $Z = \bigcup_{\alpha \in A} (\Gamma_\alpha \cap Z)$
with each $\Gamma_\alpha$ a rational convex polyhedron,~and
\item%
quasipolynomials $Q_\alpha: Z \to \QQ$ for $\alpha \in A$
\end{itemize}
such that $g(\nn) = Q_\alpha(\nn)$ for every $\nn \in \Gamma_\alpha \cap T$.  If the polynomials
defining the quasi\-polynomials $Q_\alpha$ are all linear, then $Q$ is
\emph{piecewise quasilinear}.
\end{defn}

\begin{remark}\label{r:quasi-pol-k=1}
In Definition~\ref{d:piecewise-quasipolynomial}, if $Z$ has rank~$1$,
then $g$ is piecewise quasipolynomial $\iff g(n)$ coincides with a
quasipolynomial for $n \gg 0$ and a quasipolynomial for $n \ll 0$.
\end{remark}

\begin{defn}\label{d:length}
The \emph{length} of a direct sum of $\kk$-vector spaces $V = \bigoplus_{\qq \in Q}
V_\qq$ is $\length(V) = \sum_{\qq \in Q} \dim_\kk V_\qq$.
\end{defn}

\begin{thm}\label{t:num-constructible=>quasipol}
Fix a numerically constructible family $\{V_\nn\}_{\nn \in Z}$ 
over~$Q \cong \ZZ^d$, indexed by a free abelian group~$Z$.
\begin{enumerate}
\item\label{i:length-quasipol}%
The set $T = \{\nn \in Z \mid \length(V_\nn) < \infty\}$ is semisimple and the
function $\length_V: T \to \NN$ given by $\length_V(\nn) = \length(V_\nn)$ is
piecewise quasipolynomial of degree at most~$d$.
\item\label{i:maxmin-quasipol}%
For any linear function $\lambda: Q \to \ZZ$ the sets 
\begin{align*}
  \Top_\lambda V
  &=
  \bigl\{\nn \in Z \mid V_\nn \neq 0 \text{ and } \sup\{\<\lambda,\qq\> \mid V_{\nn\qq} \neq 0\} < \infty \bigr\}
\\\text{and}\quad
  \soc_\lambda V
  &=
  \bigl\{\nn \in Z \mid V_\nn \neq 0 \text{ and } \inf\{\<\lambda,\qq\> \mid V_{\nn\qq} \neq 0\} > -\infty \bigr\}
\end{align*}
are semisimple, and the functions
$$
\begin{array}{r@{\ }l@{\quad}r@{\ }l@{}}
  \max_{\lambda,V}: \Top_\lambda V &\to \ZZ
  &\text{and}\quad
  \min_{\lambda,V}: \soc_\lambda V &\to \ZZ
\\
                     \nn &\mapsto \dis\max_{V_{\nn\qq} \neq 0} \<\lambda,\qq\>
  &
                     \nn &\mapsto \dis\min_{V_{\nn\qq} \neq 0} \<\lambda,\qq\>
\end{array}
$$
are piecewise quasilinear.
\end{enumerate}
\end{thm}
\begin{proof}
Let $G = Z \times Q$ and $G = \bigcup_{\alpha \in A} S_\alpha$ be a (finite) semisimple
subdivision numerically subordinate to $V = \bigoplus_{\nn\qq \in G} M_{\nn\qq}$.  For
each $\alpha \in A$ let $v_\alpha$ be the common value of $\dim_\kk V_g$ for $g \in
S_\alpha$.  Let $A_0 \subseteq A$ be the set of indices $\alpha$ such that $v_\alpha\neq 0$.
For the purpose of writing Presburger formulas in~$Q$, fix an
isomorphism $Q \cong \ZZ^d$.

\ref{i:length-quasipol}.  By Theorem~\ref{t:presburger} the sets $S_\alpha$
are defined by Presburger formulas~$F_\alpha(\nn,\qq)$.  The set $T$ is
semisimple because it is defined by the formula (see
Remark~\ref{r:implication} regarding~``$\implies$'')
$$
  F(\nn)
  = (\exists N \in \NN)
    \bigl(\bigvee_{\alpha \in A_0} F_\alpha(\nn,\qq) \implies \|\qq\|_\infty < N\bigr).
$$
For $\nn \in T$,
$$
  \length_V(\nn)
  = \sum_{\alpha \in A_0} \sum_{\qq \in S_\alpha \cap Q_\nn} \dim_\kk V_{\nn\qq}
  = \sum_{\alpha \in A_0} \#
      \bigl\{\qq \in \ZZ^d \mid F(\nn) \wedge F_\alpha(\nn,\qq)\bigr\} v_\alpha,
$$
where $Q_\nn = \{\nn\} \times Q$.  For piecewise quasipolynomiality, use
\cite[Theorem~1.10]{woods2015}.  The degree bound follows from the
proof of the latter theorem as there are $d$ ``counted variables'' in
our counting functions, namely the coordinates of~$\qq$ (see also
\cite{sturmfels1995}).

\ref{i:maxmin-quasipol}.  Fix $\lambda: Q \to \ZZ$.  The conclusions for
$\max_{\lambda,V}$ and $\min_{\lambda,V}$ are similar, so we only write out the
case of~$\max_{\lambda,V}$.  Consider the Presburger formula
$$
  H(\nn,w)
  = (\exists \qq\in \ZZ^d)
    \Bigl(
      \bigl(\bigvee_{\alpha \in A_0}\!\! F_\alpha(\nn,\qq)\bigr)
      \wedge
      \bigl(w = \<\lambda,\qq\>\bigr)
    \Bigr).
$$
The set $\Top_\lambda V$ is defined by the Presburger formula
$$
  J(\nn)
  =
  (\exists w \in \ZZ) H(\nn,w) \wedge (\exists N \in \NN,\,\forall w \in \ZZ) \bigl(H(\nn,w) \implies w < N\bigr).
$$
Consider the ``upper'' formula
$$
  U(\nn,n)
  =
  (\exists w \in \ZZ) \bigl(J(\nn) \wedge H(\nn,w) \wedge (1 \ls n \ls w)\bigr)
$$
and the ``lower'' formula
$$
  L(\nn,n)
  =
  (\forall w \in \ZZ)\bigl(w > 0 \implies \neg H(\nn,w)\bigr)
  \wedge
  (\exists w \in \ZZ)\bigl(J(\nn) \wedge H(\nn,w) \wedge (w \ls n \ls 0)\bigr).
$$
For $\nn \in \Top_\lambda V$, if $\max_{V_{\nn\qq} \neq 0} \<\lambda,\qq\>$ is positive, then the
expression $\#\{n \mid U(\nn,n)\}$ evaluates to $\max_{V_{\nn\qq} \neq 0} \<\lambda,\qq\>$,
and it evaluates to zero otherwise.  On the other hand, $\#\{n \mid
L(\nn,n)\}$ evaluates to zero if $\max_{V_{\nn\qq} \neq 0} \<\lambda,\qq\>$ is positive,
and it evaluates to $|\max_{V_{\nn\qq} \neq 0} \<\lambda,\qq\>| + 1$ otherwise.  Thus,
for $\nn \in P$,
$$
  {\textstyle\max_{\lambda,V}(\nn)}
  =
  \#\{n \mid U(\nn,n)\} - \#\{n \mid L(\nn,n)\} + 1.
$$
Hence $\max_{\lambda,V}$ is piecewise quasipolynomial by
\cite[Theorem~1.10]{woods2015}; the quasipolynomials involved are in
fact quasilinear as $n$ is the sole counted variable.
\end{proof}

\begin{remark}\label{r:truncate}
Following the notation in Theorem~\ref{t:num-constructible=>quasipol}, let
$S \subset G = Z \times Q$ be a semisimple set.  The truncated family
$\{V_\nn|_S\}_{\nn \in Z}$, where $V_\nn|_S = \bigoplus_{\nn\qq \in S} V_{\nn\qq}$, is also
numerically constructible.  The latter is a \emph{linear truncation}
of the family $\{V_\nn\}_{\nn \in Z}$.  Linearly truncating a given
numerically constructible family using a suitable semisimple set can
enhance the support of the functions $\length_V$, $\min_{\lambda,V}$, and
$\max_{\lambda,V}$ in Theorem~\ref{t:num-constructible=>quasipol}.
\end{remark}

\section{Constructible modules}\label{s:constructible-mods}

Quasipolynomial behavior arises, in our view, because algebraic
objects come in families that are parametrized in a semisimple manner,
or more precisely, parametrized by Presburger groups
(Definition~\ref{d:presburger-group}).  Although our interest is in
objects defined by Presburger arithmetic, the proofs are often more
transparent when phrased generally in terms of an arbitray
``class~$\XX$'' (Definition~\ref{d:class-X}), analogous to an
$o$-minimal structure.  These definitions give rise to notions of
constructible module (Definitions~\ref{d:tame}
and~\ref{d:constructible}).

This section builds foundations that focus on Presburger groups as
parametrizing vector spaces.  Later, particularly in
\S\ref{s:families}, additional layers of theory allow the
Presburger group to be interpreted as parametrizing a family of
multigraded modules instead of merely a family of vector spaces.

\begin{defn}\label{d:pogroup}
A \emph{partially ordered abelian group} is an abelian group~$Q$ with
a submonoid~$Q_+$, the \emph{positive cone}, having trivial unit
group.  The partial order is given by: $\qq \preceq \qq' \iff \qq' - \qq
\in Q_+$.  If $Q_+$
generates a subgroup of finite index in~$Q$ then $Q$ is \emph{full}.
\end{defn}

\begin{defn}\label{d:presburger-group}
A \emph{Presburger group} is a full partially ordered free abelian
group $Q$ of finite rank whose positive cone~$Q_+$ is a semisimple
subset of~$Q$.
\end{defn}

\begin{example}\label{e:affine-semigroup}
If $Q_+$ is an affine semigroup, then $Q$ is a Presburger group by
Theorem~\ref{t:presburger}; see also the more general
Proposition~\ref{p:atoms-definable}.
\end{example}

\begin{remark}\label{r:full}
The ``full'' hypothesis in Definition~\ref{d:pogroup} arises in the
process of reducing to the case of freely parametrized families
(Proposition~\ref{p:free-reduction}).
\end{remark}

For many of the statements and proofs, the salient properties of the
semisimple hypothesis in Definition~\ref{d:presburger-group} hold in
much greater generality.

\begin{defn}\label{d:class-X}
A full partially ordered abelian group~$Q$ is \emph{class~$\XX$} if it
is a subgroup of a real vector space of finite dimension and $Q_+$
belongs to a family~$\XX$ of subsets of~$Q$ closed under complements,
finite unions, negations, and Minkowski sums with~$Q_+$.
\end{defn}

\begin{prop}\label{p:valid-class-X}
The semisimple sets form a valid family $\XX$ to define a class~$\XX$
group in Definition~\ref{d:class-X}.
\end{prop}
\begin{proof}
Complements, finite unions, and negatives of Presburger definable sets
are Presburger definable.  Now, fix an embedding $Q \into \ZZ^d$.  Let $S\subset
Q$ be Presburger definable by the formula~$F(\uu)$, and let $G(\vv)$ be a
formula that defines~$Q_+$.  The Minkowski sum $S+Q_+$ is defined by
$
  H(\rr)
  =
  (\exists \uu, \vv \in \ZZ^d)
  \big(F(\uu)\wedge G(\vv)\wedge \rr=\uu + \vv\big).
$
Since Presburger definable $\iff$ semisimple by
Theorem~\ref{t:presburger}, the proof is done.
\end{proof}

\begin{defn}\label{d:Q-mod}
Fix an arbitrary poset~$Q$.  A~\emph{$Q$-module} is a $Q$-graded
$\kk$-vector space $M = \bigoplus_{\qq \in Q} M_\qq$ with a $\kk$-linear map $M_\qq \to
M_{\qq'}$ for every pair $\qq \preceq \qq'$ in $Q$ such that
\begin{itemize}
\item%
$M_\qq \to M_\qq$ is the identity and
\item%
$M_\qq\to M_{\qq''}$ is equal to the composite $M_\qq\to M_{\qq'}\to
M_{\qq''}$ if $\qq\preceq \qq'\preceq \qq''$.
\end{itemize}
\end{defn}

Definition~\ref{d:Q-mod} is equivalent to requiring that $M$
constitute a functor from the partially ordered set~$Q$, viewed as a
small category, to the category of vector spaces over~$\kk$.

\begin{defn}\label{d:constant-subdivision}
Let $M$ be a $Q$-module over a poset~$Q$.  A \emph{constant
subdivision of $Q$ subordinate to~$M$} is a partition of $Q$ such that
for each region $I$ in the partition there exists a vector space $M_I$
and isomorphisms $M_I\to M_\ii$ for every $\ii\in I$, which has \emph{no
monodromy}, i.e., for any pair of regions $I,J$ and any $\ii\in I$, $\jj\in
J$, if $\ii\preceq \jj$ then the composition $M_I\to M_\ii\to M_\jj\to M_J$ only depend on
$I$ and $J$.  In this case we also say that $M$ \emph{dominates} the
given partition of $Q$.
\end{defn}

\begin{defn}[{\cite[Definition~2.20 and
Remark~2.21]{hom-alg-poset-mods}}]\label{d:tame}
Fix a poset~$Q$.
\begin{enumerate}
\item%
A~$Q$-module~$M$ is \emph{tame} if $\dim_\kk M_\qq$ is finite for every $\qq
\in Q$ and $Q$ has a constant subdivision with finitely many regions
that is subordinate to~$M$.
\item%
If $Q$ is a class~$\XX$ group, then $M$ is \emph{of class~$\XX$} if it
is tame via a constant subdivision whose regions lie in the
family~$\XX$.
\end{enumerate}
\end{defn}

\begin{remark}\label{r:tame-vs-class-X}
The notion of tame coincides with class~$\XX$ for the family $\XX$ of
all subsets of~$Q$, so there is no need to separate claims about
tameness from those about class~$\XX$.
\end{remark}

\begin{defn}\label{d:indicator-module}
Fix a poset~$Q$.  An \emph{upset}
$U \subseteq Q$ is a subset closed under going upward in~$Q$ (so $\qq \in U$ and
$\qq \preceq \qq' \implies \qq' \in U$).  Dually, a \emph{downset} $D \subseteq Q$ is a subset
closed under going downward in~$Q$.  If $S$ is the intersection of an
upset and a downset (one of which might be all of~$Q$), then $S$ is
\emph{poset-convex}, and $\kk\{S\}$ denotes the \emph{indicator module}
of~$S$, namely the natural $Q$-module with basis $\{\xx^\qq \mid \qq \in S\}$.
When $U$ is an upset, $\kk[U]$ is an \emph{upset module}; when $D$ is a
downset, $\kk[D]$ is a \emph{downset module}.
\end{defn}

\begin{remark}\label{r:syzygy-thm}
The syzygy theorem for poset modules
\cite[Theorem~6.12]{hom-alg-poset-mods} states many equivalent
characterizations of class~$\XX$ modules, one of which is that they
admit finite presentations and resolutions by class~$\XX$ upset or
downset modules.  The generality of the term ``class~$\XX$'' in
\cite{hom-alg-poset-mods} assumes that the poset~$Q$ is a subposet of
a ``real partially ordered group'' \cite[Definitions~2.19
and~2.20]{hom-alg-poset-mods}.  It is an unfortunate oversight that
\cite[Definition~2.19]{hom-alg-poset-mods} assumes the positive cone
$Q_+$ of a real partially ordered group generates~$Q$---unfortunate
because that hypothesis is never used in \cite{hom-alg-poset-mods};
all that matters is containment in a real vector space (for conditions
like ``semialgebraic'' or ``subanalytic'') and existence of an ambient
partial order defined by a class~$\XX$ positive cone.  This is
relatively straightforward to verify, because
\cite{hom-alg-poset-mods} is written for arbitrary posets, so one only
needs to search for the phrase ``class~$\XX$''
to see that all occurrences (there are less than two dozen) have no
use for the ``generated by~$Q_+$'' hypothesis.  This slight change in
generality is relevant to the ``full'' hypothesis in
Definitions~\ref{d:pogroup} and~\ref{d:presburger-group} because, when
$Q_+$ generates a proper subgroup $\ZZ Q_+ \subsetneq Q$, the partial order
on~$Q$ is not induced by any inclusion of~$Q$ into a real partially
ordered group, as per \cite[Definition~2.19]{hom-alg-poset-mods},
which disallows \mbox{incomparable}~distinct~cosets~of~$\ZZ Q_+$.
\end{remark}

\begin{remark}\label{r:constructible}
To say a module is of class~$\XX$ means that the module is
``constructible'', which for example specializes to semialgebraically
constructible or subanalytically constructible when the building
blocks are subsets of real vector spaces that are semialgebraic or
subanalytic.  In the semisimple setting, the building blocks are
instead translates of affine semigroups built from a single given
lattice~$Q \cong \ZZ^d$ (Definition~\ref{d:semisimple}).
\end{remark}

In analogy with Remark~\ref{r:constructible}, the class of semisimple
subsets gives rise to a notion of constructible modules.

\begin{defn}\label{d:constructible}
A module~$M$ over a Presburger group~$Q$
(Definition~\ref{d:presburger-group}) is \emph{semisimply
constructible} if it is of class~$\XX$ for $\XX = \text{semisimple}$
in Definition~\ref{d:tame}.
In this paper, the term \emph{constructible} means semisimply
constructible unless otherwise stated.
\end{defn}

\begin{remark}\label{r:semisimple-not-rep-theor}
The notion of semisimple here is different than the usual
module-theoretic definition of ``semisimple'' meaning ``decomposes as
a direct sum of indecomposables''.  Arbitrary $Q$-modules in
Definition~\ref{d:Q-mod} are semisimple in this decomposition sense by
the Krull--Remak--Schmidt--Azumaya theorem
\cite[Theorem~1.1]{botnan-crawley-boevey2020}.
\end{remark}

\section{Flat complexes of modules}\label{s:flat}

Prior sections set up objects and morphisms with appropriate
finiteness conditions (semisimple, Presburger, constructible).  This
means our \emph{input} modules have appropriately structured (i.e.,
quasipolynomial, see \S\ref{s:functors-on-families}) numerics.
What we need now is functors that preserve the finiteness so the
\emph{output} modules have similarly structured numerics.  This
section initiates the process with general results concerning how to
construct flat complexes of class~$\XX$
(Corollary~\ref{c:X-flat-complex}) and why such complexes remain of
class~$\XX$ after tensoring with a module of class~$\XX$
(Theorem~\ref{t:X-flat}).

\subsection{Homomorphisms and complexes}\label{b:complexes}\mbox{}

\medskip
\noindent
Since the goal is to prove that various functors preserve
constructibility, finiteness conditions are needed not merely for
objects, but for morphisms and complexes, as~well.

\begin{defn}\label{d:morphism}
Fix a poset $Q$.  A \emph{homomorphism} of $Q$-modules $\phi: M \to N$ is a
collection of $\kk$-linear maps $M_\qq \to N_\qq$ for $\qq \in Q$ making the
diagram
$$
\begin{array}{c@{\ }c@{\ }c}
     M_{\qq}    & \too &    N_{\qq'}
\\[.5ex]
  \big\downarrow&      & \big\downarrow
\\[.5ex]
    M_{\qq'}    & \too &    N_{\qq'}
\end{array}
$$
commute for every pair $\qq \preceq \qq'$.
\end{defn}

\begin{defn}[{\cite[Definition\,4.29]{hom-alg-poset-mods}}]\label{d:morphism-tame}
Fix a poset~$Q$.  A homomorphism $\phi: M \!\to N$ of $Q$-modules is
\emph{tame} if $Q$ admits a finite constant subdivision subordinate to
both $M$ and~$N$ such that for each region $I$ in the partition the
composite $M_I \to M_\ii \to N_\ii \to N_I$ does not depend on $\ii \in I$.  The
constant subdivision is also said to be \emph{subordinate} to~$\phi$.
If~$Q$ and the constant subdivision are class~$\XX$ then $\phi$ is
\emph{of class $\XX$}.  If $Q$ is a Presburger group and the constant
subdivision is semisimple, then $\phi$ is \emph{constructible}.
\end{defn}

The notion of upset presentation in the upcoming
Definition~\ref{d:upset-presentation} is needed later, for the proof
of Proposition~\ref{p:multiply-ideal-module}.  After that, for the
record we state in Lemma~\ref{l:upset-criterion-class-X} a criterion
to detect when a morphism of modules is of class~$\XX$; it patches an
oversight in \cite[Theorem~6.12]{hom-alg-poset-mods}, which only
stated one direction.

\begin{defn}[{\cite[Definition~6.4.1]{hom-alg-poset-mods}}]\label{d:upset-presentation}
Fix a class~$\XX$ group~$Q$.  An \emph{upset presentation} of a
$Q$-module~$M$ is an expression of~$M$ as the cokernel of a
homomorphism $F_1 \to F_0$ such that each $F_i$ is a direct sum of
finitely many upset modules.
\begin{enumerate}
\item%
An upset presentation is \emph{of class~$\XX$} if all of the upsets
appearing in it are class~$\XX$.
\item%
A \emph{morphism} of upset presentations is morphism of complexes each
of which is an upset presentation.
\end{enumerate}
\end{defn}

\begin{remark}\label{r:upset-presentation}
Over an arbitrary poset, it is necessary to require every component
of the homomorphism $F_1 \to F_0$ to be connected in the sense of
\cite[Definition~3.14]{hom-alg-poset-mods}.  But partially ordered
groups are upper-connected
\cite[Definition~3.5.2]{hom-alg-poset-mods}, so all homo\-morphisms of
upset modules are automatically connected
\cite[Corollary~3.11.2]{hom-alg-poset-mods}.
\end{remark}

\begin{lemma}\label{l:upset-criterion-class-X}
A homomorphism $\phi: M \to N$ of $Q$-modules over a class~$\XX$ group~$Q$
is of class~$\XX$ if and only if it is induced by a morphism
of upset presentations.
\end{lemma}
\begin{proof}
Suppose $\phi$ is induced by a morphism of upset presentations.  Let
$\kk[U_1], \ldots, \kk[U_r]$ be the list of all summands of the four modules
involved.  Each upset $U$ subdivides~$Q$ into two regions, namely $U$
and $Q \minus U$.  The common refinement of these subdivisions is
subordinate to~$\phi$ by the connectedness in
Remark~\ref{r:upset-presentation}.  That proves the ``if'' direction.
The ``only if'' direction is part of the syzygy theorem
\cite[Theorem~6.12]{hom-alg-poset-mods}.
\end{proof}

\begin{defn}[{\cite[Definition~6.17]{hom-alg-poset-mods}}]\label{d:complex-class-X}
Fix a complex $C^\spot$ of $Q$-modules.
\begin{enumerate}
\item%
The complex $C^\spot$ is \emph{of class~$\XX$} if its modules and
morphisms are of class~$\XX$.

\item%
A constant subdivision is \emph{subordinate} to $C^\spot$ if it is
subordinate to all of the modules and morphisms therein, and then
$C^\spot$ is said to \emph{dominate} the subdivision.
\end{enumerate}
\end{defn}

\subsection{Localization}\label{b:localization}\mbox{}

\medskip
\noindent
An initial premise is that any tensor product in which one of the
modules is class~$\XX$ and the other is noetherian should result in a
class~$\XX$ module.  However, just a bit more work shows it is enough
to demand that one of the modules has a presentation by modules that
are finite direct sums of localizations of the monoid
algebra~$\kk[Q_+]$, as long as those localizations are all of
class~$\XX$.  (The noetherian case occurs when $\kk[Q_+]$ is noetherian
and the localizations are free by virtue of not inverting any nonunit
elements.)  The next definition introduces this type of module.  The
relaxation from finitely freely presented to flatly presented is
useful because, for example, local cohomology can be computed using
complexes of flat modules.

Recall Convention~\ref{conv:conventions} concerning the group~$Q$ and
field~$\kk$.  In particular, $Q$ need not be Presburger unless otherwise
stated.

\begin{defn}\label{d:localization}
Let $S \subseteq Q_+$ be a set that is closed under addition.  The
\emph{localization} of a $Q$-module~$M$ \emph{along~$S$} is the tensor
product
$$
M_S
=
M[\xx^{-\qq} \mid \qq \in S]
=
M \otimes_{\kk[Q_+]} \kk[\xx^{-\qq} \mid \qq \in S].
$$
\end{defn}

\begin{remark}\label{r:localization}
A $Q$-module~$M$ is equivalently a $Q$-graded module over the monoid
algebra $\kk[Q_+]$.  As such, the localization $M[m^{-1}]$ of~$M$ by
inverting an explicit monomial $m = \xx^\qq \in \kk[Q_+]$ is the usual notion
from commutative algebra.
\end{remark}

\begin{defn}\label{d:X-flat}
A $Q$-module $M$ is \emph{$\XX$-flat} if it is a finite direct sum of
class~$\XX$ modules that are translates of localizations of~$\kk[Q_+]$.
The module $M$ is \emph{semisimply flat} if $Q$ is a Presburger group
and $M$ is $\XX$-flat with $\XX = \text{semisimple}$.  (The meaning of
``semisimply'' here refers to Definition~\ref{d:semisimple}, but in
this special context it does not conflict with the direct sum sense in
Remark~\ref{r:semisimple-not-rep-theor}.)
\end{defn}

\begin{defn}\label{d:summable}
A class $\XX$ set $S \subseteq Q$ is \emph{Minkowski summable} if for
any other class $\XX$ set $S' \subseteq Q$ the Minkowski sum $S + S'$
is class $\XX$.
\end{defn}

\begin{remark}\label{e:summable}
For most of the commonly selected classes~$\XX$, including
semialgebraic, PL, or semisimple (that is, Presburger definable; see
Theorem~\ref{t:presburger}), every class~$\XX$ set is Minkowski
summable.  The subanalytic class is an exception, although the
Minkowski sum of a \emph{bounded} subanalytic set with another
subanalytic set is subanalytic.  As a concrete class that contains
sets that are not Minkowski summable, consider the class~$\XX(2)$
consisting of semialgebraic regions defined by polynomials that are
quadratic or linear.  As long as the positive cone~$Q_+$ is
polyhedral, this class~$\XX(2)$ is closed under complement, finite
intersection, negation, and Minkowski sum with the positive
cone~$Q_+$.  However, the Minkowski sum of two quadratic upsets might
require quartic (i.e., degree~$4$) polynomial inequalities.
\end{remark}

\begin{lemma}\label{l:summable}
Suppose that the positive cone~$Q_+$ is of class~$\XX$.  Localizing
any class~$\XX$ module or morphism along a Minkowski summable subset
of~$Q_+$ yields a module or morphism of class~$\XX$.
\end{lemma}
\begin{proof}
Localizing $\kk\{U\}$ for an upset~$U$ along~$S$ yields $\kk\{U - S\}
= \kk\{-(-U + S)\}$, which only requires negation and adding~$S$
to~$-U$.  Any class~$\XX$ module~$M$ admits a class~$\XX$ upset
presentation by \cite[Theorem~6.12]{hom-alg-poset-mods} (see
Remark~\ref{r:syzygy-thm}).  Localizing any such presentation yields
an upset presentation of the localization of~$M$ by exactness of
localization.  If $M \to N$ is a class~$\XX$ morphism, then it lifts
to a class~$\XX$ upset presentation by the syzygy theorem
\cite[Theorem~6.12]{hom-alg-poset-mods}, whose localization along~$S$
is an upset presentation of the localized morphism~$M_S \to N_S$,
which is class~$\XX$ by Definition~\ref{d:summable} if $S$~is
Minkowski summable.
\end{proof}

\subsection{\texorpdfstring{$\XX$}{X}-flat complexes}\label{b:X-flat}%

\begin{prop}\label{p:Cech-tame}
Suppose that the positive cone~$Q_+$ is of class~$\XX$.  If $M$ is of
class~$\XX$ then any fixed scalar times the localization
morphism $M \to M_S$ along a class~$\XX$ Minkowski summable subset~$S$
of~$Q_+$ is of class~$\XX$.
\end{prop}
\begin{proof}
Since $S$ is Minkowski summable, the localized upset $Q_+ - S = -(-Q_+
+ S)$ is of class~$\XX$.  Tensoring an upset presentation of~$M$ with
the localization morphism $\kk[Q_+] \to \kk[Q_+ - S]$ yields a
homomorphism of upset presentations which presents the morphism $M \to
M_S$.
\end{proof}

\begin{cor}\label{c:X-flat-complex}
Any complex in which the modules are $\XX$-flat is a complex of
class~$\XX$.  Any such complex is called an \emph{$\XX$-flat complex}.
\end{cor}
\begin{proof}
The content of the statement is that the morphisms are of class~$\XX$
automatically if the objects are $\XX$-flat.  By
\cite[Corollary~3.11.2]{hom-alg-poset-mods}, all morphisms $\kk\{U\}
\to \kk\{U'\}$ between indecomposable upset modules are inclusion
followed by global scalar multiplication.  If $\kk\{U\} \to \kk\{U'\}$
is a nonzero morphism of localizations, with $U = Q_+ - S$ and $U' =
Q_+ - S'$, then $U \subseteq U'$, by
\cite[Corollary~3.11.2]{hom-alg-poset-mods}.  Thus $\kk\{U\} \to
\kk\{U'\}$ is a natural localization morphism $\kk\{U\} \to
\kk\{U\}_{S'}$ to which Proposition~\ref{p:Cech-tame} applies.
\end{proof}

\begin{thm}\label{t:X-flat}
Fix a class~$\XX$ group~$Q$, and let $C_\spot$ be an $\XX$-flat complex.
If $M$ is of class~$\XX$, then $C_\spot \otimes_Q M$ is of class~$\XX$, as is
the homology of~$C_\spot \otimes M$.
\end{thm}
\begin{proof}
By Corollary~\ref{c:X-flat-complex}, the complex~$C_\spot$ is of
class~$\XX$.  If $\kk\{U\} \to \kk\{U'\}$ is a nonzero morphism of
indecomposable upset modules, then tensoring this morphism with~$M$
yields a scalar multiple of a natural localization map from $\kk\{U\}
\otimes M \to \kk\{U'\} \otimes M$ by the argument in the proof of
Corollary~\ref{c:X-flat-complex}.  Therefore
Proposition~\ref{p:Cech-tame} applies.  The homogy is class $\XX$ by \cite[Proposition~4.30.1]{hom-alg-poset-mods}.
\end{proof}

\subsection{Faces and localizations along them}\label{b:faces}\mbox{}

\medskip
\noindent
Some further notes about localization help clarify the geometry of
$\XX$-flat modules.

\begin{defn}\label{d:face}
An \emph{ideal} in a monoid $Q_+$ is a subset $I$ closed under adding
any element of~$Q_+$, meaning that $\qq + I \subseteq I$ for all $\qq
\in Q$.  A \emph{face} of the positive cone~$Q_+$ of a partially
ordered abelian group~$Q$ is a submonoid $F \subseteq Q_+$ whose
complement $Q_+ \minus F$ is an ideal of the monoid~$Q_+$.  Sometimes
it is simpler to say that~$F$ is a \emph{face~of~$Q$}.
\end{defn}

\begin{remark}\label{r:prime-ideal}
The usual definition of prime ideal~$\pp$ in a commuative ring~$R$ via
$ab \in \pp \implies a \in \pp$ or $b \in \pp$ is equivalent to
stipulating that the complement of~$\pp$ multiplicative monoid of~$R$
is a face: $a \in R \minus \pp$ and $b \in R \minus \pp \implies ab
\in R \minus \pp$.
\end{remark}

The notions of localizing by inverting elements and localizing along
faces coincide when the positive cone~$Q_+$ is an affine semigroup.

\begin{prop}\label{p:localization}
Assume $Q_+$ is an affine semigroup.  For any $Q$-module~$M$ and
monomial $m = \xx^\qq \in \kk[Q_+]$, the localization $M[m^{-1}]$ is
naturally isomorphic to $M_F$, where $F$ is the \emph{support face}
of~$\qq$\emph{:} the smallest face of $Q_+$ containing~$\qq$.
\end{prop}
\begin{proof}
There is a natural map $M[m^{-1}] \to M_F$ whenever $m = \xx^\qq$ with $\qq \in
F$.  If $Q_+$ is an affine semigroup and $\qq$ lies interior to~$F$,
then this natural map is an isomorphism because every $\ff\in F$ precedes
some positive integer multiple of~$\qq$ in the partial order on~$Q$
\cite[Exercise~7.15]{cca}.
\end{proof}

\begin{cor}\label{c:face-X-flat}
Fix an affine semigroup $Q_+$ whose faces are class $\XX$ and
Minkowski summable.  Assume further that all translates of~$Q_+$ are
class~$\XX$.  Then every $\XX$-flat module is a finite direct sum of
$Q$-translates of localizations of\/~$\kk[Q_+]$ along faces of\/~$Q$.
\end{cor}
\begin{proof}
Immediate from Definition~\ref{d:X-flat},
Proposition~\ref{p:localization}, and the fact that $Q_+$ has only
finitely many faces.
\end{proof}

\begin{example}\label{e:missing-faces}
The monoid $A$ depicted in Example~\ref{e:non-fg} is ``missing'' two
faces, namely the $x$- and $y$-axes.  Taking $Q_+ = A$, the
$Q$-modules of any class~$\XX$ are closed under localization.  Indeed,
$Q_+$ has only two faces: the origin and $Q_+$ itself.  Localizing
along the origin is the identity operation.  And $Q_+$ itself is
always Minkowski summable by Definition~\ref{d:class-X}.  There is no
way to localize along either of the missing axes.
\end{example}

\section{Functors on class \texorpdfstring{$\XX$}{X} modules}\label{s:functors}

The stage is set to specify hypotheses guaranteeing that various
familiar homological functors preserve constructibility
(Definitions~\ref{d:constructible} and~\ref{d:morphism-tame}), or
generally any class~$\XX$.

As usual, fix and field~$\kk$ and class~$\XX$ group~$Q$ as in
Convention~\ref{conv:conventions}.

\subsection{Local cohomology}\label{b:local-cohomology}

\begin{defn}\label{d:local-cohomology}
The \emph{local cohomology} of a $Q$-module~$M$ \emph{supported on} a
monomial ideal $I = \<m_1, \dots, m_r\> \subseteq R = \kk[Q_+]$ is the cohomology
of $M \otimes \vC^\spot_I$, where
$$
  \vC^\spot_I:
  0
  \to
  R
  \to
  \bigoplus_{i=1}^n R[m_i^{-1}]
  \to
  \cdots
  \to\!\!\!\!
  \bigoplus_{i_1 < \dots < i_\ell} \!\!\!\! R[m_{i_1}^{-1} \cdots m_{i_\ell}^{-1}]
  \to
  \cdots
  \to
  R[m_1^{-1} \cdots m_r^{-1}]
  \to
  0
$$
is the \emph{\v Cech complex} of $m_1,\dots,m_r$.  Write
$\vC^\spot_I(M) = M \otimes \vC^\spot_I$.
\end{defn}

\begin{thm}\label{t:cech}
Fix a monomial ideal $I = \<m_1, \dots, m_r\>$ in the ring $R = \kk[Q_+]$,
with $m_i = \xx^{\qq_i}$ for all~$i$.  If a $Q$-module $M$ is of
class~$\XX$ and, for each~$i$, the semigroup generated by $\qq_i$ is
Minkowski summable, then the \v Cech complex~$\vC^\spot_I(M)$ is of
class~$\XX$, as is the cohomology~$H^\spot_I(M)$ of~$\vC^\spot_I(M)$.
\end{thm}
\begin{proof}
The localizations in~$\vC^\spot_I$ are of class~$\XX$ by
Lemma~\ref{l:summable}.  Hence $\vC^\spot_I$ is of class~$\XX$ by
Proposition~\ref{p:Cech-tame}.  The theorem is therefore a special
case of Theorem~\ref{t:X-flat}.
\end{proof}

\begin{remark}\label{r:always-satisfied}
The Minkowski summable hypothesis of Theorem~\ref{t:cech} is always
satisfied when $\XX = $ semialgebraic, PL, or semisimple, or when
$\XX$ is the family of all subsets~of~$Q$.
\end{remark}

\subsection{Tensor products and Tor}\label{b:tensor-products}

\begin{example}\label{e:tor-not-tame}
Localizing along a Minkowski summable set is a particular case of
tensor product.  General tensor products do not preserve class~$\XX$
or even tameness of any sort.  An example is the indicator
$\ZZ^2$-module $\kk\{L\}$ supported on the antidiagonal line $L =
\{(x,y) \mid y = -x\}$.  The tensor product of this module with itself
over $\kk[\NN^2]$ is
$$
  \kk\{L\} \otimes \kk\{L\}
  \cong
  \bigoplus_{\aa \in L} \kk\{L\},
$$
so for each~$\aa \in L$, the vector space $\bigl(\kk\{L\} \otimes
\kk\{L\}\bigr)_\aa$ in degree $\aa = (a, -a)$ is the direct sum
$\bigoplus_{i \in \ZZ} \kk\{\aa+\ii\} \otimes \kk\{-\ii\}$ of tensor
products of the $1$-dimensional vector spaces at $\aa + \ii = (a+i,
-a-i)$ and~$-\ii = (-i, i)$.

This phenomenon recurs along the antidiagonal~$L$ in the tensor
product $\kk\{U\} \otimes \kk\{U\}$, where $U = L + \NN^2$ is the
upset generated by~$L$.  Thus tensor products of individual upset
modules with individual upset modules need not be tame, let alone
class~$\XX$ for any particular~$\XX$.  Note that these examples of
$\kk\{L\}$ and $\kk\{U\}$ are constructible modules.
\end{example}

Example~\ref{e:tor-not-tame} notwithstanding, tensor products or
$\Tor$ in which one of the input modules is class~$\XX$ and the other
is noetherian still result in class~$\XX$ output.  In fact, more is
true: it is enough to demand that one of the modules has class~$\XX$
flat resolution.

\begin{defn}\label{d:tensorQ}
Henceforth $\otimes_Q$ and $\Tor_i^Q$ denote $\otimes_{\kk[Q_+]}$ and
$\Tor_i^{\kk[Q_+]}$, respectively.
\end{defn}

\begin{thm}\label{t:tor}
Fix a class~$\XX$ group~$Q$.  If $M$ is any $Q$-module of class~$\XX$
and $N$ is a $Q$-module admitting a flat resolution $C_\spot$ that is
$\XX$-flat in homological degrees $i - 1$, $i$, and $i + 1$, then
$\Tor_i^Q(M,N)$ is of class~$\XX$.
\end{thm}
\begin{proof}
$\Tor_i^Q(M,N)$ can be computed as the middle homology of the
three-term complex $M \otimes_Q C_{i-1} \from M \otimes_Q C_i \from M
\otimes_Q C_{i+1}$, which is of class~$\XX$ by Theorem~\ref{t:X-flat}.
\end{proof}

\begin{remark}\label{r:noetherian-assump}
Assume $Q_+$ affine semigroup and that all translates of~$Q_+$ are
class~$\XX$.  In Theorem~\ref{t:tor}, if $N$ is noetherian, then any
free resolution with finite rank in each homological degree is
$\XX$-flat in all homological degrees.  Assume further that every face
of $Q_+$ is class $\XX$ and Minkowski summable.  If $N$ is artinian,
then its Matlis dual $N^\vee$ (Definition~\ref{d:matlis}) is
noetherian, and so by \cite[Theorem~11.30]{cca} $N^\vee$ has an
injective resolution whose modules are Matlis duals of
\mbox{$\XX$-flat} modules (see also Remark~\ref{r:injectives}).
Taking Matlis dual of such resolution and using
Lemma~\ref{l:vee-preserves-X} proves that any artinian $N$ admits a
flat resolution $C_\spot$ that is $\XX$-flat in all homological
degrees.  Note that all these assumptions on $Q_+$ are satisfied when
$\XX = \text{semisimple}$.
\end{remark}

\subsection{Hom and Ext}\label{b:hom-and-ext}\mbox{}

\medskip
\noindent 
The results concerning flat complexes and $\Tor$ perhaps surprisingly
can be massaged to work for $\Ext$, as well, despite $\Ext$ being
asymmetric in its two arguments.  The main tool is Matlis duality,
which behaves extremely well when the relevant vector space dimensions
are all finite, as is the case when constructibility (or, more
generally, class~$\XX$) is involved.  The following review of Matlis
duality is based on \cite[\S11.3]{cca}.

\begin{defn}\label{d:matlis}
The \emph{Matlis dual} of a $Q$-module~$M$ is the $Q$-module~$M^\vee$
defined by
$$
  (M^\vee)_\qq = \Hom_\kk(M_{-\qq},\kk),
$$
so the homomorphism $(M^\vee)_\qq \to (M^\vee)_{\qq'}$ is transpose to
$M_{-\qq'} \to M_{-\qq}$.
\end{defn}

The degree-by-degree formula for~$M^\vee$ becomes more transparent
when all degrees are considered simultaneously via the following
standard notion of graded~$\Hom$.

\begin{defn}\label{d:hhom}
Fix a $Q$-graded ring~$R$, such as~$\kk$ (concentrated in degree~$\0$)
or~$\kk[Q_+]$.  For any $Q$-graded $R$-modules~$M$ and~$N$, set
$$
\hhom_R(M,N)_\qq
=
\hbox{ degree $\qq$ homogeneous $R$-homomorphisms }M \to N
$$
in which
$M_\aa \to N_{\aa + \qq}$ for all~$\aa$.  The \emph{graded~Hom} is
$$
\hhom_R(M,N) = \bigoplus_{\qq \in Q} \hhom_R(M,N)_\qq
$$
defined as a $Q$-graded $\kk[Q_+]$-module by $(\xx^\qq \phi)(m) = \phi(\xx^\qq m)$ for
$\phi \in \hhom_R(M,N)$.
\end{defn}

\begin{defn}\label{d:homQ}
Write $\Hom_Q$ and $\Ext^i_Q$ to mean $\Hom_{\kk[Q_+]}$ and
$\Ext^i_{\kk[Q_+]}$, respectively.
\end{defn}

\begin{prop}\label{p:vee}
If $F_\spot$ is any complex of $Q$-modules with Matlis dual $I^\spot =
F_\spot^\vee$, and $C_\spot$ is any complex of $Q$-modules, then
canonically
$$
\hhom_Q(C_\spot, I^\spot)
\cong
(C_\spot \otimes_Q F_\spot)^\vee.
$$
\end{prop}
\begin{proof}
This is \cite[Lemma~11.16]{cca}, which is stated for modules but works
as well for complexes by functoriality of~$\hhom$ and~$\otimes$.  Also note
that the proof there---and~hence the statement here---is a simple abstract
adjunction with no noetherian requirement.\!\!
\end{proof}

\begin{cor}\label{c:vee}
For any $Q$-modules~$M$ and~$N$, there is a natural isomorphism
$$
  \eext^i_Q(M, N^\vee) = \Tor_i^Q(M, N)^\vee.
$$
\end{cor}
\begin{proof}
Apply Proposition~\ref{p:vee} to flat resolutions $C_\spot$ and~$F_\spot$
of~$M$ and~$N$, respectively, so the Matlis dual $I^\spot = (F_\spot)^\vee$ is an
injecive rsolution of~$N^\vee$.  Actually, it suffices to let $F_\spot$ be a
flat resolution of~$N$:
\begin{align*}
\eext^i_Q(M,N^\vee)&
  = H^i\hhom_Q(M,I^\spot)
\\*&
  = H^i\hhom_Q\bigl(M,\hhom_\kk(F_\spot, \kk)\bigr)
\\*&
  = H^i\hhom_\kk\bigl(M \otimes_Q F_\spot, \kk)\bigr)
\\*&
  = \hhom_\kk\bigl(H_i(M \otimes_Q F_\spot), \kk)\bigr)
\\*&
  = H_i(M \otimes_Q F_\spot)^\vee
\\*&
  = \Tor_i^Q(M, N)^\vee.
\qedhere
\end{align*}
\end{proof}

\begin{lemma}\label{l:vee-preserves-X}
Matlis duality preserves class~$\XX$: if $M$ is of class~$\XX$ then so
is~$M^\vee$.  Moreover, in that case, $(M^\vee)^\vee \cong M$, as in any case where
$\dim_\kk M_\qq < \infty$ for all~$q \in Q$.
\end{lemma}
\begin{proof}
If a subdivision is subordinate to~$M$, then taking the negative
(in~$Q$) of each region in the subdivision yields a subdivision
subordinate to~$M^\vee$ that is class~$\XX$ by the negation clause in
Definition~\ref{d:tame}.  The claim about $(M^\vee)^\vee$ follows from
the fact that any vector space of finite dimension is naturally
isomorphic to its double dual; note that class~$\XX$ implies tame
(Definition~\ref{d:tame}), which includes this finite dimensionality.
\end{proof}

\begin{remark}\label{r:matlisexact}
Matlis duality is exact (it is vector space duality, degree by degree)
and it swaps flat and injective modules \cite[Lemma~11.23]{cca}.  Note
that $\kk[Q_+]$ does not need to be noetherian for the brief
adjunction argument there.
\end{remark}

The following lemma gives a structural result about tame injective
modules, or more generally (Remark~\ref{r:tame-vs-class-X}) injective
modules of class~$\XX$.

\begin{lemma}\label{l:injective-decompose}
Every class~$\XX$ injective $Q$-module $I$ is a finite direct sum of
class~$\XX$ indecomposable injective modules.
\end{lemma}
\begin{proof}
Since $I$ is tame by Definition~\ref{d:tame}, $\dim_\kk M_\qq$ is
finite for every $\qq \in Q$ also by Definition~\ref{d:tame}.
Therefore, the Krull--Remak--Schmidt--Azumaya theorem
\cite[Theorem~1.1]{botnan-crawley-boevey2020} yields a decomposition
$I = \bigoplus_{\lambda\in\Lambda} I_\lambda$ with each $I_\lambda$ an
indecomposable $Q$-module.  Being a direct summand of an injective
module, each $I_\lambda$ is also injective.  For the first claim, it
remains is to show that $\Lambda$ is finite.

Fix $I_\lambda$ for some $\lambda \in \Lambda$ and a non-zero element
$a \in (I_\lambda)_{\qq_0}$ for some $\qq_0 \in Q$.  Multiplication
by~$\xx^{\qq_0}$ induces a degree~$\0$ homomorphism $\kk\{\qq_0+Q_+\}
\to I_\lambda$ that extends to a degree~$\0$ homomorphism $\phi:
\kk\{Q\} \to I_\lambda$ because $I_\lambda$ is injective.  Since the
image of~$\phi$ is a quotient of~$\kk\{Q\}$, it is a downset module
$\kk\{D_\lambda\} \cong \image(\phi) \subseteq I_\lambda$.  But every
finite collection of downsets in~$Q$ has nonempty intersection (this
is elementary to prove directly; partially ordered abelian groups are
``lower-connected'' in the language of
\cite[Definition~3.5.3]{hom-alg-poset-mods}) and $\dim_\kk M_\qq$ is
bounded thanks to any subordinate constant subdivision which has only
finitely many constant regions.  It follows that $\Lambda$~is~finite.

The second claim is by Lemma~\ref{l:indecomp-class-X}, which is
separated off for the record.
\end{proof}

\begin{lemma}\label{l:indecomp-class-X}
Assume that a $Q$-module $M \cong \bigoplus_{\lambda\in\Lambda}
M_\lambda$ is decomposable as a direct sum.  If $M$ is tame or of
class~$\XX$, then so is $M_\lambda$ for each index~$\lambda$.
\end{lemma}
\begin{proof}
The direct sum induces projection maps $M_\aa \to (M_\lambda)_\aa$ for all $\aa \in Q$.
The structure maps of~$M$ (multiplications by monomials) induce
structure maps on each of its summands.  Any constant subdivision
of~$Q$ subordinate to~$M$ is therefore automatically subordinate
to~$M_\lambda$, by direct check of Definition~\ref{d:constant-subdivision}.
\end{proof}

\begin{defn}\label{d:X-injective}
A $Q$-module $M$ is \emph{$\XX$-injective} if its Matlis dual~$M^\vee$
(Definition~\ref{d:matlis}) is $\XX$-flat (Definition~\ref{d:X-flat}).
An $\XX$-injective $Q$-module is \emph{semisimply injective} if $Q$ is
a Presburger group and $\XX = \text{semisimple}$.  (Note again that
the meaning of ``semisimply'' here refers to
Definition~\ref{d:semisimple}; cf.~Definition~\ref{d:X-flat}.)
\end{defn}      

\begin{remark}\label{r:injectives}
If $Q_+$ is an affine semigroup, then by \cite[Theorem~11.30]{cca} the
injective $\kk[Q_+]$-modules are direct sums of (possibly infinitely
many) indecomposable injectives of the form $\kk\{\qq + F - Q_+\}$ for
some $\qq \in Q$ and face~$F$ of~$Q_+$ (Definition~\ref{d:face}).  By
Lemma~\ref{l:injective-decompose} the number of indecomposables is
finite if the injective module is tame.
\end{remark}

\begin{example}\label{e:X-injective}
Over a general class~$\XX$ partially ordered abelian group~$Q$, there
could be a distinction between $\XX$-injective $Q$-modules and
injective $Q$-modules of class~$\XX$,
because a class~$\XX$ indecomposable injective need not be Matlis dual
to a localization along a face.  Take $Q = \RR^n$, for instance, with
$Q_+ = (\RR_{\geq0})^n$.  If $U = (\RR_{>0})^n$ is the interior of the
positive orthant, then the upset module $\kk[U]$ is flat
\cite[Definition~2.2 and Proposition~2.4]{gldim} and indecomposable
but not a localization of~$\kk[Q_+]$.
\end{example}

\begin{remark}\label{r:not-restrictive}
Despite the caveat in Example~\ref{e:X-injective}, using only
$\XX$-injectives instead of arbitrary injectives of class~$\XX$ is not
restrictive in the cases of interest here, because discrete settings
like Presburger definable behave more like Remark~\ref{r:injectives}
than Example~\ref{e:X-injective}.  Indeed, using Lazard's criterion
\cite{lazard1964} to express flat modules as filtered colimits of free
modules, discrete multigradings force degrees to head off to infinity.
The failure of discreteness in Example~\ref{e:X-injective} is
decisive, and it can cause far worse behavior than oddly shaped flat
or injective modules: tame modules need not admit flat presentations
with countably many summands \cite[Example~1.2]{essential-real}, let
alone finitely many localization summands.  Tame modules do, however,
always admit finite resolutions by finite direct sums of upset or
downset modules \cite[Theorem~6.12]{hom-alg-poset-mods}, so numerics
in these more general situations are not entirely hopeless.
\end{remark}

\begin{thm}\label{t:ext}
Fix a class~$\XX$ group~$Q$.  If one of the following two scenarios is
in effect for given $Q$-modules~$M$ and~$N$, then $\eext^i_Q(M,N)$ is
of class~$\XX$.
\begin{enumerate}
\item\label{i:inj}%
$M$ is of class~$\XX$ and $N$ has an injective resolution~$I^\spot$ that
is $\XX$-injective in cohomological degrees $i - 1$, $i$, and $i + 1$.
\item\label{i:flat}%
$N$ is of class~$\XX$ and $M$ has a flat resolution~$C_\spot$ that is
$\XX$-flat in homological degrees $i - 1$, $i$, and $i + 1$.
\end{enumerate}
\end{thm}
\begin{proof}
By Remark~\ref{r:matlisexact}, in Scenario~\ref{i:inj}, $F_\spot =
(I^\spot)^\vee$ is a flat resolution of~$N^\vee$.  The $\XX$-injective
hypothesis in cohomological degrees surrounding~$i$ imply that the
three-term complex $I^{i-1} \to I^i \to I^{i+1}$ is Matlis dual to
$F_{i-1} \from F_i \from F_{i+1}$ by the double-dual part of
Lemma~\ref{l:vee-preserves-X}.  Thanks to Proposition~\ref{p:vee},
$\eext_Q^i(M, N)$ is Matlis dual to the middle homology of the
three-term complex \mbox{$M \otimes_Q F_{i-1} \from M \otimes_Q F_i
\from M \otimes_Q F_{i+1}$}.  The resulting cohomology is of
class~$\XX$ by Theorem~\ref{t:X-flat} and
Lemma~\ref{l:vee-preserves-X}.

In Scenario~\ref{i:flat}, $\eext_Q^i(M, N)$ can be computed as the
Matlis dual of the middle homology of the three-term complex $C_{i-1}
\otimes_Q N^\vee \to C_i \otimes_Q N^\vee \to C_{i+1} \otimes_Q N^\vee$ by
Proposition~\ref{p:vee} and the double-dual part of
Lemma~\ref{l:vee-preserves-X} (the latter applied to~$N$).  The
resulting cohomology is of class~$\XX$ by Theorem~\ref{t:X-flat} and
Lemma~\ref{l:vee-preserves-X}.
\end{proof}

\begin{remark}\label{r:noetherian-assump'}
Assume that $Q_+$ is an affine semigroup whose faces are class $\XX$
and Minkowski summable.  Assume further that all translates of~$Q_+$
are class~$\XX$.  (These assumptions on $Q_+$ are all satisfied when
$\XX = \text{semisimple}$.)
\begin{enumerate}
\item\label{i:inj-assump}%
The assumption on $I^\spot$ in Theorem~\ref{t:ext}.\ref{i:inj} is
satisfied in all cohomological degrees by any noetherian or artinian
$N$ by Remarks~\ref{r:matlisexact},~\ref{r:noetherian-assump},
and~\ref{r:injectives}.

\item\label{i:flat-assump}%
The assumption on $C_\spot$ in Theorem~\ref{t:ext}.\ref{i:flat} is
satisfied in all homological degrees by any noetherian or artinian $M$
as shown in Remark~\ref{r:noetherian-assump}.
\end{enumerate}
\end{remark}

\section{Constructible families}\label{s:families}

This section extends the notion of (semisimple) constructibility from
modules to families of modules.  It essentially takes the theoretical
join of \S\ref{s:numerics} and~\ref{s:constructible-mods}, which
respectively treat families at a numerical level and constructibility
for individual modules as opposed to in families.

In this section, $Q$ is a Presburger group
(Definition~\ref{d:presburger-group}) unless otherwise specified, and
$Z$ is a free abelian group (Convention~\ref{conv:conventions}), as
is~$Y$.

\subsection{Rees monoids}\label{b:rees-monoids}\mbox{}

\medskip
\noindent
Classically, the Rees algebra of a family of ideals lends control over
numerics, if the Rees algebra is noetherian.  Here, what lends that
control is the Rees monoid.

\begin{defn}\label{d:rees-monoid}
A \emph{Rees monoid} over~$Q$ is the positive cone~$G_+$ of a
Presburger group $G = Z \times Q$ such that $G_+ \cap \bigl(\{\0\} \times Q \bigr)
= \{\0\} \times Q_+$ (so the zero-slice of the positive cone~$G_+$ is a
copy~of~$Q_+$).  A~Rees monoid $G_+$ over~$Q$ is \emph{free} or
\emph{flat} if, respectively, the monoid algebra $\kk[G_+]$ is free or
flat as a module over~$\kk[Q_+]$.
\end{defn}

\begin{example}\label{e:sharpened-pencils}
Consider a case where $Q \cong \ZZ^2$ and $Z \cong \ZZ$.  In the left-hand image,
$$
\includegraphics[height=27ex]{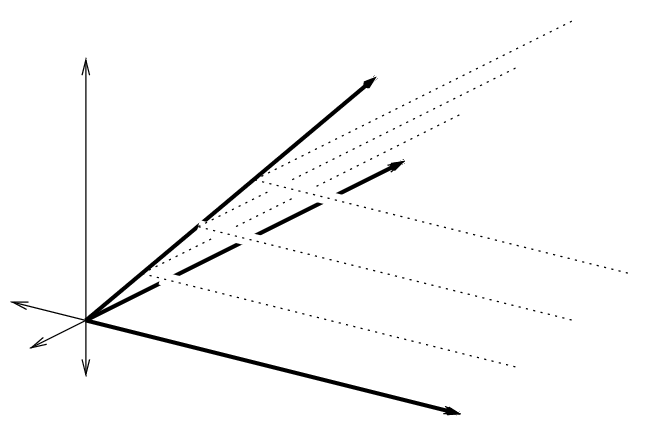}
\qquad\quad
\includegraphics[height=27ex]{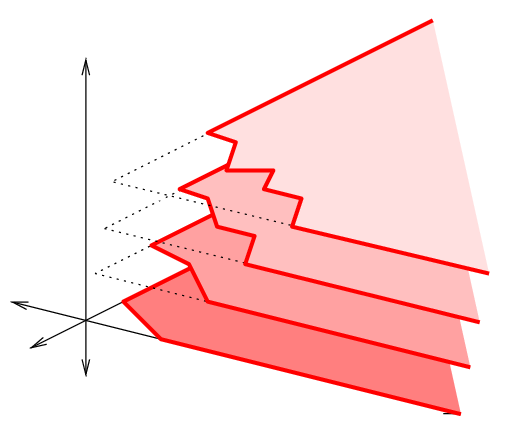}
$$
let $G_+$ be the monoid of lattice points in the cone generated by the
three bold rays and let $Q$ be the horizontal plane, which is
identified with $\{\0\} \times Q$ as the set of lattice points in the cone
spanned by the two horizontal bold rays.  Horizontal slices of~$G_+$
parallel to~$Q$ are outlined with dashed rays, marching upward from
the horizontal slice $G_+ \cap \bigl(\{\0\} \times Q \bigr)$,
which equals~$\{\0\} \times Q_+$ as specified by
Definition~\ref{d:rees-monoid}.  The right-hand image depicts a
$G$-family, as explained in Example~\ref{e:sharpened-pencils'}.
\end{example}

The rubric for Definition~\ref{d:rees-monoid} is the quintessential
construction of the Rees algebra from powers of an ideal.  It is worth
worth isolating this special case.  It is also useful to widen the
context to allow independent powers of finitely many ideals.

\begin{defn}\label{d:rees-algebra}
The \emph{Rees algebra} of a set of monomial ideals $J_1, \ldots, J_k$ in
$\kk[Q_+]$ is the algebra $\kk[Q_+][J_1 t_1, \ldots, J_k t_k]\subseteq \kk[Q_+][t_1, \ldots,
t_k]$, where $t_1, \ldots, t_k$ are indeterminates.
\end{defn}

\begin{example}\label{e:rees-alg}
Fix nonzero constructible monomial ideals $J_1, \ldots, J_k$ of~$\kk[Q_+]$.
If $G_+ \subseteq \NN^k \times Q_+$ is the semigroup whose monoid algebra~$\kk[G_+]$
equals $\kk[Q_+][J_1 t_1, \ldots, J_k t_k]$, then by
Proposition~\ref{p:atoms-definable} $G_+$ is a Rees monoid over~$Q$,
with $G = \ZZ G_+ = \ZZ^k \times Q$: the atoms of $G_+$ correspond to atoms
of~$Q_+$ and minimal generators of $J_1, \ldots, J_k$.
\end{example}

\begin{defn}\label{d:ideal-rees-monoid}
In Example~\ref{e:rees-alg}, $G_+$ is the \emph{Rees monoid} of the
ideals $J_1, \ldots, J_k$.
\end{defn}

\subsection{Constructible families}\label{b:constructible-families}\mbox{}

\medskip
\noindent 
Classically, a family yields a module over the Rees algebra.  The same
idea works here.  Added freedom in this multigraded context arises
because this Rees algebra module need not be noetherian; it only needs
to be definable by Presburger arithmetic.  Recall the notion of
(semisimply) constructible module from
Definition~\ref{d:constructible}.

\begin{defn}\label{d:family}
Fix a Rees monoid $G_+$ over $Q$ with $G = Z \times Q$.
A~family~$\{M_\nn\}_{\nn \in Z}$ of~$Q$-modules is a \emph{$G$-family}
if its direct sum yields a $G$-module
$$
  \cM
  = \bigoplus_{\nn \in Z} M_\nn(-\nn)
  = \bigoplus_{\ggg \in G} M_\ggg
  = \bigoplus_{\nn\qq} M_{\nn\qq}
$$
where $M(-\nn)$ is the $Q$-module $M$ shifted to the slice $\{\nn\} \times Q$
of~$G$.  A~\emph{constructible family} over~$G_+$ is a $G$-family
$\{M_\nn\}_{\nn \in Z}$ whose \emph{associated} $G$-module~$\cM$ is
constructible.
\end{defn}

\begin{example}\label{e:sharpened-pencils'}
The right-hand image in Example~\ref{e:sharpened-pencils} is a
$G$-family for the Rees monoid~$G_+$ in the left-hand image there.
The direct sum $\cM$ as in Definition~\ref{d:family}, represented by the
union of the red horizontal slices, is contained in a monoid algebra
$\kk[G'_+]$ for a monoid~$G'_+$ that is larger than~$G_+$, as indicated
by the dotted lines in the right-hand image.  However, $\cM$ is not an
ideal in~$\kk[G'_+]$.  Neither is~$\cM$ an ideal in~$\kk[G_+]$.  What
makes~$\cM$ a $G$-family is the slope of the diagonally vertical bold
ray in the left-hand image, which is slanted enough to ensure that
lifting each horizontal slice upward along that ray places the red
region of the lifted slice inside the red region of the target slice.
This picture is familiar from usual Rees algebra theory (especially of
monomial ideals), but here there is more freedom to choose~$G_+$ given
the direct~sum~$\cM$.
\end{example}

\begin{remark}\label{r:fam-singleton}
In Definition~\ref{d:family}, if $Z$ has rank~$0$ then a $G$-family
$\{M_\nn\}_{\nn \in Z}$ consists of a single $Q$-module~$M$, which is
constructible when the family is constructible.
\end{remark}

The simplest examples of constructible families are those over Rees
algebras of finitely many constructible monomial ideals.  It is useful
to develop methods of detecting this constructibility.

\begin{defn}\label{d:membership}
Fix a family\/ $\MM = \{M_\nn\}_{\nn \in Z}$ of indicator $Q$-modules indexed
by a free abelian group~$Z$, so $M_\nn = \kk\{S_\nn\}$ for some poset-convex
subset $S_\nn \subseteq Q$ (Definition~\ref{d:indicator-module}).  To say
\emph{membership in~$\MM$ is Presburger} means that the set $\bigcup_{\nn \in Z}
\{\nn\} \times S_\nn = \{\nn\qq \in Z \times Q \mid \dim_\kk M_{\nn\qq} = 1\}$ is semisimple.
\end{defn}

\begin{lemma}\label{l:membership}
Membership in a family $\bbI = \{I_\nn\}_{\nn \in Z}$ of monomial ideals is
Presburger if and only if the \emph{top set} $T \subset Z \times Q_+$ of all $\nn\qq$
such that $\xx^\qq$ is a minimal monomial generator of~$I_\nn$ is
semisimple.
\end{lemma}
\begin{proof}
Choose isomorphisms $Q \cong \ZZ^d$ and $Z \cong \ZZ^k$, and use
Theorem~\ref{t:presburger} to pass freely between ``semisimple'' and
``Presburger definable''.  Let $Q_+$ be defined by the Presburger
formula~$G(\vv)$.  Assume membership in $\bbI$ is definable by~$F(\nn,\qq)$.
The top set $T$ can be defined via the formula
$$
  H(\nn,\qq)
  =
  \neg(\exists \aa, \vv \in \ZZ^d) \bigl(F(\nn,\aa) \wedge G(\vv) \wedge \neg(\qq = \0) \wedge (\qq = \aa + \vv)\bigr).
$$
Conversely, assume $T$ is defined by $H(\nn,\aa)$.  Membership in $\bbI$ is
Presburger definable by the formula
\begin{equation*}
  F(\nn,\qq)
  =
  (\exists \aa,\vv \in \ZZ^d) \bigl(H(\nn,\aa) \wedge G(\vv) \wedge (\qq = \aa + \vv)\bigr).\qedhere
\end{equation*}
\end{proof}

\begin{prop}\label{p:rees-is-presburger}
Fix a family $\MM = \{M_\nn\}_{\nn\in\ZZ^k}$ of indicator $Q$-modules and
nonzero constructible monomial ideals $J_1, \ldots, J_k$ of\/~$\kk[Q_+]$.
The following are equivalent.
\begin{enumerate}
\item\label{i:fam-constructible}%
$\MM$ is constructible over the Rees monoid of $J_1,\ldots,J_k$.
\item\label{i:other-conditions}%
Membership in~$\MM$ is Presburger and $J_i \subseteq
\bigcap_{\nn\in\ZZ^k} M_{\nn+\ee_i} :_R M_\nn$ for every~$i$, where
$\ee_i$ is the $i$th standard basis vector of\/~$\ZZ^k$.
\end{enumerate}
\end{prop}
\begin{proof}
The Rees monoid $G_+$ of $J_1, \ldots, J_k$ makes $G = \ZZ G_+$ a Presburger
group because the ideals are constructible.  This allows the notion of
constructible $G$-family to make sense in the first place.

Now define $\cM = \bigoplus_{\nn \in \ZZ^k} M_\nn(-\nn)$.  The condition for $\cM$ to be a
$G$-module is that $J_i M_\nn \subseteq M_{\nn+\ee_i}$ for every~$i$ and all~$\nn \in
\ZZ^k$.  This condition is equivalent to requiring that $J_i \subseteq \bigcap_{\nn\in\ZZ^k}
M_{\nn+\ee_i}:_R M_\nn$ for every~$i$.  Moreover, $\cM$ dominates a constant
subdivision of~$G$ with two regions
$S_0 = \{\ggg\in G \mid \cM_\ggg = 0\}$ and $S_1 = \{\ggg\in G \mid \dim_\kk \cM_\ggg = 1\}$
because the supports of the $Q$-slices~$M_\nn$ are poset-convex by
Definition~\ref{d:indicator-module}.  (The poset-convexity ensures
that structure homomorphisms $\kk = \cM_\ggg \to \cM_{\ggg'} = \kk$ are all
isomorphisms and not forced to factor through some $\cM_{\ggg''} = 0$.)
This subdivision is semisimple precisely when membership in~$\MM$ is
Presburger definable.
\end{proof}

\begin{cor}\label{r:restrict-rees-monoid}
If the family $\MM$ is constructible over the Rees monoid of
constructible monomial ideals $J_1, \ldots, J_k$, and $J'_i \subseteq J_i$ are
nonzero constructible monomial ideals, then $\MM$ is also constructible
over the Rees monoid of $J_1', \ldots, J_k'$.
\end{cor}
\begin{proof}
Proposition~\ref{p:rees-is-presburger}.\ref{i:other-conditions}
immediately passes to the sub-ideals~$J_i'$.
\end{proof}

The following proposition allows us to restrict to free (and hence
flat) Rees monoids when dealing with tensor products in the proof of
Theorem~\ref{t:tor-family}.

\begin{prop}\label{p:free-reduction}
Each constructible family $\{M_\nn\}_{\nn \in Z}$ over a Rees
monoid~$G_+$ over~$Q$, with $G = Z \times Q$, has a \emph{free
reduction}: a free Rees monoid $G_+' \subseteq G_+$ over~$Q$, with $G'
= Z' \times Q$ of the same rank as~$G$, such that $\{M_\nn\}_{\nn \in
Z}$ is a constructible family~over~$G'$.
\end{prop}
\begin{proof}
Let $k$ be the rank of~$Z$ and $\pi: G = Z \times Q \to Z$ be
projection.  Since $G_+$ generates a finite index subgroup of~$G$ by
the full hypothesis in Definition~\ref{d:rees-monoid} (see
Definition~\ref{d:pogroup}), there exist $k$ linearly independent
elements $\nn_1, \ldots, \nn_k$ in $\pi(G_+)$.  Denoting each ordered
pair $(\nn,\qq) \in Z \times Q$ by~$\nn\qq$, for notational simplicity
(and to match the subscripts in Definition~\ref{d:family}), for
each~$i$ choose $\nn_i \qq_i \in G_+ \cap \pi^{-1}(\nn_i)$.  Let
$G_+'\subseteq G_+$ be the monoid generated by $\{\0\} \times Q$ and
the $\nn_i \qq_i$.  By construction, $G_+'$ is a free Rees monoid
over~$Q$.  In particular, $G_+'$ is semisimple, and $G' = \ZZ G_+' =
Z'\times Q$ has the same rank as~$G$, where $Z' = \ZZ\{\nn_1, \ldots,
\nn_k\}$.

The restriction~$M_{\zz + Z'}$ of~$\cM$ to the coset of~$Z'$ based at
$\zz \in Z$ is a $G'$-family because $\cM$ is a $G$-family and $G_+'
\subseteq G_+$.  Moreover, if $G = \bigcup_{\alpha \in A} S_\alpha$ is
a semisimple constant subdivision of~$G$ subordinate to~$\cM$, then
intersecting each $S_\alpha$ with the coset $\zz + Z'$ yields a
semisimple constant subdivision of~$G$ subordinate to~$M_{\zz + Z'}$
after appending one additional constant region $Z \minus (\zz + Z')$
where $M_{\zz + Z'}$ vanishes.  Hence $M_{\zz + Z'}$ is constructible,
and so is~$\cM$, being the finite direct sum of these over the cosets
of~$Z'$.
\end{proof}

\begin{example}\label{e:cone-over-square}
View $Q_+ = \NN^2$ as lying in a horizontal plane, and let $G_+$ be the
cone over a ``square'', namely generated by $Q_+$ and the two vectors
$(1,0,1)$ and $(0,1,1)$ directly above the two generators of~$Q_+$.
The submonoid $G_+'$ in Proposition~\ref{p:free-reduction} could
be generated by, say, $Q_+$ and the vector $(1,0,1)$ that sits above
the $x$-axis.  Although any constructible family over $G_+$ is
automatically constructible over~$G_+'$ by the Proposition, it is not
true that a noetherian $G$-module must be noetherian as a $G'$-module.
Indeed, even $\kk[G_+]$ itself fails to be finitely generated
over~$\kk[G_+']$.  Constructibility is a weaker---meaning more
inclusive---absolute combinatorial condition, whereas noetherian is a
stronger relative algebraic property that depends strongly on the base
ring.
\end{example}

One strength of constructible module theory is the ease of extending
it to complexes.

\begin{defn}\label{d:complexes-constructible}
Fix a Rees monoid $G_+$ over $Q$ with $G = Z \times Q$.
A~\emph{constructible family of complexes} is a family
$\{C^\spot_\nn\}_{\nn \in Z}$ of complexes of $Q$-modules whose direct sum
$C^\spot$ constitutes a constructible complex of $G$-modules; that is,
for each $i \in \ZZ$,
$$
  C^i = \bigoplus_{\nn \in Z} C^i_\nn(-\nn)
      = \bigoplus_{\ggg \in G} C^i_\ggg = \bigoplus_{\nn\qq} C^i_{\nn\qq}
$$
\end{defn}

\begin{prop}\label{p:constructible-cohomology}
Fix a constructible family of complexes
$\{C^\spot_\nn\}_{\nn\in\ZZ^k}$ of $Q$-modules.  For each $i \in \ZZ$
the cohomologies $\{H^i(C^\spot_\nn)\}_{\nn\in\ZZ^k}$ form a
constructible family of $Q$-modules.
\end{prop}
\begin{proof}
This follows directly from Proposition~\ref{p:valid-class-X} and the
fact that kernels and cokernels of class~$\XX$ morphisms are of
class~$\XX$ \cite[Proposition~4.30.1]{hom-alg-poset-mods}.
\end{proof}

\subsection{Quasipolynomiality from constructibility}\label{b:quasipolynomiality}\mbox{}

\medskip
\noindent 
Quasipolynomial consequences of constructibility of course pass
through the numerical version,\hspace{-.8pt} via
an unassuming lemma,\hspace{-.8pt} the tie that binds tameness to
Presburger~\mbox{arithmetic}.

\begin{lemma}\label{l:unassuming}
Any constructible family is numerically constructible.
\end{lemma}
\begin{proof}
Any semisimple subdivision subordinate to the module $\cM$ associated to
a constructible $G$-family has $\dim_\kk M_\ggg$ constant for $\ggg$ in any
constant region of~$G$.
\end{proof}

The following two examples show that the converse of
Lemma~\ref{l:unassuming} need not hold.

\begin{example}\label{e:not-presburger-family}
Assume the situation of Example~\ref{e:numerical} with $G_+ = \NN
\times Q_+$, and further assume that the monomial ideal~$I$ is proper
and nonzero.  The family $\{M_n\}_{n \in \ZZ}$ is numerically
constructible by Example~\ref{e:numerical}, but there is no way to
endow $\cM = \bigoplus_{n\in\ZZ} M_n(-n)$ with a $G$-module structure
in such a way that it is a constructible $G$-module.  Indeed,
homomorphisms between non-adjacent slices can't be isomorphisms,
because any map $\kk[x,y] \to \kk[x,y]/I \oplus I$ of $Q$-modules must
contain~$I$ in the kernel, and any $Q$-module map $\kk[x,y]/I \oplus I
\to \kk[x,y]$ must vanish on the summand $\kk[x,y]/I$, which is
torsion.  Consequently, any constant subdivision subordinate to any
$G$-module structure on~$\cM$ must place each odd slice $\kk[x,y]$ in
a different region than every other odd slice.
This prevents a finite constant subdivision.
\end{example}

\begin{example}\label{e:semisimple}
Fix $Q = \ZZ$ and $Q_+ = \NN$.  Let $G = \ZZ \times Q = \ZZ^2$
with positive cone~$\color{darkgreen}G_+$ (depicted in
{\color{darkgreen}green}) generated by $[\twoline 01]$ and
$\color{darkpurple}[\twoline 11]$
(the action of the latter depicted in {\color{darkpurple}purple}).
Let $\beta_0 \beta_1 \beta_2 \ldots$ be a transcendental binary string.  Then set
\begin{itemize}
\item%
$M_n = \kk[\NN]$ if $\beta_n = 0$;
\item%
$M_n = \kk \oplus \kk\{1 + \NN\}$ if $\beta_n = 1$, where $\kk \cong
\kk[\NN]/\kk\{1 + \NN\}$; and
\item%
$M_n = 0$ if $n < 0$.
\end{itemize}
\vspace{-1ex}
$$
\text{For example, with $\beta = 01000101100\ldots$\ :\qquad}\hfill
\begin{array}{@{}c@{}}
  \psfrag{x}{\raisebox{-.5ex}{\footnotesize$Q_+$}}
  \psfrag{y}{\footnotesize${}$}
  \psfrag{M0}{\tiny$\!\!M_0$}
  \psfrag{M1}{\tiny$\!\!M_1$}
  \psfrag{M2}{\tiny$\!\!M_2$}
  \psfrag{M3}{\tiny$\!\!M_3$}
  \psfrag{M4}{\tiny$\!\!M_4$}
  \psfrag{M5}{\tiny$\ \vdots$}
  \includegraphics[height=19ex]{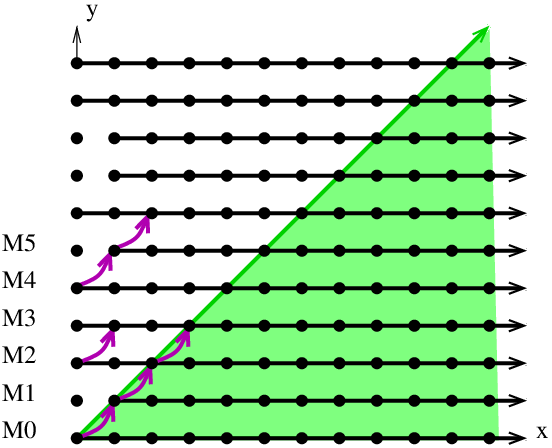}
\end{array}
\qquad\qquad\qquad\quad\ \ \,
$$
Note that $\cM = \bigoplus_{n \in \ZZ} M_n(-n)$ is a $G$-module, where every nonzero
element of~$G_+$ acts as a (shifted) inclusion on the $\kk[\NN]$-free
summands and as~$0$ on the torsion summands~$\kk$.  Thus $G_+$ acts by
translation along slices (as $Q_+$ does) or takes each $Q$-slice to a
shifted copy of that slice in higher levels.  These shifts ensure that
the aperiodic wiggling of the slice placements does not prevent the
free summand of a lower slice from landing inside of the free summand
of any higher slice.

This family~$\{M_n\}_{n \in \ZZ}$ is numerically constructible, with
two constant regions:\ the nonnegative quadrant of $G \cong \ZZ^2$
(where all of the vector-space dimensions are~$1$) and its complement
(where all of the vector-space dimensions are~$0$).

However, the module~$\cM$ is not constructible because, for example, the
locations of the vector space summands~$\kk$ occur at a subset of $\ZZ
\times \{0\}$ whose transcendence prevents it from being semisimple.
To be more precise, $\cM$ dominates a finite constant subdivision with
just three regions:
\begin{itemize}
\item%
the torsion degrees $[\twoline n0]$, each of which indexes a torsion
summand $\kk$ when $\beta_n = 1$;
\item%
the torsion-free degrees, comprising all other nonnegative vectors
in~$\ZZ^2$; and
\item%
the complement of the nonnegative quadrant, where $\cM$ vanishes.
\end{itemize}
Thus $\cM$ is tame but not semisimply constructible.
%
%
%
\end{example}

\subsection{Constructible families of ideals}\label{b:from-ideals}\mbox{}

\medskip
\noindent
Many main results in this paper assert that some operation on a
constructible family yields another constructible family.  That is,
turning some crank preserves constructibility and hence can yield
intricate output with piecewise quasipolynomial numerics.  But where
is the recursion to start?  The purpose of this section is to provide
rich sources of constructible families comprising---or in some cases,
arising from---monomial ideals.

\begin{defn}\label{d:classical-defs}
Let $I$ and $K$ be monomial ideals of $R = \kk[Q_+]$.
\begin{enumerate}
\item\label{i:newton-polyhedron}%
The \emph{Newton polyhedron} of~$I$ is the convex hull $\NP(I)$ of
$\{\qq \in Q_+ \mid \xx^\qq \in I\}$.

\item\label{i:integral-closure}%
The \emph{integral closure} of $I$ is ideal $\oI$ whose monomials
have some $k$th power in~$I^k$:
$$
  \oI = \kk\{\qq \in Q \mid \xx^{k\qq} \in I^k \text{ for some } k \in \NN\}.
$$

\item\label{i:saturation}%
The \emph{saturation} of $I$ \emph{with respect to}~$K$ is the ideal
$$
  I:_R K^\infty
  =
  \bigcup_{r \gs 0} I :_R K^r.
$$

\item\label{i:symbolic}%
The $n$th~\emph{symbolic power} of~$I$ is
$$
  I^{(n)} = \bigcap_{\pp \in \Min(I)} (I^n R_\pp \cap R),
$$
where $\Min(-)$ denote the set of minimal primes.

\item\label{i:multiplier-ideal}%
$R$ is \emph{$\QQ$-Gorenstein} if some symbolic power of the ideal
$\kk\{Q_+^\circ\}$ is principal, where $Q_+^\circ \subset Q_+$ is the upset of lattice
points interior to~$Q_+$.  If $\kk\{Q_+^\circ\}^{(r)} = \kk[\ww + Q_+]$, then
the \emph{multiplier ideal} of~$I$ is the monomial ideal $\cJ(I)$ such
that
$$
  \xx^\qq \in \cJ(I)
  \iff
  \qq + \mathbf{w}/r \in \NP(I)^\circ,
$$
the interior of the Newton polyhedron~$\NP(I)$.
\end{enumerate}
\end{defn}

\begin{remark}\label{r:classical-defs}
The item numbers here refer to those in
Definition~\ref{d:classical-defs}.
\begin{enumerate}
\item[\ref{i:integral-closure}.]%
It is nontrivial that $\oI$ is an ideal, but in fact $\oI = \kk\{\NP(I)
\cap Q\}$ \cite[\S1.4]{swanson-huneke2006}.

\item[\ref{i:saturation}.]%
The saturation $I^\sat = I :_R \mm^\infty$ of~$I$ with respect to the
homogeneous maximal ideal $\mm = \kk\{Q_+\minus\{\0\}\}$ of~$R$ is often called
simply \emph{the saturation of~$I$}.

\item[\ref{i:symbolic}.]%
Write $\Ass(-)$ for the associated primes and $\Ass^\infty(I) = \bigcup_{n \gs 1}
\Ass(I^n)$.  Then
$$
  I^{(n)} = I^n :_R K^\infty \text{ for the ideal }
  K = \bigcap_{\pp \in \Ass^\infty(I) \minus \Min(I)} \pp.
$$
The set $\Ass^\infty(I)$ is finite by \cite{brodman79} in much more
generality than the monomial case here, but this finiteness is obvious
over an affine semigroup ring, because the positive cone~$Q_+$ has
only finitely many faces.

\item[\ref{i:multiplier-ideal}.]%
$\cJ(I)$ is the characteristic-free combinatorial characterization
from \cite[Theorem 4.8]{hara-yoshida2003} (see also
\cite{howald2001}); multiplier ideals have its origins and more
general definitions in algebraic geometry \cite{lazarsfeld2004}.  The
$\QQ$-Gorenstein condition is satisfied when, for instance, $Q_+$ is
smooth (Remark~\ref{r:smooth}).
\end{enumerate}
\end{remark}

The following technical lemma is needed for the proof of
Theorem~\ref{t:ideals-families-strong}.

\begin{lemma}\label{l:uniformizing}
Fix a smooth affine semigroup $N \subseteq \ZZ^n$ and a linear map $\lambda: \ZZ^n \to \ZZ$.
Assume the
hyperplane~$H = \lambda^{-1}(1)$
intersects every extremal ray of the real cone $\RR_{\gs 0} N$.  There
exists a smooth affine subsemigroup $N'\subseteq N$ and a positive integer $r$
such that $N \cap rH \subseteq N'$.  For all such $N'$ there exists a finite
partition $N = \bigcupdot_i(\nn_i + N')$.
\end{lemma}
\begin{proof}
Let $\qq_1, \ldots, \qq_r$ be the generators of~$N$ and $\ppp_i = \RR
\qq_i \cap H \in \QQ N$.  Fix $a_i,b_i \in \ZZ_{>0}$ such that $a_i
\ppp_i = b_i \qq_i$.  The positive integer $r$ can be taken to be any
positive integer multiple of the product of the~$a_i$.  The submonoid
$N'$ is then generated by $r\ppp_1, \ldots, r\ppp_t$.  The last
statement holds since the generators of~$N'$ are multiples of those
of~$N$.
\end{proof}

The following theorem and proposition provide a wide range of examples
of constructible families originating from ideals.  Note in particular
item~\ref{i:constant}, which allows the starting point to be any
monomial ideal.  As a matter of notation, assume the customary
convention that $I^n = R$ whenever $n \ls 0$.

\begin{thm}\label{t:ideals-families-strong}
Fix an affine semigroup $Q_+$ with monoid algebra $R = \kk[Q_+]$.
Let $\{I_n\}_{n \in \ZZ}$ be a family of monomial ideals in~$R$ that is
constructible over the Rees monoid of a nonzero monomial ideal $J \subseteq
R$.  Fix a monomial ideal $K \subseteq R$.  Each of the following is a
constructible family over the indicated Rees monoid~$G_+$.
\begin{enumerate}
\item\label{i:constant}%
The constant family $\{K\}_{n \in \ZZ}$, over the Rees monoid $G_+ = \NN \times
Q_+$.
\item\label{i:powers}%
The powers $\{I_n^{an + b}\}_{n \in \ZZ}$ for any fixed $a,b \in \ZZ$, over
the Rees monoid~$G_+$ of $J^a$.
\item\label{i:integral-closurs}%
The integral closures $\{\ol{I_n}\}_{n \in \ZZ}$, over the Rees
monoid~$G_+$ of~$J$.
\item\label{i:multipliers}%
For $\QQ$-Gorenstein~$R$, the multiplier ideals
$\{\cJ(I_n)\hspace{-.37pt}\}_{n \in \ZZ}$, over the Rees
monoid~of~$J$.
\end{enumerate}
\end{thm}
\begin{proof}
By Proposition\,\ref{p:rees-is-presburger}, showing that each of these
families $\{M_n\}_{n \in \ZZ}$ is constructible over the given Rees monoid
$G_+$ is proved in two steps.
\settowidth\leftmargini{\quad Step 2.}
\begin{itemize}
\item[Step~1.]%
$\cM = \bigoplus_{n \in \ZZ} M_n(-n)$ is a $G$-module, where $G=\ZZ G_+$.
\item[Step~2.]%
Membership in $\{M_n\}_{n \in \ZZ}$ is Presburger
(Definition~\ref{d:membership}).
\end{itemize}
Step~1 is immediate for items~\ref{i:constant}--\ref{i:powers}.  For
item~3 it follows from \cite[Corollary 6.8.6]{swanson-huneke2006}.
For item~4, note that $J\cJ(I_n)\subseteq \cJ(JI_n)\subseteq\cJ(I_{n+1})$, since adding a
point in $\NP(J)$ with an interior point in $\NP(I_n)$ leads to a
point in the interior of~$\NP(JI_n)$.

The interesting part is Step~2.  It follows for item~\ref{i:constant}
because $K$ has a finite set of monomial generators, so focus on
items~\ref{i:powers}--\ref{i:multipliers}.  As usual when dealing with
explicit Presburger formulas, identify $Q \cong \ZZ^d$.

For item~\ref{i:powers}, let $S \subset \ZZ \times \ZZ^d$ be the set of all $n\qq$ with
$\qq$ a minimal monomial generator of~$I_n$.  By
Lemma~\ref{l:membership} $S$ is Presburger definable.
Theorem~\ref{t:presburger} and Definition~\ref{d:semisimple} afford a
finite partition $S = \bigcupdot_{i=1}^c (n_i\qq_i + N_i)$ with $n_i \qq_i \in \ZZ \times
\ZZ^d$ and $N_i \subseteq \ZZ \times \ZZ^d$ smooth affine semigroups.  It must be that
$\bigl(\{0\} \times \ZZ^d\bigr) \cap N_i = \{0\} \times \{\0\}$, since each $I_n$ has
a finite number of generators.  Therefore, by
Lemma~\ref{l:uniformizing}, assume that for each~$i$ all generators
of~$N_i$ have the same first component.  Let $F_i(n,\qq)$ be a
Presburger formula that defines membership in~$N_i$.
Thus membership in $\{I_n^{an+b}\}_{n\in\ZZ}$ is defined by
\[
\begin{split}
F(n,\qq)
  =
  (\exists m_1, & \dots,  m_c \in \NN, \exists \aa_1, \dots, \aa_c \in \ZZ^d)
\\&
    \Big(\bigwedge_{i=1}^c  F_i(m_i(n-n_i),\aa_i)
    \wedge
    \sum_{i=1}^cm_i=an+b\wedge\qq \succeq \sum_{i=1}^c (m_i\qq_i + \aa_i)
    \Big),
\end{split}
\]
where $\succeq$ denotes the partial order on $Q$. 

For item~\ref{i:integral-closurs} the set
$$
  \Bigl\{n\qq \mid \exists k \in \NN, \ppp_1, \ldots, \ppp_k \in S, k\qq \succeq \sum_{i=1}^k \ppp_i\Bigr\},
$$
with $S$ as above, needs to be Presburger definable.  And indeed it
is, via the formula
\[
\begin{split}
  G(n,\qq)
  =
  (\exists k, m_1, & \dots, m_c \in \NN, \exists \aa_1, \dots, \aa_c \in \ZZ^d)
\\&
  \Bigl(
    \bigwedge_{i=1}^c F_i(m_i(n-n_i), \aa_i)
    \wedge
    \sum_{i=1}^c m_i = k
    \wedge
    k\qq \succeq \sum_{i=1}^c (m_i\qq_i + \aa_i)
  \Bigr).
\end{split}
\]

For item~\ref{i:multipliers} it suffices to show that the set $\bigcup_{n\in\ZZ}
\bigl\{\{n\} \times \bigl(\NP(I_n)^\circ \cap \ZZ^d\bigr)\bigr\}$ is Presburger
definable.  This set can expressed as the intersection
\[
\begin{split}
  &\Bigl\{n\qq \mathbin{\big|} \exists k \in \NN, \ppp_1, \ldots, \ppp_k \in S,
          k\qq + \ee_i \succeq \sum_{i=1}^k \ppp_i\Bigr\}
\\
  \cap\,
  &\Bigl\{n\qq \mathbin{\big|} \exists k \in \NN, \ppp_1, \ldots, \ppp_k \in S,
          k\qq - \ee_i \succeq \sum_{i=1}^k \ppp_i\Bigr\},
\end{split}
\]
for $1 \ls i \ls d$.  The proof now follows since each of these finitely
many sets can be defined using a slight variation of the formula
$G(n,\qq)$.
\end{proof}

\begin{prop}\label{p:ideals-families}
Assume $Q_+$ is an affine semigroup and set $R=\kk[Q_+]$.  Let
$\{I_\nn\}_{\nn \in \ZZ^k}$ be a family of monomial ideals that is
constructible over the Rees monoid of nonzero monomial ideals
$J_1,\ldots,J_k$.  Fix a monomial ideal $K$.  The following are
constructible families over the specified Rees monoid.
\begin{enumerate}
\item\label{i:colons}%
The saturations $ \{I_\nn:_R K^{\infty}\}_{\nn \in \ZZ^k}$, over the
Rees monoid of $J_1,\ldots,J_k$.
\item%
The colon ideals $\{I_\nn :_R K^{\aa\cdot\nn+b}\}_{\nn \in \ZZ^k}$
for fixed $\aa = (a_1, \ldots, a_k) \in \ZZ^k$ and $b \in \ZZ$, over
the Rees monoid of $J_1 :_R K^{a_1}, \ldots, J_k :_R K^{a_k}$.
\end{enumerate}\setcounter{separated}{\value{enumi}}
Let $\{K_\nn\}_{\nn \in \ZZ^k}$ be another family of monomial ideals
that is constructible over the Rees monoid of nonzero monomial ideals
$H_1, \ldots, H_k$.  The following are constructible families over the
Rees monoid of $J_1 H_1, \ldots, J_k H_k$.
\begin{enumerate}\setcounter{enumi}{\value{separated}}
\item%
The products $\{I_\nn K_\nn\}_{\nn \in \ZZ^k}$.
\item%
The sums $\{I_\nn +K_\nn\}_{\nn \in \ZZ^k}$.
\item%
The intersections $\{I_\nn \cap K_\nn\}_{\nn \in \ZZ^k}$.
\item\label{i:quotients}%
The quotients $\{I_\nn/K_\nn\}_{\nn \in \ZZ^k}$, when $K_\nn \subseteq I_\nn$ for every~$\nn$.
\end{enumerate}\setcounter{separated}{\value{enumi}}
Let $\{Y_\mmm\}_{\mmm \in \ZZ^\ell}$ be another family of monomial ideals that is
constructible over the Rees monoid of nonzero monomial ideals $X_1, \ldots,
X_\ell$.  The following are constructible families over the Rees monoid
of $J_1, \ldots, J_k, X_1, \ldots, X_\ell$.
\begin{enumerate}\setcounter{enumi}{\value{separated}}
\item\label{i:sums}%
The sums $\{I_\nn + Y_\mmm\}_{\nn\mmm \in \ZZ^{k+\ell}}$.
\item\label{i:products}%
The products $\{I_\nn Y_\mmm\}_{\nn\mmm \in \ZZ^{k+\ell}}$.
\item\label{i:intersections}%
The intersections $\{I_\nn \cap Y_\mmm\}_{\nn\mmm \in \ZZ^{k+\ell}}$.
\end{enumerate}
\end{prop}
\begin{proof}
The argument as in Step~1 and Step~2 from
Theorem~\ref{t:ideals-families-strong} is straightforward.
\end{proof}

\begin{remark}\label{r:symbolic_powers}
Theorem~\ref{t:ideals-families-strong}.\ref{i:powers} and
Proposition~\ref{p:ideals-families}.\ref{i:colons} imply that for
$a,b\in \ZZ$ the symbolic powers $\{I^{(an+b)}\}_{n\in \ZZ}$ form a
constructible family over the Rees monoid of $I$.
\end{remark}

\begin{example}\label{e:combine-constructions'}
The different operations in Theorem~\ref{t:ideals-families-strong}
combine with Proposition~\ref{p:ideals-families} to produce intricate
constructible families.  For instance, the family in
Example~\ref{e:combine-constructions} is constructible over the Rees
monoid of~$I :_R J^3$.
\end{example}

\begin{example}\label{e:multigraded-families}
Proposition~\ref{p:ideals-families}.\ref{i:sums}--\ref{i:intersections}
may be used to produce constructible families indexed by $\ZZ^k$ for $k
> 1$, such as the family $\{I^{2n-1}\ol{J^{5m+1}}K^{(p)}\}_{(n,m,p) \in
\ZZ^3}$ of monomial ideals, which is constructible over the Rees monoid
of $I^2, J^5, K$.
\end{example}

\begin{remark}  
In Lemma~\ref{l:annihilators-constructible} it is shown that given a
monomial ideal $I$ and a constructible family $\{M_\nn\}_{\nn\in\ZZ^k}$ of
$Q$-modules, the submodules of the $M_\nn$ consisting of the elements
annihilated by $I$ also form a constructible family.
\end{remark}

\begin{remark}\label{r:multiplier-general}
There is a more general definition of multiplier ideals over arbitrary
normal rings \cite{deFernex-Hacon2009}, adapted to the monomial case
in \cite{hsiao-matusevich2018}.  The conclusion of
Theorem~\ref{t:ideals-families-strong}.\ref{i:multipliers} still holds
with that general definition.
\end{remark}

\subsection{Products of ideals and modules}\label{b:products}\mbox{}

\medskip
\noindent
Additional constructible families of modules arise by multiplying any
given family by ideals in a constructible family.  The two main points
in the proof are that (i)~a module is constructible if and only if it
admits an upset presentation with semisimple upsets, by the syzygy
theorem \cite[Theorem~6.12.4]{hom-alg-poset-mods}, and (ii)~Minkowski
sums of semisimple sets (in particular, upsets) are semisimple.

The proposition making this precise requires an elementary lemma,
which codifies a situation with two Rees monoids that occurs
repeatedly in subsequent sections.

\begin{lemma}\label{l:double-rees-monoid}
Fix Rees monoids $G_+$ and~$H_+$ over~$Q$ with $G = Z \times Q$ and $H = Y
\times Q$.  The monoid $K_+ = G_+ \times_Q H_+$ defined by
$$
  \kk[K_+]
  =
  \kk[G_+] \otimes_Q \kk[H_+]
$$
is a Rees monoid over~$Q$, called the \emph{double Rees monoid}
of\/~$G_+$ and\/~$H_+$ over\/~$Q$.
\end{lemma}
\begin{proof}
$K_+$ is a Rees monoid over~$Q$ because it is the submonoid $G_+
\times_Q H_+$ of the product $K = Z \times Y \times Q$ generated by
the images of~$G_+$ and~$H_+$ under their embeddings via $G_+ \into G
\cong Z \times \{\0\} \times Q \subseteq K$ and $H_+ \into H \cong
\{\0\} \times Y \times Q \subseteq K$.
\end{proof}

\begin{prop}\label{p:multiply-ideal-module}
\!Fix Rees monoids $G_+\!$ and~$H_+\!$ over~$Q$ with
$G \hspace{-.75pt}=\hspace{-.75pt} Z \times Q$
and~$H \hspace{-.75pt}=\hspace{-.75pt}\nolinebreak Y \!\times\nolinebreak Q$.  If
$\{M_\nn\}_{\nn \in Z}$ is a constructible $G$-family of\/ $Q$-modules and
$\{I_\mmm\}_{\mmm \in Y}$ is a constructible $H$-family of ideals in~$\kk[Q_+]$
with $I_\0 = \kk[Q_+]$, then over
$K_+ = G_+ \times_Q H_+$
\begin{enumerate}
\item\label{i:products'}%
the products $\{I_\mmm M_\nn\}_{\nn\mmm \in Z \times Y}$ form a
constructible family of $Q$-modules, and
\item\label{i:quotients'}%
the quotients $\{M_\nn/I_\mmm M_\nn\}_{\nn\mmm \in Z \times Y}$ form a
constructible family of $Q$-modules.
\end{enumerate}
\end{prop}
\begin{proof}
Since each ideal~$I_\mmm$ is the upset module $\kk\{T_\mmm\} \subseteq
\kk\{Q\}$ for an upset $T_\mmm \subseteq Q$
(Definition~\ref{d:indicator-module}), the $H$-module $\cI =
\bigoplus_{\mmm \in Y} I_\mmm(-\mmm)$ associated to $\{I_\mmm\}_{\mmm
\in Y}$ is the upset $H$-module $\kk\{T\}$ for a semisimple upset $T
\subseteq H$ whose slices parallel to~$Q$ are $\{T_\mmm\}_{\mmm \in
Y}$.  If $U \subseteq G$ is any upset, then in~$K = Z \times Y \times
Q$, identifying $T$ with $\{\0\} \times T$ and $U$ with $U \times
\{\0\}$, the Minkowski sum $T + U \subseteq K$ is an upset whose slice
parallel to~$Q$ through $\nn\mmm \in Z \times Y$ is $T_\mmm + U_\nn$.
If $U$ is semisimple then the Minkowski sum $T + U$ is semisimple
because any Minkowski sum of semisimple sets is semisimple (the proof
of Proposition~\ref{p:valid-class-X} shows this, for example).

The family $\{I_\mmm\}_{\mmm \in Y}$ of ideals is a subfamily of the
\emph{ambient family} $\{A_\mmm\}_{\mmm \in Y}$ that has $A_\mmm =
\kk[Q_+]$ instead of $I_\mmm$ whenever $I_\mmm \neq 0$,~so
$$
  \{A_\mmm\}_{\mmm \in Y}
  =
  \{\kk[Q_+] \mid \mmm \in Y \text{ and } I_\mmm \neq 0\}.
$$
We claim that the ambient family is constructible whenever
$\{I_\mmm\}_{\mmm \in Y}$ is.  To see why, note first that any linear
projection of a semisimple set is semisimple by
Theorem~\ref{t:presburger}, because the image of an affine semigroup
under any linear map is finitely generated and hence another affine
semigroup.  Now apply this observation to the projections of~$H_+$
and~$T$ to~$Y$.  It yields semisimple subsets $\pi_Y(H_+)$ and
$\pi_Y(T)$ inside of~$Y$.  Let
$$
  H^A_+ = \pi_Y(H_+) + Q_+
  \qquad\text{and}\qquad
  S = \pi_Y(T) + Q_+
$$
be the \emph{ambient Rees monoid} constructed from~$H_+$ and the
\emph{ambient upset} constructed from~$T$.  These are both semisimple,
given that $Q_+$ is semisimple---again, because Min\-kowski sums of
semisimple sets are semisimple---and $\kk\{S\}$ is an upset
$H^A$-module by construction.  In fact, $\kk\{S\}$ is the $H^A$-module
associated to the ambient family.

The original family $\{I_\mmm\}_{\mmm \in Y}$ and the ambient family
$\{A_\mmm\}_{\mmm \in Y}$ are two constructible families of ideals
of~$\kk[Q_+]$, with $I_\mmm \subseteq A_\mmm$ for all~$\mmm \in Y$.
However, their Rees monoids need not coincide.  (The reason: it is
unlikely that $H^A_+ + T \subseteq T$, because in most circumstances
the ideals~$I_\mmm$ decrease as~$\mmm$ gets farther from the origin.)
That said, hypothesis $I_\0 = \kk[Q_+]$ implies that
$$
  \kk[H_+]_\mmm
  =
  \kk[Q_+]\,\kk[H_+]_\mmm
  =
  I_\0\,\kk[H_+]_\mmm
  \subseteq
  I_\mmm
  \subseteq
  \{\mmm\} \times Q_+,
$$
so the slice of~$H^+$ parallel to~$Q$ over $\mmm \in Y$ is contained
in the corresponding slice of~$H^A_+$.  It follows that $H^A_+
\supseteq H_+$.  Therefore the ambient family is an $H$-family, so it
is constructible over the Rees monoid~$H_+$.

Going back to the first paragraph of the proof, think of $U$ as an
upset in an upset presentation (Definition~\ref{d:upset-presentation})
over~$\kk[G_+]$ of the associated $G$-module $\cM = \bigoplus_{\nn \in
Z} M_\nn(-\nn)$:
$$
  \bigoplus_{\rho=1}^r \kk\{U^\rho\}
  \to
  \bigoplus_{\sigma=1}^s \kk\{U^\sigma\}
$$
which exists by the syzygy theorem for constructible modules
\cite[Theorem~6.12.4]{hom-alg-poset-mods} applied to the class~$\XX =$
semisimple via Proposition~\ref{p:valid-class-X} and
Definition~\ref{d:tame}.  The goal is to compare what happens when
$S$~and~$T$ are added to the upsets in this presentation.

Adding the ambient upset~$S$ yields an upset presentation
$$
  \bigoplus_{\rho=1}^r \kk\{S + U^\rho\}
  \to
  \bigoplus_{\sigma=1}^s \kk\{S + U^\sigma\}
$$
of the \emph{ambient} $K$-module $\cM^A$ whose component in degree
$\mmm \in Y$ is~$M_\nn$, independent of~$\mmm$, by computing as in
the first paragraph of the proof.  The target
in this presentation has the $K$-submodule $\bigoplus_{\sigma=1}^s
\kk\{T + U^\sigma\}$, whose image in~$\cM^A$ is~$\cI\cM$ since slicing
this image through $\nn\mmm \in Z \times Y$ yields the image of
$\bigoplus_{\sigma=1}^s \kk\{T_\mmm +\nolinebreak U^\sigma_\nn\}
\!\to\!  M_\nn$, which is~$I_\mmm M_\nn$.  The induced $K$-module
homomorphism $\bigoplus_{\sigma=1}^s \kk\{T + U^\sigma\} \to \cM^A$ is
constructible because any semisimple subdivision of~$K$ subordinate to
the given presentation of~$\cM^A$ becomes subordinate to the induced
homomorphism upon common refinement with
the subdivisions of~$K$ induced by the
semisimple upsets $T + U^\sigma$
for~all~$\sigma$.

Constructibility of the products
in item~\ref{i:products'} follows because kernels and cokernels---and
hence images---of constructible morphisms are constructible
\cite[Proposition~4.30.1]{hom-alg-poset-mods}.  Constructibility of
the products
in item~\ref{i:quotients'} now uses only the cokernel part.
\end{proof}

\begin{example}\label{e:direct-sums}
Since constructibility behaves well under direct sums
\cite[Proposition~4.30]{hom-alg-poset-mods}, additional constructible
families ensue.  For example, if $M$ is a constructible module, then
the family $\{M_n\}_{n\in\NN}$ of $Q$-modules for
$$
  M_n = \mathcal{J}(I^n :_R J^n)/\ol{I^n}
        \oplus (I^{(n)}/I^n)
        \oplus (I^n :_R J^\infty)
	\oplus I^n M.
$$
is constructible over the Rees monoid of~$I$.
\end{example}

\begin{remark}\label{r:class-X-not-presburger}
The proof of Proposition~\ref{p:multiply-ideal-module} uses that
Minkowski sums of semisimple sets are semisimple.  The statement and
proof generalize to any class~$\XX$ that is closed under Minkowski
sums, such as when $\XX = $ semialgebraic or piecewise-linear.
\end{remark}


\section{Functors preserving constructibility of families}\label{s:functors-on-families}

Applying the results in \S\ref{s:functors}---particularly
Theorems~\ref{t:cech}, \ref{t:tor}, and~\ref{t:ext}---in the context of
constructible families (Definition~\ref{d:family}) leads to the
conclusion that various functors applied to constructible families
yield constructible families.  This section makes this family-wise
application precise, with subsections numbered as in
\S\ref{s:functors}.

Throughout this section, assume $Q$ is a Presburger group
(Definition~\ref{d:presburger-group}) and $Z$ is a free abelian group
(Convention~\ref{conv:conventions}), as is~$Y$.

\subsection{Local cohomology}\label{b:local-cohomology-families}

\begin{thm}\label{t:loc-coh-family}
Let $\{M_\nn\}_{\nn\in Z}$ be a constructible family of~$Q$-modules over
a Rees monoid~$G_+$ with $G \hspace{-1pt}=\hspace{-1pt} Z \times Q$ and fix
a monomial ideal~$I$ in the monoid algebra~$R
\hspace{-1pt}=\hspace{-1pt}\nolinebreak \kk[Q_+]$.  The family
$\bigl\{H^i_I(M_\nn)\bigr\}_{\nn\in Z}$ of local cohomology modules is
constructible~over~$G_+$~for~all\/~$i$.
\end{thm}
\begin{proof}
Let $\cM = \bigoplus_{\nn \in Z} M_\nn(-\nn)$.  The ideal $I$ is generated by
monomials $\xx^{\qq_1}, \ldots, \xx^{\qq_r}$ with $\qq_i \in Q_+$.  View these
generators as lying in $\kk\bigl[G_+ \cap \bigl(\{\0\} \times Q\bigr)\bigr] =
\kk\bigl[\{\0\} \times Q_+\bigr]$ by Definition~\ref{d:rees-monoid}, and use
them to construct the \v Cech complex~$\vC^\spot_I(\cM)$ to deduce that
$H^i_I(\cM) = \bigoplus_{\nn \in Z} H^i_I(M_\nn)(-\nn)$ for all~$i$.  Since
$\cM$ is a constructible $G$-module by hypothesis,
Theorem~\ref{t:cech} implies that $\vC^\spot_I(\cM)$ is a constructible
complex of~$G$-modules and~$H^\spot_I(\cM)$ is a constructible $G$-module.
\end{proof}

\subsection{Tensor products and Tor}\label{b:tensor-products-families}\mbox{}

\medskip
\noindent
The next definition is made for convenience, to avoid needlessly
repeating the extensive hypotheses and notation of
Lemma~\ref{l:tensor-prod}.

\begin{defn}\label{d:pair-of-families}
The families $\{M_\nn\}_{\nn \in Z}$ and~$\{N_\mmm\}_{\mmm \in Y}$ of
$Q$-modules over~$G = Z \times Q$ and~$H = Y \times Q$ in
Lemma~\ref{l:tensor-prod} form a \emph{pair indexed by the double Rees
monoid}~$K_+$, with associated $G$-module $\cM = \bigoplus_{\nn \in Z}
M_\nn(-\nn)$ and $H$-module~$\cN =\nolinebreak \bigoplus_{\mmm \in Y}
N_\mmm(-\mmm)$.
\end{defn}

Recall the notion of semisimply flat module from
Definition~\ref{d:X-flat}.  The proof of Theorem~\ref{t:tor-family}
uses existence of semisimply flat resolutions---where the module in each
homological degree is semisimply flat---over affine semigroup rings.
The claim is of sufficient utility to separate it off into its own
lemma.

\begin{lemma}\label{l:noetherian-or-artinian}
If $N$ is a localization of a noetherian or artinian module over an
affine semigroup ring~$\kk[Q_+]$, then $N$ admits a semisimply flat
resolution and a semisimply injective resolution as a $Q$-module.
\end{lemma}
\begin{proof}
The claim is immediate from Remarks~\ref{r:noetherian-assump}
and~\ref{r:noetherian-assump'}.
%
%
%
\end{proof}

\begin{lemma}\label{l:tensor-prod}
Fix Rees monoids $G_+$ and~$H_+$ over~$Q$ with $G = Z \times Q$ and~$H
=\nolinebreak Y \times\nolinebreak Q$.  Let $K_+ = G_+ \times_Q H_+$
be their double Rees monoid over~$Q$ (Lemma~\ref{l:double-rees-monoid}).
For any $G$-module~$\cM$ and $H$-module $\cN$, if $\cM^K = \cM \otimes_G \kk[K_+]$
there are natural isomorphisms
$$
  \cM^K = \cM \otimes_Q \kk[H_+]
  \quad\text{and}\quad
  \cM \otimes_Q \cN = \cM^K \otimes_H \cN.
$$
In particular, for a $G$-family $\{M_\nn\}_{\nn \in Z}$ and an
$H$-family $\{N_\mmm\}_{\mmm \in Y}$ of $Q$-modules, if
$$
  \cM = \bigoplus_{\nn \in Z} M_\nn(-\nn)
  \quad\text{and}\quad
  \cN = \bigoplus_{\mmm \in Y} N_\mmm(-\mmm)
\vspace{1ex}%
$$
then\vspace{-3ex}
$$
  \cM \otimes_Q \cN
  =
  \bigoplus_{\nn\mmm} \bigl(M_\nn \otimes_Q N_\mmm\bigr) (-\nn,-\mmm)
  =
  \cM^K \otimes_H \cN.
$$
\end{lemma}
\begin{proof}
The isomorphisms are computed starting with
\begin{align*}
\cM^K
& = \cM \otimes_G \kk[K_+]
\\*
& = \cM \otimes_G \kk[G_+] \otimes_Q \kk[H_+]
\\*
& = \cM \otimes_Q \kk[H_+],
\end{align*}
which implies that\vspace{-3.7ex}
\begin{align*}
\cM^K \otimes_H \cN
& = \cM \otimes_Q \kk[H_+] \otimes_H \cN
\\
& = \cM \otimes_Q \cN.
\end{align*}
The final sentence follows because $\otimes_Q$ distributes over the direct
sums in $\cM \otimes_Q \cN$.
\end{proof}

\pagebreak[2]

\begin{thm}\label{t:tor-family}
Fix a pair of families $\{M_\nn\}_{\nn \in Z}$ and $\{N_\mmm\}_{\mmm \in
Y}$ of $Q$-modules over~$G = Z \times Q$ and~$H = Y \times Q$ indexed by the
double Rees monoid~$K_+$.  Assume that
\begin{itemize}
\item%
$\cM$ is a constructible $G$-module;
\item%
$\kk[H_+]$ is flat as a module over~$\kk[Q_+]$,
and
\item%
$\cN\!$ admits an $H$-flat resolution that is semisimply flat in
degrees $i-1$, $i$, and~$i+1$.
\end{itemize}
Then $\{\Tor_i^Q(M_\nn, N_\mmm)\}_{\nn\mmm \in Z \times Y}$ is a
constructible family over~$K_+$.
In particular,
$$
  \{\Tor_i^Q(M_\nn, N)\}_{\nn\in Z}
$$
is constructible if the input family is constructible and the fixed
input module is a localization of a noetherian or artinian module over
an affine semigroup ring~$\kk[Q_+]$.
\end{thm}
\begin{proof}
By Lemma~\ref{l:tensor-prod}, the functor $\cM \mapsto \cM \otimes_Q \cN$ is a
direct sum over~$\nn$ and~$\mmm$ of functors $M_\nn \mapsto M_\nn \otimes_Q
N_\mmm$.  Each left derived functor of $\cM \mapsto \cM \otimes_Q \cN$ is
therefore a direct sum of the desired $\Tor$ modules.  But
Lemma~\ref{l:tensor-prod} also says that $\cM \mapsto \cM \otimes_Q \cN$ is
isomorphic to the composite
$$
  \cM \mapsto \cM \otimes_Q k[H_+] \mapsto \cM \otimes_Q k[H_+] \otimes_H \cN
$$
which is by definition
$$
  \cM \mapsto \cM^K \mapsto \cM^K \otimes_H \cN.
$$
The flat hypothesis on $H_+$ implies that $\cM \mapsto \cM^K$ is exact, so
the derived functor of the composite can be computed as the homology
of
$$
  \cM^K \otimes_H C_\spot
$$
where $C_\spot$ is the hypothesized $H$-flat resolution of~$\cN$.

The final goal, for the general claim about $\{\Tor_i^Q(M_\nn,
N_\mmm)\}_{\nn\mmm \in Z \times Y}$, is to show that $\cM^K \otimes_H
C_\spot$ is a constructible complex of $K$-modules.  First treat the
case where $G_+$ is a free Rees monoid over~$Q$.  For any
$H$-module~$\cL$, write $\cL^K$ for the extension to~$K$, so $\cL^K =
\cL \otimes_H \kk[K_+]$.  Since $G_+$ is a free Rees mononid,
$\kk[K_+]$ is free as a $\kk[H_+]$-module.  Hence $C_\spot^K$ is flat
and remains semisimply flat in homological degrees $i-1$, $i$,
and~$i+1$ because extension to~$K$ of a localization is a localization
of~$\kk[K_+]$.  The complex $\cM^K \otimes_H C_\spot$ is therefore
constructible thanks to Theorem~\ref{t:tor} with $\XX =
\text{semisimple}$.

If the Rees monoid $G_+$ is not free over~$Q$, then use
Proposition~\ref{p:free-reduction} to replace~$G_+$ with a free
reduction~$G_+'$.  Replacing $G_+$ with its reduction~$G_+'$ has no
effect on $\cM \otimes_Q \cN$ in Lemma~\ref{l:tensor-prod}, but it replaces
the double Rees monoid~$K_+$ with $K_+' = G_+' \times_Q H_+$ throughout
Lemma~\ref{l:tensor-prod}, so $\cM \otimes_Q \cN = \cM^{K'} \otimes_H \cN$.  Thus
the desired result holds by the free case already proven.

For the ``in particular'' claim, note that $\kk[H_+] = \kk[Q_+]$ here, so
Lemma~\ref{l:noetherian-or-artinian} suffices to show that the
hypotheses are satisfied.
\end{proof}

\begin{remark}\label{r:in-practice}
In practice, the flatness hypotheses on~$H_+$ and~$\cN$ are usually
satisfied by taking $H_+ = Q_+$ and assuming that $Q_+$ is noetherian,
so that $\cN$ can be any finitely generated $Q$-module.
\end{remark}

\subsection{Hom and Ext}\label{b:hom-and-ext-families}

\begin{lemma}\label{l:ext}
Let~$G_+$ and~$H_+$ be Rees monoids over~$Q$ with $G = Z \times Q$
and~$H = Y \times Q$.  If $\cM$ and~$\cN$ are the $G$- and $H$-modules
associated to a pair of families $\{M_\nn\}_{\nn \in Z}$ and
$\{N_\mmm\}_{\mmm \in Y}$ of $Q$-modules, and $\cN \cong
\cN^{\vee\vee}\!$ (equivalently, $\dim_\kk \cN_\hhh <\nolinebreak
\infty$ for all $\hhh \in\nolinebreak H\!$)~then
$$
  \bigoplus_{\nn\mmm} \eext^i_Q(M_\nn, N_\mmm)
  =
  \bigoplus_{\nn\mmm} \Tor_i^Q(M_\nn, N_\mmm^\vee)^\vee
$$
\end{lemma}
\begin{proof}
Swap the roles of~$N$ and~$N^\vee$ in Corollary~\ref{c:vee}, as allowed
by Lemma~\ref{l:vee-preserves-X}.
\end{proof}

\begin{thm}\label{t:ext-family}
Fix a pair of families $\{M_\nn\}_{\nn\in Z}$ and $\{N_\mmm\}_{\mmm \in Y}$ of
$Q$-modules over~$G = Z \times Q$ and~$H = Y \times Q$ indexed by the double
Rees monoid~$K_+$, with associated $G$-mod\-ule~$\cM$ and
$H$-module~$\cN$.  Assume that one of the following two scenarios is
in~effect.
\begin{enumerate}
\item\label{i:inj-family}%
$\{M_\nn\}_{\nn\in Z}$ is constructible, $\kk[H_+]$ is flat
over~$\kk[Q_+]$, and some injective resolution of $\cN =
\cN^{\vee\vee}$ is semisimply injective in cohomological degrees $i -
1$, $i$, and\/~$i + 1$.
\item\label{i:flat-family}%
$\{N_\mmm\}_{\mmm \in Y}$ is a constructible family, $\kk[G_+]$ is
flat over~$\kk[Q_+]$, and some flat resolution of~$\cM$ is semisimply
flat in homological degrees $i - 1$, $i$, and\/~$i + 1$.
\end{enumerate}
Then $\{\eext^i_Q(M_{-\nn}, N_\mmm)\}_{\nn\mmm\in Z\times Y}$ is a constructible family
over~$K_+$.  In particular,
$$
  \{\eext^i_Q(M_\nn, N)\}_{\nn\in Z}
  \quad\text{and}\quad
  \{\eext^i_Q(M, N_\mmm)\}_{\mmm\in Y}
$$
are constructible if the input family is constructible and the fixed
input module is a localization of a noetherian or artinian module over
an affine semigroup ring~$\kk[Q_+]$.
\end{thm}
\begin{proof}
The direct sum $\bigoplus_{\nn\mmm} \Tor_i^Q(M_\nn, N_\mmm)(-\nn,-\mmm)$ is a constructible
$K$-module by Theorem~\ref{t:tor-family} combined with symmetry
of~$\Tor$, as long as the itemized hypotheses of
Theorem~\ref{t:tor-family} are sastified verbatim or they are
satisfied with $\cM$ and~$\cN$ as well as~$G$ and~$H$ swapped.  The latter
option is precisely the second scenario here.  The Matlis dual of this
$\Tor$ direct sum is
$$
  \bigoplus_{\nn\mmm} \Tor_i^Q(M_{-\nn}, N_{-\mmm})^\vee(-\nn,-\mmm),
$$
which remains constructible by Lemma~\ref{l:vee-preserves-X} as long
as the hypotheses remain in force unchanged.  Replacing $\cN$ with~$\cN^\vee$
(the Matlis dual here is over~$H$, although it hardly matters, given
the degree-by-degree phrasing of Definition~\ref{d:matlis}), yields
$$
  \bigoplus_{\nn\mmm} \Tor_i^Q(M_{-\nn}, N_\mmm^\vee)^\vee(-\nn,-\mmm),
$$
and this altered version is again constructible by
Theorem~\ref{t:tor-family}, either under the second scenario
unchanged, or under the hypotheses of Theorem~\ref{t:tor-family} as
written but with $\cN^\vee$ in place of~$\cN$.  Given that $\cN = \cN^{\vee\vee}$, the
semisimply flat hypothesis on the resolution of~$\cN^\vee$ from
Theorem~\ref{t:tor-family} is equivalent to the semisimply injective
hypothesis on the resolution of~$\cN$ in the first scenario by
Lemma~\ref{l:vee-preserves-X}.  Applying Lemma~\ref{l:ext} to this
last $\Tor$ direct sum yields the desired result.

The hypotheses for the ``in particular'' claim are set up so that the
fixed input module has a resolution that is semisimply flat or
injective by Lemma~\ref{l:noetherian-or-artinian} or its Matlis dual
(see Lemma~\ref{l:vee-preserves-X} and Remark~\ref{r:matlisexact}).
For $\eext^i_Q(M_\nn, N)$ the injective resolution of~$N$ is used for
scenario~\ref{i:inj-family}; for $\eext^i_Q(M, N_\mmm)$ the flat
resolution of~$M$ is used for scenario~\ref{i:flat-family}.
\end{proof}

\subsection{Functorial quasipolynomiality}\label{b:functorial-quasipol}\mbox{}

\medskip
\noindent
Combining the piecewise quasipolynomial behavior of numerically
constructible families (\S\ref{b:num-const}) with the functorial
preservation of constructibility
(\S\ref{b:local-cohomology-families}--\ref{b:hom-and-ext-families})
yields prototypical piecewise quasipolynomiality results as
corollaries.

Recall the length, max, and min functions $\length_V$, $\max_{\lambda,V}$
and $\min_{\lambda,V}$ defined for a family $\{V_\nn\}_{\nn \in Z}$ of $Q$-graded
vector spaces in Theorem~\ref{t:num-constructible=>quasipol}.  To
simplify the subsequent corollaries, as well as results in
\S\ref{s:applications}, use the following two notational conventions.

\begin{conv}\label{conv:domains}
Domains of functions $\nn \to \ZZ$ shown to be quasipolynomials are not
explicitly indicated.  For example, if the function is $\length(M_\nn)$,
the conclusion implicitly re\-fers to those $\nn$ for which the length
is finite; and if the function is $\max\bigl\{|\aa| \mathbin{\big|}
(M_\nn)_\aa \neq 0\bigr\}$, the conclusion implicitly refers to those $\nn$ for
which the maximum~exists.
\end{conv}

\begin{defn}\label{d:integers-qpoly}
The integers $\{m_\nn\}_{\nn \in T}$ indexed by a subset $T \subseteq \ZZ^k$ are
\emph{piecewise quasipolynomial} or \emph{piecewise quasilinear} if
the function $g: T \to \ZZ$ with $g(\nn) = m_\nn$ is piecewise quasipolynomial
or piecewise quasilinear
(Definition~\ref{d:piecewise-quasipolynomial}).  Recall that if $k =
1$ and $T \subseteq \NN$, this means $g(n)$ coincides with a quasipolynomial for
$n \gg\nolinebreak 0$ (Remark~\ref{r:quasi-pol-k=1}).
\end{defn}

\begin{cor}\label{c:loc-coh-family-quasipol}
Fix a monomial ideal~$I$ in the monoid algebra~$R = \kk[Q_+]$ and a
constructible family $\{M_\nn\}_{\nn \in Z}$ of~$Q$-modules over a Rees
monoid~$G_+$ with $G = Z \times Q$.  If $\{V_\nn\}_{\nn \in Z}$ is the family of
local cohomology modules $V_\nn = H^i_I(M_\nn)$, then the length
$\length_V: Z \to \NN$ is piecewise quasipolynomial of degree at most
$\rank Q$.  Moreover, the functions $\max_{\lambda,V}$ and $\min_{\lambda,V}$ are
piecewise quasilinear for any linear function $\lambda: Q \to \ZZ$.
\end{cor}
\begin{proof}
Apply
Theorem~\ref{t:num-constructible=>quasipol} to the
conclusion of Theorem~\ref{t:loc-coh-family} via
Lemma~\ref{l:unassuming}.
\end{proof}

\begin{cor}\label{c:tor-family->quasipol}
Fix Rees monoids $G_+$ and~$H_+$ over~$Q$ with $G = Z \times Q$ and~$H = Y
\!\times Q$.  Let $\{M_\nn\}_{\nn \in Z}$ and $\{N_\mmm\}_{\mmm \in Y}$ be a $G$-family
and an $H$-family of $Q$-modules.  Assume $\{M_\nn\}$ is constructible,
$H_+$ is flat over~$Q$, and $\cN = \bigoplus_{\mmm \in Y} N_\mmm(-\mmm)$ admits an $H$-flat
resolution that is semisimply flat in
degrees $i-1$, $i$, and~$i+1$.  Let $V_{\nn\mmm} = \Tor_i^Q(M_\nn, N_\mmm)$.
\begin{enumerate}
\item%
The length $\length_V: Z \times Y \to \NN$ is piecewise quasipolynomial of
degree at most~$\rank Q$.
\item%
If $\lambda \hspace{-1pt}: Q \to \ZZ$ is linear, then $\max_{\lambda,V}$ and
$\min_{\lambda,V}$ are piecewise quasilinear~on~$Z \times\nolinebreak Y$.
\end{enumerate}
In particular, if $N$ is a localization of a noetherian or artinian
module over an affine semigroup ring~$\kk[Q_+]$ and $\{M_\nn\}_{\nn\in\ZZ^k}$
is a constructible family of~$Q$-modules, then
$$
  \length\bigl(\Tor_i^Q(M_\nn, N)\bigr)
  \quad\text{and}\quad
  {\textstyle\max}_\lambda\bigl(\Tor_i^Q(M_\nn, N)\bigr)
  \quad\text{and}\quad
  {\textstyle\min}_\lambda\bigl(\Tor_i^Q(M_\nn, N)\bigr)
$$
are piecewise quasipolynomial on~$\ZZ^k$ of degrees at most $\rank Q$,
$1$, and~$1$, respectively.
\end{cor}
\begin{proof}
Apply
Theorem~\ref{t:num-constructible=>quasipol} to the
conclusion of Theorem~\ref{t:tor-family} via Lemma~\ref{l:unassuming}.
\end{proof}

\pagebreak[2]

\begin{cor}\label{c:ext-family->quasipol}
Fix Rees monoids~$G_+\!$ and~$H_+\!$ over~$Q$ with $G = Z \times Q$
and~$H = Y\! \times Q$.  Let $\{M_\nn\}_{\nn\in\ZZ^k}$ be a $G$-family
and $\{N_\mmm\}_{\mmm\in\ZZ^\ell}$ an $H$-family of $Q$-modules.
Assume that one of the following two scenarios is in effect.
\begin{enumerate}
\item\label{i:inj-family'}%
$\{M_\nn\}_{\nn \in Z}$ is constructible, $H_+$ is flat over~$Q$, and
$\cN = \bigoplus_{\mmm\in Y} N_\mmm(-\mmm)$ has a $\kk[H_+]$-injective
resolution that is semisimply injective in
degrees $i - 1$, $i$, and\/~$i + 1$.
\item\label{i:flat-family'}%
$\{N_\mmm\}_{\mmm\in Y}$ is constructible, $G_+$ is flat over~$Q$, and
$\cM = \bigoplus_{\nn\in\ZZ^k} M_\nn(-\nn)$ has a $\kk[G_+]$-flat
resolution that is semisimply flat in
degrees $i - 1$, $i$, and\/ $i + 1$.
\end{enumerate}
Let $V_{\nn\mmm} = \eext^i_Q(M_\nn,N_\mmm)$ for $\nn\mmm \in Z \times Y$.
\begin{enumerate}
\item%
The length $\length_V: Z \times Y \to \NN$ is piecewise quasipolynomial of
degree at most~$\rank Q$.
\item%
If $\lambda \hspace{-1pt}: Q \to \ZZ$ is linear, then
$\max_{\lambda,V}$ and $\min_{\lambda,V}$ are piecewise
quasilinear~on~$Z \times\nolinebreak Y$.
\end{enumerate}
In particular, if $N$ is a localization of a noetherian or artinian
module over an affine semigroup ring~$\kk[Q_+]$ and $\{M_\nn\}_{\nn\in\ZZ^k}$ is
a constructible family of~$Q$-modules, then
$$
  \length\bigl(\eext^i_Q(M_\nn, N)\bigr)
  \quad\text{and}\quad
  {\textstyle\max}_\lambda\bigl(\eext^i_Q(M_\nn, N)\bigr)
  \quad\text{and}\quad
  {\textstyle\min}_\lambda\bigl(\eext^i_Q(M_\nn, N)\bigr)\phantom{.}
$$
are piecewise quasipolynomial on~$\ZZ^k$ of degrees $\ls \rank Q$, $1$,
and~$1$, respectively,~as~are
$$
  \length\bigl(\eext^i_Q(N, M_\nn)\bigr)
  \quad\text{and}\quad
  {\textstyle\max}_\lambda\bigl(\eext^i_Q(N, M_\nn)\bigr)
  \quad\text{and}\quad
  {\textstyle\min}_\lambda\bigl(\eext^i_Q(N, M_\nn)\bigr).
$$
\end{cor}
\begin{proof}
Apply
Theorem~\ref{t:num-constructible=>quasipol} to the
conclusion of Theorem~\ref{t:ext-family} via Lemma~\ref{l:unassuming}.
\end{proof}

\begin{remark}\label{r:truncate-functors}
The piecewise polynomial conclusions of all three corollaries in this
section remain valid for any linear truncation
(Remark~\ref{r:truncate}) of the families of local cohomology, $\Tor$,
or $\eext$ modules.  This can be seen by inserting
Remark~\ref{r:truncate} between the applications of
Lemma~\ref{l:unassuming} and
Theorem~\ref{t:num-constructible=>quasipol} in the proofs of the
corollaries.
\end{remark}


\section{Ubiquity of quasipolynomials}\label{s:applications}

This section covers applications of our theory, showcasing how it can
be used to prove that fundamental homological invariants behave
quasipolynomially.  Some of the conclusions here greatly extend, in
the context of families of multigraded modules over affine semigroup
rings, several results in the literature.

\begin{conv}\label{conv:applications}
Throughout this section, in addition to Convention~\ref{conv:domains}
and Definition~\ref{d:integers-qpoly}, assume the following.
\begin{enumerate}
\item\label{i:affine}%
$Q_+$ is an affine semigroup (Definition~\ref{d:affine-semigroup})
with $Q=\ZZ Q_+\cong \ZZ^d$.  Set $R = \kk[Q_+]$ and $\mm = \kk\{Q_+
\minus \{\0\}\}$ its homogeneous maximal ideal.
\item%
A family $\{M_\nn\}_{\nn\in\ZZ^k}$ of $Q$-modules is
\emph{constructible} if it is so over some Rees monoid~$G_+$ over~$Q$
with $G = \ZZ^k \times Q$ as in Definition~\ref{d:family}.
\end{enumerate}
\end{conv}

\subsection{Local cohomology, \texorpdfstring{$\aa$}{a}-invariants, and regularity}
\label{b:qpoly-local-coh}

\begin{lemma}\label{l:maxmin-nonzero}
Let $\{M_\nn^1\}_{\nn\in\ZZ^k}, \ldots, \{M_\nn^r\}_{\nn\in\ZZ^k}$ be
constructible families of $Q$-modules.  The functions $\nn \mapsto
\max\{i \mid M_\nn^i \neq 0\}$ and $ \nn \mapsto \min\{i \mid M_\nn^i
\neq 0\}$ are piecewise quasiconstant (piecewise quasipolynomial of
degree~$0$).
\end{lemma}
\begin{proof}
By Theorem~\ref{t:presburger} it suffices to show that for each $j$
the sets
\begin{align*}
  A_j &= \bigl\{\nn \in \ZZ^k \mid j = \max\{i \mid M_\nn^i \neq 0\}\bigr\}
\\\text{and}\quad
  B_j &= \bigl\{\nn \in \ZZ^k \mid j = \min\{i \mid M_\nn^i \neq 0\}\bigr\}
\end{align*}
are semisimple.  For each $i$ some Presburger formula $F_i(\nn)$
defines the set $\{\nn \in \ZZ^k \mid M_\nn^i \neq 0\}$.  Thus $A_j$
is defined by $F_j(\nn) \wedge \bigwedge_{i > j} \neg F_i(\nn)$ and
$B_j$ by $F_j(\nn) \wedge \bigwedge_{i < j} \neg F_i(\nn)$.
\end{proof}

\begin{remark}\label{r:assymp-periodic}
If $k = 1$, piecewise quasiconstant is equivalent to eventually
periodic.
\end{remark}

\begin{thm}\label{t:loc-coh-final}
Let $\{M_\nn\}_{\nn\in\ZZ^k}$ be a constructible family of $Q$-modules.
\begin{enumerate}
\item\label{i:lengths-local-coh}%
For any monomial ideal $I$ and  $a,b,c,d\in \ZZ$ the lengths 
$$
  \length \biggl(\dis\bigoplus_{\ a|\nn|+b \,\ls\, |\aa| \,\ls\, c|\nn|+d} \HH iI{M_\nn}_\aa\biggr)
  \quad\text{and}\quad
  \length\bigl(\HH i\mm{M_\nn}_{\gs a|\nn|+b}\bigr)
$$
are piecewise quasipolynomial of degree at most $d$.
\item\label{i:a-inv}%
The $a$-invariants $a_i\left(M_\nn\right)=\max\{|\aa|\mid \HH
i\mm{M_\nn}_\aa \neq 0\}$ are piecewise quasilinear.

\item\label{i:regs}%
The Castelnuovo--Mumford regularities $\reg(M_\nn) = \max\{a_i(M_\nn)
+ i\}$ are piecewise quasilinear.

\item\label{i:depths}%
The depths $\depth(M_\nn) = \min\{i \in \NN \mid \HH i\mm{M_\nn} \neq 0\}$ and
dimensions $\dim(M_\nn) = \max\{i \in \NN \mid \HH i\mm{M_\nn} \neq 0\}$
are piecewise quasiconstant.
\end{enumerate}
\end{thm}
\begin{proof}
Part~\ref{i:lengths-local-coh} follows from
Corollary~\ref{c:loc-coh-family-quasipol} and Remark~\ref{r:truncate};
Part~\ref{i:a-inv} from Corollary~\ref{c:loc-coh-family-quasipol}, and
Part~\ref{i:regs} from Part~\ref{i:a-inv}.  Part~\ref{i:depths}
follows from Lemma~\ref{l:maxmin-nonzero} applied to the families
$\{\HH i\mm{M_\nn}\}_{\nn\in\ZZ^k}$.
\end{proof}
   
\begin{remark}\label{r:dao-montaño}
Theorem~\ref{t:loc-coh-final}.\ref{i:lengths-local-coh} generalizes a
result by Dao and Monta\~no \cite[Theorem 3.8]{dao-montano-2019},
which states that for a monomial ideal $I \subseteq R = \kk[\NN^d]$, the lengths
of $H^i_\mm(R/I^n)$ are eventually quasipolynomial if they are finite
for $n \gg 0$.  The proof there use Takayama's generalizaton
\cite{takayama2005} of Hochster's formula \cite{hochster1977} to
describe the constant regions in the local cohomology of monomial
ideals explicitly by indexing them with labels that are subcomplexes
of a fixed simplicial complex on~$d$ vertices.  That reduces the
characterization of constant regions in local cohomology to ideal
membership in powers of monomial ideals, which they directly prove to
be semisimple.  Constructible module theory here allows conclusions on
the semisimplicity of constant regions without identifying the regions
explicitly, and with arbitrary \mbox{constructible}~\mbox{input}.
\end{remark}

\begin{remark}\label{r:das-lengths}
Using the ideas in \cite[Theorem 3.8]{dao-montano-2019}, Das showed
that for monomial ideals $I_1, \ldots, I_k$ in $R = \kk[\NN^d]$ the lengths
$\length\bigl(\HH 0\mm{R/(I_1^{n_1} \cdots I_k^{n_k})}\bigr)$ are
quasipolynomial for $n_1,\ldots, n_k \gg 0$.
Theorem~\ref{t:loc-coh-final}.\ref{i:lengths-local-coh} generalizes
this result to constructible families.  In particular, by
Proposition~\ref{p:ideals-families}.\ref{i:quotients}
and~\ref{p:ideals-families}.\ref{i:products}, our theorem applies to
$R/I_{1,n_1} \cdots I_{k,n_k}$, where each $\{I_{i,n}\}_{n\in\ZZ}$ is a
constructible family of ideals.
\end{remark}

\begin{remark}\label{r:reg-quasilinear}
Regularities of powers of homogeneous ideals in arbitrary standard
graded noetherian rings are known to eventually agree with a linear
function \cite{cutkosky-herzog-trung99,kodiyalam2000,trung-wang2005}.
Theorem~\ref{t:loc-coh-final}.\ref{i:regs} shows that, over affine
semigroup rings, this linear phenomenon occurs in the much larger
class of constructible families of modules.
\end{remark}

\begin{remark}\label{r:depth-eventually-periodic}
When $R = \kk[\NN^d]$ and $\cM = \bigoplus_{n\in\ZZ} M_n(-n)$ is
noetherian over an underlying Rees monoid, classical methods show that
$\depth(M_n)$ is periodic for $n \gg 0$ \cite{herzog-hibi05}: it holds
for finitely generated graded modules over a noetherian positively
graded algebra whose zeroth component is either a local ring or a
standard graded ring over a field.
In the multigraded setting,
Theorem~\ref{t:loc-coh-final}.\ref{i:depths} applies more generally
when $\cM$ only needs to be constructible (see
Remark~\ref{r:syzygy-thm}).  In \cite{nguyen-trung19} the authors show
that any eventually periodic function can be realized as the depths of
symbolic powers of a homogeneous (not necessarily monomial) ideal in a
polynomial ring.
\end{remark}

\begin{example}\label{e:weird-quasilinear}
To indicate an arbitrary instance of the kinds of composite operations
that still yield controlled numerical growth, let $I, J, K$ be
monomial ideals in the polynomial ring $R = \kk[\NN^d]$.  Regularity and
depth of~$\Ext$ modules of integral closures of saturations of
multiplier ideals of colons of powers exhibit
controlled numerical growth:
\begin{enumerate}
\item\label{i:regs_compos}%
$\reg\left(\Ext_R^i\bigl(\ol{\cJ(I_n : J^{3n-2}) : K^\infty},
R/K\bigr)\!\right)$ is quasilinear, and
\item\label{i:depths_compos}%
$\depth\left(\Ext_R^i\bigl(R/\ol{\mathcal{J}(I^n :_R J^{3n-2}) :_R
K^\infty},R/K\bigr)\!\right)$
is periodic for $n \gg 0$,
\end{enumerate}\setcounter{separated}{\value{enumi}}
as do lengths of local cohomology of quotients mod integral closures
of colons of~powers:
\begin{enumerate}\setcounter{enumi}{\value{separated}}
\item%
$\length\bigl(\HH i{\<x,y\>}{R/\ol{I^n :_R \mm^n}}_{-n \ls |\aa| \ls n}\bigr)$
is quasipolynomial of degree at most~$d$ for $n \gg 0$.
\end{enumerate}
Note that the ideals in items~\ref{i:regs_compos}
and~\ref{i:depths_compos} do not necessarily form graded families.
\end{example}	 

\subsection{Tor and Betti numbers}\label{b:tor-betti}\mbox{}

\medskip
\noindent
Specializing the tensor products, lengths, $\max$, and~$\min$
in Corollary~\ref{c:tor-family->quasipol} yields conclusions about
minimal generators and higher Betti numbers.  Recall
Convention~\ref{conv:applications}.

\pagebreak[2]

\begin{thm}\label{t:tor-final}
Let $\{M_\nn\}_{\nn\in\ZZ^k}$ be a constructible family of $Q$-modules.
\begin{enumerate}

\item\label{i:gens-tor}%
If $N$ is a localization of a noetherian or artinian $R$-module, then
as functions of $\nn \in \ZZ^k$, the length and number of generators
$$
  \length\bigl(\Tor^Q_i(M_\nn,N)\bigr)
  \quad\text{and}\quad
  \mu\bigl(\Tor^Q_i(M_\nn,N)\bigr)
$$
are piecewise quasipolynomials of degree at most~$d$.

\item\label{i:bettis}%
The Betti numbers $\beta_i(M_\nn) = \dim_\kk\Tor_i^Q(M_\nn,\kk)$ are piecewise
quasipolynomial of degree at most~$d$.  This includes the number of
generators $\mu(M_\nn) =\nolinebreak \beta_0(M_\nn)$.
%
\end{enumerate}
\end{thm}
\begin{proof}
The length conclusion in part~\ref{i:gens-tor} is part of the
fixed-$N$ case of Corollary~\ref{c:tor-family->quasipol}.
Part~\ref{i:bettis} follows from the length conclusion in
part~\ref{i:gens-tor} with input module $N = \kk$.  The number of
generators conclusion in part~\ref{i:gens-tor} follows from
part~\ref{i:bettis} and the constructibility of $\Tor^Q_i(M_\nn,N)$ from
Theorem~\ref{t:tor-family} with fixed input module~$N$.
\end{proof}

\begin{remark}\label{r:recursive}
The proof of Theorem~\ref{t:tor-final} feels routine, but note that it
rests heavily on the recursive strength of functorial preservation of
constructibility from \S\ref{s:functors-on-families}.
\end{remark}

\begin{remark}\label{r:num-gens}
Theorem~\ref{t:tor-final}.\ref{i:bettis} implies that for a monomial
ideal $I$, the \emph{symbolic defects} $\mu(I^{(n)}/I^n)$ are
quasipolynomial for $n \gg 0$.  This recovers
\cite[Theorem~2.4]{drabkin-guerrieri2020} in the case of monomial
ideals in affine semigroup rings.
\end{remark}

\begin{remark}\label{r:bivariate-tor}
Resolutions of Rees algebras as modules over polynomial rings have
been used to conclude polynomial growth of $\Tor_i(M/I^n M, N/J^m N)$
in a general commutative local noetherian setting, when these modules
all have finite length \cite{theodorescu2003}.  The relevant
resolutions in our setting would be multigraded by the monoid
underlying the Rees algebra but not finely graded over the polynomial
ring.  It might be possible to extend the constructibility methods
here to that setting, but such an extension lies beyond the scope of
the current initial investigation.
\end{remark}

\begin{defn}\label{d:graded-family-of-ideals}
A sequence of ideals $\{I_n\}_{n\in\NN}$ is a \emph{graded family} if $I_0
= R$ and $I_n I_m \subseteq I_{n+m}$ for every $n, m \in \NN$.  A graded family is
\emph{noetherian} if the graded algebra $\bigoplus_{n \in \NN} I_n(-n)$ is
noetherian.
\end{defn}

\begin{example}\label{e:graded-family-of-ideals}
Let $I$ be a monomial ideal of $R$.  By
\cite[Corollary~9.2.1]{swanson-huneke2006}, which is stated for any
ideal in any analytically unramified Noetherian local ring, the
integral closures $\{\ol{I^n}\}_{n\in \NN}$ form a noetherian graded
family.  By \cite[Theorem~3.2]{herzog-hibi-trung2007} so do the
saturations $\{I^n :_R K^\infty\}_{n\in\NN}$ of a monomial ideal $I$ with
respect to a monomial ideal~$K$.
In particular, the symbolic powers $\{I^{(n)}\}_{n\in\NN}$ of a
monomial ideal $I$ form a noetherian graded family.  When $Q_+$ is
smooth, the direct sum of multiplier ideals $\bigoplus_{n\in\NN}
\mathcal{J}(I^n)(-n)$ is a noetherian module over the Rees monoid of
$I$ \cite[Proposition~18.2.4]{swanson-huneke2006}, which is stated for
ideals in any regular~domain.
\end{example}

\begin{example}\label{e:length-of-tor}
Fix monomial ideals $I$ and~$J$ in $R = \kk[\NN^d]$, and let $M$ be a
noetherian $R$-module.  The number of generators
\noheight{$\mu\bigl(\Tor_i^Q(R/\mathcal{J}(I^n :_R J^n), M)\bigr)$} and the
lengths $\length\bigl(\Tor_i^Q(R/\mathcal{J}(I^n :_R J^n), M)\bigr)$ are
quasipolynomials of degree at most~$d$ for $n \gg 0$.
\end{example}

\subsection{Ext and Bass numbers}\label{b:ext-bass}

\begin{defn}\label{d:bass}
The \emph{Bass number} $\mu_i(\pp,M)$ of a $Q$-module $M$ at the monomial
prime ideal $\pp$ in cohomological degree $i$ is the number of
indecomposable summands isomorphic to a shift of the multigraded
injective hull of $R/\pp$ appearing at stage $i$ in any minimal graded
injective resolution of $M$.
\end{defn}

\begin{remark}\label{r:bass}
Definition~\ref{d:bass} of Bass numbers
agrees with the usual (ungraded) Bass numbers from commutative algebra
\cite[Theorems~1.2.3 and~1.3.4]{goto-watanabe1978ii}.  These numbers
also determine the ungraded Bass numbers at all primes
\cite[Theorem~1.2.3]{goto-watanabe1978ii}. For this reason we state
our results on Bass numbers only for monomial prime ideals.
\end{remark}

\begin{lemma}\label{l:bass}
If $\pp$ is a monomial prime then $\pp = \kk[Q_+ \minus F]$ for some face $F$
of~$Q_+$ (Definition~\ref{d:face}).  The Bass numbers of~$M$ at~$\pp$
can be computed as
$$
  \mu_i(\pp,M)
  =
  \rank_{(R/\pp)_F} \eext_Q^i(R/\pp,M)_F
$$
where $(-)_F$ denotes localization as in
Definition~\ref{d:localization} and $\rank_{(R/\pp)_F}$ is the rank as a
free module over~$(R/\pp)_F$.
\end{lemma}
\begin{proof}
See \cite[Theorem~1.1.4]{goto-watanabe1978ii}.
\end{proof}

\begin{thm}\label{t:ext-final}
Let $\{M_\nn\}_{\nn\in\ZZ^k}$ be a constructible family of $Q$-modules.
\begin{enumerate}

\item\label{i:gens-ext}%
If $N$ is a localization of a noetherian or artinian $R$-module, then
as functions of $\nn \in \ZZ^k$, the lengths
$$
  \length\bigl(\eext^i_Q(M_\nn,N)\bigr)
  \quad\text{and}\quad
  \length\bigl(\eext^i_Q(N,M_\nn)\bigr)
$$
and numbers of generators
$$
  \mu\bigl(\eext^i_Q(M_\nn,N)\bigr)
  \quad\text{and}\quad
  \mu\bigl(\eext^i_Q(N,M_\nn)\bigr)
$$
are piecewise quasipolynomials of degree at most~$d$.

\item\label{i:bass}%
If $\pp$ is a monomial prime ideal, then the Bass numbers $\mu_i(\pp,M_\nn)$
are piecewise quasipolynomial of degree at most $d - \dim F$.
%
\end{enumerate}
\end{thm}
\begin{proof}
The length conclusion in part~\ref{i:gens-tor} is part of the
fixed-$N$ case of Corollary~\ref{c:ext-family->quasipol}.  The number
of generators conclusion in part~\ref{i:gens-tor} follows from
Theorem~\ref{t:tor-final}.\ref{i:bettis} and the constructibility of
the relevant families of $\eext$ modules from
Theorem~\ref{t:ext-family} with fixed input module~$N$.  For
part~\ref{i:bass}, the family $\eext_Q^i(R/\pp,M_\nn)$ is
constructible by Theorem~\ref{t:ext-family} with $N = R/\pp$.  Since
localization preserves constructibility by Lemma~\ref{l:summable}, the
family $\eext_Q^i(R/\pp,M_\nn)_F$ is constructible.  Let $F^\perp
\subseteq Q$ be a complement to~$\ZZ F$, meaning that $F^\perp \cap
\ZZ F = \{\0\}$ and $\rank F^\perp + \rank F = d = \rank Q$.  Then
\begin{itemize}
\item%
the restriction $\{E_\nn\}_{\nn\in\ZZ^k}$ of the family
$\eext_Q^i(R/\pp,M_\nn)_F$ to~$F^\perp$ is constructible over the Presburger
group $F^\perp \cap Q$ because the intersection of any semisimple set with a
subgroup of~$Q$ is semisimple; and
\item%
for each $\nn \in \ZZ^k$, the rank of $\eext_Q^i(R/\pp,M_\nn)_F$ over $(R/\pp)_F$
equals the dimension of~$E_\nn$ as a vector space over~$\kk$ because
$\eext_Q^i(R/\pp,M_\nn)_F$ is free over $(R/\pp)_F$.
\end{itemize}
These vector space dimensions $\dim_\kk E_\nn$ are therefore the desired
Bass numbers by Lemma~\ref{l:bass}, and they are piecewise
quasipolynomial of degree at most $n - \dim F$ by
Theorem~\ref{t:num-constructible=>quasipol}.\ref{i:length-quasipol}
via Lemma~\ref{l:unassuming}, because that is the rank of~$F^\perp$.
\end{proof}

\begin{remark}\label{r:kodiyalam}
Theorems~\ref{t:tor-final}.\ref{i:gens-tor}
and~\ref{t:ext-final}.\ref{i:gens-ext}, together with
Proposition~\ref{p:multiply-ideal-module}, extend to constructible
modules the results in \cite{kodiyalam1993,theodorescu2002} about
polynomial behavior of lengths and number of generators of $\Tor$
and~$\Ext$.
\end{remark}

%
%
%
%
\begin{example}\label{e:bass-numbers}
For any monomial ideal $I$ and monomial prime ideal $\pp$,
the 
Bass numbers $\mu_i\bigl(\pp, \ol{I^n}/I^{5n-2})\bigr)$ are
quasipolynomials of degree at most~$d$ for $n \gg 0$.
\end{example}

\subsection{\texorpdfstring{$v$}{v}-invariants}\label{b:v-invariants}

\begin{defn}\label{d:v-invariant}
The {$v$-invariant} of a monomial ideal~$I$ at a prime $\pp \in \Ass(R/I)$
is
$$
  v_\pp(I)
  =
  \min\{|\uu| \mid \pp = I :_R f \text{ for some } f \in R_\uu\}.
$$
\end{defn}

\begin{remark}\label{r:v-invariant}
In \cite{conca2024}, \cite{ficarra-sgroi2023} it was shown that for
any $\pp \in \Ass^\infty(I)$ the sequence $v_\pp(I^n)$ agrees with a
linear polynomial for $n \gg 0$ (their results work more generally for
homogeneous ideals in noetherian standard graded domains).  The
present goal is to extend this result to constructible families of
$Q$-modules.
\end{remark}


\begin{lemma}\label{l:annihilators-constructible}
Given a constructible family $\{M_\nn\}_{\nn \in Z}$ of $Q$-modules and a
monomial ideal $I \subseteq R$, the family $\{0:_{M_\nn} I\}_{\nn \in Z}$ of
$Q$-modules is constructible.
\end{lemma}
\begin{proof}
Let $G_+$ be an underlying Rees monoid of $\{M_\nn\}_{\nn \in Z}$, so
the associated $G$-module $\cM = \bigoplus_{\nn \in Z} M_\nn(-\nn)$ is
constructible.  By Convention~\ref{conv:applications}.\ref{i:affine},
the ideal $I$~is finitely generated by monomials $\xx^{\qq_1}, \ldots,
\xx^{\qq_r}$.  Multiplication by $\xx^{\qq_i}$ induces an endomorphism
$\phi_i: \cM \to \cM$ that is constructible by
Proposition~\ref{p:multiply-ideal-module}.  The kernel therefore
yields a constructible submodule $\ker(\phi_i) \subseteq \cM$
\cite[Proposition~4.30.1]{hom-alg-poset-mods}.  The common refinement
of any constant subdivisions subordinate to these kernels is
subordinate to their intersection.  Therefore $\bigoplus_{\nn \in Z}
(0:_{M_\nn} I)(-\nn) = 0:_\cM I = \bigcap_{i=1}^r \ker(\phi_i)$ is a
constructible $G$-module.
\end{proof}

\begin{defn}\label{d:lambda-v-invariant}
Fix a $Q$-module~$M$, any monomial prime $\pp \subseteq \kk[Q_+]$,
and a linear function $\lambda: Q \to \ZZ$.  The \emph{$v$-invariant} of~$M$
at~$\pp$ \emph{in direction~$\lambda$} is
$$
  v_\pp(\lambda,M) = \min\{\<\lambda,\uu\> \mid \pp = I :_R f \text{ for some } f \in R_\uu\}.
$$
The case $\lambda = (1,\ldots,1)$ is Definition~\ref{d:v-invariant}.
\end{defn}

\begin{thm}\label{t:v-invariants}
For a constructible family $\{M_\nn\}_{\nn \in Z}$ of $Q$-modules, monomial
prime~$\pp$, and linear $\lambda \hspace{-1pt}:\hspace{-1pt} Q \to \ZZ$, the
$v$-invariants $v_\pp(\lambda,\hspace{-1.1pt} M_\nn)$ in direction~$\lambda$ are
piecewise~\mbox{quasilinear}.
\end{thm}
\begin{proof}
Let $J$ be the product of the (finitely many) monomial prime ideals
properly containing $\pp$, with $J = R$ if no such prime exists.  The
proof of \cite[Lemma~1.2]{conca2024} applies more generally to our
setting to show that for a $Q$-module $M$ the set of homogeneous
elements $\{f \in M \mid \pp = 0:_R f\}$ are those whose images in $(0 :_M \pp)
/ \bigl(0 :_M (\pp + J^\infty)\bigr)$ is nonzero, where $0 :_M (\pp + J^\infty) =
\bigcup_{m \gs 0} \bigl(0 :_M (\pp + J^m)\bigr)$.  The conclusion now follows by
applying
Theorem~\ref{t:num-constructible=>quasipol}.\ref{i:maxmin-quasipol} to
the family $\bigl\{(0:_{M_\pp} \pp) / \bigl(0:_{M_\nn}
(\pp+Q^\infty)\bigr)\bigr\}_{\nn \in Z}$, which is constructible by
Lemma~\ref{l:annihilators-constructible}.
\end{proof}

\begin{remark}
Theorem~\ref{t:v-invariants} recovers for monomial ideals results in
\cite{ficarra-sgroi2025} and \cite{a-sarkar2025} about linear behavior
of $v$-invariants of noetherian graded families of ideals.  Moreover,
our theorem shows that for monomial ideals $I_1, \ldots, I_k$ and a
monomial prime ideal~$\pp$, the $v$-invariants $v_\pp(I_1^{n_1} \cdots
I_k^{n_k})$ are piecewise quasilinear.
\end{remark}

\subsection{Degrees and multiplicities}\label{b:degrees}

\begin{defn}\label{d:deg}
Recall Convention~\ref{conv:applications} regarding $\mm \subseteq R = \kk[Q_+]$.
Fix an $\mm$-primary ideal~$I$, so $(\xx^\qq)^n \in I$ for $n \gg 0$ for every
nonunit monomial~$\xx^\qq$, and a noetherian $R$-module $M$.  The
\emph{Hilbert polynomial (of $M$ with respect to~$I$)} is the
polynomial
$$
  H_{I,M}(m) \in \QQ[m] \text{ of degree } \delta = \dim(M)
$$
such that $\length(M/I^mM) = H_{I,M}(m)$ for all $m \gg 0$.  When this
polynomial is written as
$$
  H_{I,M}(m)
  = e_0(I,M)\binom{m+\delta}{\delta}m^{\delta}
    - e_1(I,M)\binom{m+\delta-1}{\delta-1}m^{\delta-1}
    + \cdots + (-1)^{\delta}e_{\delta}(I,M),
$$
the integers $e_i(I,M)$ are the \emph{Hilbert coefficients}.  The
coefficient $e_0(I,M) \in \NN$ is commonly referred to as the
\emph{Hilbert-Samuel multiplicity}.  When $I = \mm$, the leading
coefficient $e_0(\mm,M)$ is also called the \emph{degree of $M$} and
denoted by $\deg(M)$.
\end{defn}

\begin{thm}\label{t:hilbert-coeffs}
Fix a constructible family $\{M_\nn\}_{\nn\in\ZZ^k}$ of noetherian $Q$-modules
and an $\mm$-primary monomial ideal $I$ in the affine semigroup ring $R
= \kk[Q_+]$.  The Hilbert coefficients $e_i(I,M_\nn)$ for each fixed~$i$
are piecewise quasipolynomial of degree at most~$d$.
\end{thm}
\begin{proof}
The family $\{M_\nn/I^m M_\nn\}_{m\nn \in \ZZ^{k+1}}$ is constructible
by Proposition~\ref{p:multiply-ideal-module}.  Thanks to
Theorem~\ref{t:num-constructible=>quasipol}.\ref{i:length-quasipol}
the lengths $\length(M_\nn/I^m M_\nn)$ are piecewise quasipolynomial
of degree at most~$d$.  Let $\ZZ^{k+1} = \bigcup_{\alpha \in A}
(\Gamma_\alpha \cap \ZZ^{k+1})$ be a finite polyhedral partition
afforded by Definition~\ref{d:piecewise-quasipolynomial}, with the
piecewise quasipolynomial $\length(M_\nn/I^m M_\nn)$ equal to the
quasipolynomial $Q_\alpha(m,\nn)$ on~$\Gamma_\alpha$.  The dimensions
$\dim(M_\nn)$ are quasiconstant by
Theorem~\ref{t:loc-coh-final}.\ref{i:depths}, so further assume that
if $P(m,\nn)$ is one of the polynomials that define $Q_\alpha(m,\nn)$,
then $\dim(M_\nn)$ is constant on the values of $\nn$ for which
$Q_\alpha(m,\nn) = P(m,\nn)$ for $m\nn \in \Gamma_\alpha$.

Let $A' \subseteq A$ be the subset indexing the polyhedra
$\Gamma_\alpha$ in the partition such that $S_\alpha = \{\nn \mid m\nn
\in \Gamma_\alpha \text{ for } m \gg 0\}$ is nonempty.  The sets
$S_\alpha$ for $\alpha \in A'$ polyhedrally partition of~$\RR^k$.  Fix
$\alpha \in A'$ and $\nn_0 \in S_\alpha$.  Then $Q_\alpha(m,\nn_0)$
coincides with the Hilbert polynomial $H_{I,M_{\nn_0}}(m)$ for $m \gg
0$.  Since $Q_\alpha$ is defined by finitely many polynomials,
$Q_\alpha(m,\nn_0) = H_{I,M_{\nn_0}}(m)$ for every $m$ with $m\nn_0
\in \Gamma_\alpha$.  Thus, for every $\nn \in S_\alpha$ the Hilbert
coefficient $e_i(I,M_\nn)$ is the coefficient of $m^{\dim(M_\nn)-i}$
in $Q_\alpha(m,\nn)$, which is a quasipolynomial in~$\nn$ of degree at
most~$d$.
\end{proof}

\begin{remark}\label{r:hilbert-coeffs}
Theorem~\ref{t:hilbert-coeffs} answers in the affirmative (for
constructible families) several questions in
\cite[Introduction]{herzog-puth-verma2008} on the quasipolynomial
behavior of Hilbert coefficients of several sequences of modules, such
as that of $e_i(I^n M / J^n N)$.  Theorem~\ref{t:hilbert-coeffs} also
extends several related results already included in
\cite{herzog-puth-verma2008}.
\end{remark}

\begin{defn}[{\cite[Definition~2.8]{vasconcelos-hom-deg-1998}}]\label{d:hdeg}
Assume $\kk[Q_+]$ is Gorenstein and that $Q_+$ is generated by elements
that lie on a hyperplane (for example, $Q_+ = \NN^d$).  The
\emph{homological degree} of a noetherian $\NN$-graded module~$M$ of
Krull dimension $\delta = \dim(M)$ is defined recursively by
$$
  \hdeg(M) 
  =
  \deg(M) + \sum_{i=d-\delta+1}^d \binom{\delta-1}{i-d+\delta-1} \hdeg\bigl(\Ext^i_R(M,R)\bigr).
$$
\end{defn}

Vasconcelos introduced homological degrees with the goal of extending
good properties of of degrees to modules that are not Cohen--Macaulay,
for instance to provide bounds for numbers of generators.

\begin{cor}\label{c:hdeg}
Assume $Q_+$ satisfies the hypothesis in Definition~\ref{d:hdeg}.  Let
$\{M_\nn\}_{\nn \in Z}$ be a constructible family of noetherian $Q$-modules.
The homological degrees $\hdeg(M_\nn)$ are piecewise quasipolynomial of
degree at most~$d$.
\end{cor}
\begin{proof}
Since $\hdeg(M_\nn)$ is a linear combination of~$\deg(M_\nn)$ and
degrees of iterated $\Ext$ modules of the form
$\Ext_R^{i_1}(\Ext_R^{i_2}(\cdots(\Ext_R^{i_r}\big(M_\nn,R\big),R),\cdots,R)$,
the result is a consequence of Theorem~\ref{t:hilbert-coeffs} and
Theorem~\ref{t:ext-family}.
\end{proof}

\begin{remark}\label{r:hdegs}
As far as we are aware, this polynomial behavior for homological
degrees was not known, even for quotients $\{R/I^n\}_{n\in\NN}$ by powers
of a fixed monomial~ideal.
\end{remark}

\begin{remark}\label{r:adeg}
The same statement and proof work for the \emph{arithmetic degree}
$$
  \adeg(M)
  =
  \sum_{\pp\in\Ass M} {\rm mult}_M(\pp)\,\deg(R/\pp)
$$
\cite[Definition~2.3]{vasconcelos-hom-deg-1998} by its
characterization \cite[Proposition~2.4]{vasconcelos-hom-deg-1998} as a
sum over double $\Ext$ modules.  It would be difficult to verify
whether arbitrary extended degrees
\cite[p.\,347]{vasconcelos-cohom-degrees-1998}, such as \emph{unmixed
degree} \cite{cuong-quy2025}, behave piecewise quasipolynomially.
\end{remark}



\begin{thebibliography}{IWW05}
\raggedbottom


\bibitem[AS25]{a-sarkar2025}
Vanmathi A and Parangama Sarkar, \emph{$v$-Numbers of symbolic power
  filtrations}, Collect. Math., to appear.\ \
  \textsf{arXiv:math.AC/2403.09175}

\bibitem[BC20]{botnan-crawley-boevey2020}
Magnus Botnan and William Crawley-Boevey, \emph{Decomposition of
  persistence modules}, Proc.\ Amer.\ Math.\ Soc.\ \textbf{148}
  (2020), no.\,11, 4581--4596.\ \
  \textsf{arXiv:math.RT/1811.08946}

\bibitem[BG09]{bruns-gubeladze2009}
Winfried Bruns and Joseph Gubeladze,
  \emph{Polytopes, Rings, and $K$-theory},
  Springer Monographs in Mathematics,
  Springer, Dordrecht, 2009.



\bibitem[Bro79]{brodman79}
Markus Brodmann, \emph{Asymptotic stability of
  $\operatorname{Ass(M/I^nM)}$}, Proceedings of the American
  Mathematical Society \textbf{74} (1979), 16--18.

\bibitem[Cav25]{cavey25}
Ian Cavey, \emph{Graded Ehrhart theory and toric geometry},
  preprint, 2025.  \textsf{arXiv:math.CO/ 2508.19176}


\bibitem[Cut14]{cutkosky14} S.~D. Cutkosky, \emph{Asymptotic multiplicities of graded families of ideals and linear series}, Adv. Math. \textbf{264} (2014), 55--113.

\bibitem[CHT99]{cutkosky-herzog-trung99}
Steven Dale Cutkosky, Jürgen Herzog, and Ngô Viêt Trung,
  \emph{Asymptotic behaviour of the Castelnuovo--Mumford regularity},
  Compositio Math. \textbf{118} (1999), 243--261.


\bibitem[Con24]{conca2024}
Aldo Conca, \emph{A note on the $v$-invariant},
  Proc. Amer. Math. Soc. \textbf{152} (2024), 2349--2351.

\bibitem[CQ25]{cuong-quy2025}
Nguyen Tu Cuong and Pham Hung Quy,
  \emph{On the structure of finitely generated modules and the unmixed degrees},
  J. Pure and Applied Algebra \textbf{229} (2025) 30~pages, article~108000.






\bibitem[dFH09]{deFernex-Hacon2009}
Tommaso de~Fernex and Christopher D. Hacon, \emph{Singularities on
  normal varieties}, Compos. Math. \textbf{145} (2009), 393--414. 

\bibitem[DG20]{drabkin-guerrieri2020}
B. Drabkin and L. Guerrieri, \emph{Asymptotic invariants of ideals
  with Noetherian symbolic Rees algebra and applications to cover
  ideals}, J. Pure Appl. Algebra \textbf{224} (2020), 300--319.

\bibitem[DIV12]{d'alessandro-intrigila-varricchio2012}
Flavio D'Alessandro, Benedetto Intrigila, and Stefano Varricchio,
  \emph{Quasipolynomials, linear Diophantine equations, and semi-linear sets},
  Theoret. Comput. Sci.\ \textbf{416} (2012), 1--16.

\bibitem[DM19]{dao-montano-2019}
Hailong Dao and Jonathan Monta\~no, \emph{Length of local cohomology
  of powers of ideals}, Transactions of the American Mathematical
  Society \textbf{371} (2019), 3483--3503.


\bibitem[ES69]{eilenberg-schutzenberger1969}
Samuel Eilenberg, Marcel-Paul Sch\"utzenberger,
  \emph{Rational sets in commutative monoids}, J. Algebra \textbf{13}
  (1969) 173--191.

\bibitem[FS23]{ficarra-sgroi2023}
Antonino Ficarra, Emanuele Sgroi, \emph{Asymptotic behaviour of the
  v-number of homogeneous ideals}, preprint, 2023. 
  \texttt{https://arxiv.org/abs/2306.14243}

\bibitem[FS25]{ficarra-sgroi2025}
Antonino Ficarra, Emanuele Sgroi, \emph{Asymptotic behavior of integer
  programming and the $v$-function of a graded filtration}, J. Algebra
  Appl., to appear.

\bibitem[GM23]{gldim}
Nathan Geist and Ezra Miller,
\emph{Global dimension of real-exponent polynomial rings},
  Algebra \& Number Theory \textbf{17} (2023), no.\,10, 1779--1788.\ \
  doi:10.2140/ant.2023.17.1779\ \
  \textsf{arXiv:math.AC/2109.04924}

\bibitem[GW78]{goto-watanabe1978ii}
Shiro Goto and Keiichi Watanabe, \emph{On graded rings, II
  ($\ZZ^n$-graded rings)}, Tokyo J. Math. \textbf{1} (1978), no.\,2,
  237--261.


\bibitem[HH05]{herzog-hibi05}
Jürgen Herzog and Takayuki Hibi, \emph{The depth of powers of an
  ideal}, J. Algebra \textbf{291} (2005),  534--550. 

\bibitem[HHT07]{herzog-hibi-trung2007}
Jürgen Herzog, Takayuki Hibi and Ng\^o~Vi\^et~Trung, \emph{Symbolic
  powers of monomial ideals and vertex cover algebras},
  Adv. Math. \textbf{210} (2007),  304--322.

\bibitem[HM05]{helm-miller2005}
David Helm and Ezra Miller, \emph{Algorithms for graded injective resolutions
  and local cohomology over semigroup rings}, Journal of Symbolic
  Computation \textbf{39} (2005), 373--395.\ \
  \textsf{arXiv:math.AC/0309256}

\bibitem[HM18]{hsiao-matusevich2018}
Jen-Chieh Hsiao and Laura Felicia Matusevich, \emph{Bernstein-Sato
  polynomials on normal toric varieties}, Michigan Math. J. \textbf{67}
  (2018), 117--132.
 
\bibitem[Hoc75]{hochster1977}
Melvin Hochster, Cohen-Macaulay rings, combinatorics, and simplicial
  complexes, in \emph{Ring theory, II (Proc. Second Conf.,
  Univ. Oklahoma, Norman, Okla., 1975)}, pp. 171--223, Lect. Notes
  Pure Appl. Math., Vol. 26, Dekker, New York-Basel,  1977


\bibitem[How01]{howald2001}
Jason A. Howald, \emph{Multiplier ideals of monomial ideals},
  Trans. Amer. Math. Soc.  \textbf{353} (2001), 2665--2671.

\bibitem[HPV08]{herzog-puth-verma2008}
Jürgen Herzog, Tony J. Puthenpurakal and Jugal K. Verma, \emph{Hilbert
  polynomials and powers of ideals}, Math. Proc. Cambridge
  Philos. Soc. \textbf{145} (2008),  623--642.

\bibitem[HY03]{hara-yoshida2003}
Nobou Hara and Ken-ichi Yoshida, \emph{A generalization of tight
  closure and multiplier ideals}, Trans. Amer. Math. Soc. \textbf{355}
  (2003), 3143--3174.
  
\bibitem[Kod93]{kodiyalam1993}
Vijay Kodiyalam, \emph{Homological invariants of powers of an ideal},
  Proc. Amer. Math. Soc. \textbf{118} (1993),  757--764.

\bibitem[Kod00]{kodiyalam2000}
Vijay Kodiyalam, \emph{Asymptotic behaviour of Castelnuovo--Mumford
  regularity}, Proc. Amer. Math. Soc. \textbf{128} (2000),  407--411.

\bibitem[KS18]{kashiwara-schapira2018}
Masaki Kashiwara and Pierre Schapira, \emph{Persistent homology and
  microlocal sheaf theory}, J. of Appl. and Comput.\ Topology
  \textbf{2}, no.\,1--2 (2018), 83--113.\ \
  \textsf{arXiv:math.AT/1705.00955v6}

\bibitem[KS19]{kashiwara-schapira2019}
Masaki Kashiwara and Pierre Schapira, \emph{Piecewise linear sheaves},
  International Math. Res. Notices [IMRN] (2021), no.\,15, 11565--11584.
  doi:10.1093/imrn/rnz145\ \
  \textsf{arXiv:math.AG/ 1805.00349v3}

\bibitem[Laz04]{lazarsfeld2004}
Robert Lazarsfeld, \emph{Positivity in algebraic geometry. II},
  Ergebnisse der Mathematik und ihrer Grenzgebiete. 3. Folge. A Series
  of Modern Surveys in Mathematics, 49, Springer, Berlin, 2004.

\bibitem[Laz64]{lazard1964}
Daniel Lazard, \emph{Sur les modules plats},
  C.\,R.\,Acad.\,Sci.\,Paris \textbf{258} (1964), 6313--6316.



\bibitem[Mil00]{alexdual2000}
Ezra Miller, \emph{The Alexander duality functors and local duality
  with monomial support}, Journal of Algebra \textbf{231} (2000),
  180--234.

\bibitem[Mil09]{affine-strat}
Ezra Miller,
  \emph{Affine stratifications from finite mis\`ere quotients},
  Journal of Algebraic Combinatorics \textbf{37} (2013), 1--9.
  doi:10.1007/s10801-012-0355-3\ \
  \textsf{arXiv:math.CO/1009.2199}



\bibitem[Mil20]{essential-real}
Ezra Miller, \emph{Essential graded algebra over polynomial rings with
  real exponents}, submitted, 2020.  \textsf{arXiv:math.AC/2008.03819}

\bibitem[Mil23]{strat-conical}
Ezra Miller, \emph{Stratifications of real vector spaces from
  constructible sheaves with conical microsupport}, Journal of Applied
  and Computational Topology \textbf{7} (2023), no.\,3, 473--489.\ \
  doi:10.1007/s41468-023-00112-1\ \
  \textsf{arXiv:math.AT/2008.00091}


\bibitem[Mil25]{hom-alg-poset-mods}
Ezra Miller, \emph{Homological algebra of modules over posets},
  SIAM Journal on Applied Algebra and Geometry \textbf{9} (2025),
  no.\,3, 483--524.\ \ doi:10.1137/22M1516361\ \ 
  \textsf{arXiv:math.AT\!/\!2008.00063}

\bibitem[MMW05]{matusevich-miller-walther2005}
\!\!Laura Matusevich, Ezra Miller, and Uli Walther, \emph{Homological
  methods for hypergeometric families}, Journal of the American Math
  Society \textbf{18} (2005), no.~4, 919--941.\ \
  doi:10.1090/S0894-0347-05-00488-1\ \
  \textsf{arXiv:math.AG/0406383}


\bibitem[MS05]{cca}
Ezra Miller and Bernd Sturmfels, \emph{Combinatorial commutative
  algebra}, Graduate Texts in Mathematics, vol.~227, Springer-Verlag,
  New York, 2005.


\bibitem[NT19]{nguyen-trung19}
Dang~Hop~Nguyen and Ngo~Viet~Trung, Depth functions of symbolic powers
  of homogeneous ideals, Invent. Math. \textbf{218} (2019), 779--827.


\bibitem[Pre30]{presburger1930}
Mojżesz Presburger, \emph{\"Uber der Vollst\"andigkeit eines
  gewissen Systems der Arithmetik ganzer Zahlen, in welchen die
  Addition als einzige Operation hervortritt}, in F.\,Leja (ed.),
  Comptes Rendus Premier Congr\`es des Math\'maticienes des Pays
  Slaves, Varsovie 1929 / Sprawozdanie z I Kongresu matematyk\'ow
  kraj\'ow s\l owia\'nskich, Warszawa 1929.  Warsaw, Lw\'ow and
  Krakow (1930), pp.\,92--101.

\bibitem[RR25]{reiner-rhoades25}
Victor Reiner and Brendon Rhoades, \emph{Harmonics and graded Ehrhart
  theory}, preprint, 2025.  \textsf{arXiv:math.CO/2407.06511}

\bibitem[Sch86]{schrijver1986}
A. Schrijver,
  \emph{Theory of linear and integer programming},
  Wiley-Interscience Series in Discrete Mathematics A
  Wiley-Interscience Publication, Wiley, Chichester, 1986.

\bibitem[SH06]{swanson-huneke2006}
Irena Swanson and Craig Huneke, \emph{Integral closure of Ideals,
  Rings and Modules}, London Math. Soc. Lect. Note
  Ser. 336. Cambridge University Press, Cambridge, 2006.


\bibitem[Sta82]{stanley1982}
Richard~P. Stanley, \emph{Linear Diophantine equations and local
  cohomology}, Invent. Math. \textbf{68} (1982), 175--193.

\bibitem[Stu95]{sturmfels1995}
Bernd Sturmfels, \emph{On vector partition functions}, Journal of
  Combinatorial Theory, Series A, \textbf{72} (1995), 302–-309.

\bibitem[Tak05]{takayama2005}
Yukihide Takayama,
  \emph{Combinatorial characterizations of generalized Cohen--Macaulay
  mono\-mial ideals}, Bull. Math. Soc. Sci. Math. Roumanie (N.S.)
  \textbf{48(96)} (2005), 327--344.

\bibitem[The02]{theodorescu2002}
Emanoil Theodorescu, \emph{Derived functors and Hilbert polynomials},
  Math. Proc. Cambridge Philos. Soc. \textbf{132} (2002),  75--88.

\bibitem[The03]{theodorescu2003}
Emanoil Theodorescu, \emph{Bivariate Hilbert functions for the torsion
  functor}, J. Algebra \textbf{265} (2003), 136--147.

\bibitem[TW05]{trung-wang2005}
Ngô Viêt Trung and Hsin-Ju Wang, \emph{On the asymptotic linearity of
  Castelnuovo--Mumford regularity}, J. Pure Appl. Algebra \textbf{201}
  (2005), 42--48.

\bibitem[Vas98a]{vasconcelos-hom-deg-1998} 
Wolmer V. Vasconcelos, \emph{The homological degree of a module},
  Trans. Amer. Math. Soc. \textbf{350} (1998), 1167--1179.

\bibitem[Vas98b]{vasconcelos-cohom-degrees-1998} 
Wolmer V. Vasconcelos, \emph{Cohomological degrees of graded modules},
  in \emph{Six lectures on commutative algebra} (Bellaterra, 1996),
  Progress in Math., vol.~166, Birkh\"auser Verlag, Basel, 1998,
  pp.\,345--392.

\bibitem[Woo15]{woods2015}
Kevin Woods, \emph{Presburger arithmetic, rational generating
  functions, and quasi-polynomials}, J. Symb. Log. \textbf{80} (2015),
  no.~2, 433--449.




\end{thebibliography}
\end{document}